\documentclass[12pt,notitlepage]{amsart}
\usepackage{latexsym,amsfonts,amssymb,amsmath,amsthm}
\usepackage{graphicx}
\usepackage{color}
\usepackage{mathrsfs}
\usepackage{amsmath, amsthm, amsfonts, amssymb, cite}
\usepackage{wrapfig}
\usepackage{stackrel}
\usepackage{color}
\usepackage{float}
\usepackage{enumitem}
\usepackage{mathtools}
\usepackage{multicol}
\usepackage[normalem]{ulem}

\usepackage[inner=1.0in,outer=1.0in,bottom=1.0in, top=1.0in]{geometry}

\newcommand{\R}{\ensuremath{\mathbb{R}}}
\newcommand{\mz}{\ensuremath{\mathbb Z}}
\newcommand{\mn}{\ensuremath{\mathbb N}}
\newcommand{\mr}{\ensuremath{\mathbb R}}
\newcommand{\mh}{\ensuremath{\mathbb H}}

\newcommand{\shortmod}{\ensuremath{\negthickspace \negthickspace \negthickspace \pmod}}

\newcommand{\intR}{\int_{-\infty}^{\infty}}

\newcommand{\sumstar}{\sideset{}{^*}\sum}

\newcommand{\real}{\mathop{\rm Re}}

\newcommand{\lcm}{\mathop{\rm lcm}}

\renewcommand{\bf}[1]{\mathbf{#1}}

\newcommand{\eps}{\varepsilon}
\newcommand{\es}[1]{\begin{equation}\begin{split}#1\end{split}\end{equation}}
\newcommand{\est}[1]{\begin{equation*}\begin{split}#1\end{split}\end{equation*}}

\newcommand{\legen}[2]{\left(\frac{#1}{#2}\right)}

\theoremstyle{plain}		
	\newtheorem{mytheo}{Theorem} [section]
	
	\newtheorem{myprop}[mytheo]{Proposition}
	\newtheorem{mycoro}[mytheo]{Corollary}
     \newtheorem{mylemma}[mytheo]{Lemma}
	\newtheorem{mydefi}[mytheo]{Definition}

\theoremstyle{remark}

\numberwithin{equation}{section}
\numberwithin{figure}{section}

\begin{document}
\title{A generalized cubic moment and the Petersson formula for newforms}
\author{Ian Petrow} 
\email{ian.petrow@math.ethz.ch}
\address{Section de Math\'ematiques\\
	Ecole Polytechnique F\'ed\'erale de Lausanne\\
	B\^atiment MA, Station 8\\
	1015 Lausanne\\
	Switzerland}
\curraddr{ ETH Z\"urich\\
Department of Mathematics\\ 
R\"amistrasse 101\\
8092 Z\"urich\\
Switzerland}

\author{Matthew P. Young} 
\email{myoung@math.tamu.edu}
\address{Department of Mathematics\\
	  Texas A\&M University\\
	  College Station\\
	  TX 77843-3368\\
		U.S.A.}

%\classification{11F11, 11F37, 11F66, 11M99}

 \thanks{This material is based upon work supported by the National Science Foundation under agreement No. DMS-1401008 (M.Y.).  Any opinions, findings and conclusions or recommendations expressed in this material are those of the authors and do not necessarily reflect the views of the National Science Foundation. \\
 The first author was partially supported by Swiss national science foundation grants 200021\_137488 and PZ00P2\_168164 and an AMS-Simons travel grant.}

\begin{abstract}
 Using a cubic moment, we prove a Weyl-type subconvexity bound for the quadratic twists of a  holomorphic  newform of square-free level, trivial nebentypus, and arbitrary even weight.  This generalizes work of Conrey and Iwaniec in that the newform that is being twisted may have arbitrary square-free level, and also that the quadratic character may have even conductor.  One of the new tools developed in this paper is a more general Petersson formula for newforms of square-free level.
% \subclass{11F11, 11F37, 11F66, 11M99}
 \end{abstract}

\maketitle

\section{Introduction}

\subsection{Cubic moments}
\label{section:introCubic}
Let $\chi_q$ be a real, primitive character of conductor $q$ and $\tilde{q}= \operatorname{rad}(	q)$ its square-free kernel. 
Let $H^*_\kappa(N)$ be the set of  Hecke-normalized  holomorphic newforms for $\Gamma_0(N)$, of weight $\kappa$, and trivial central character.  
Our main result is
\begin{mytheo}
\label{thm:cubicmoment}
 For any square-free $r$ with $(r,q)=1$ we have
 \begin{equation}
 \label{eq:cubicmoment}
  \sum_{\substack{ f \in H^*_\kappa(rq') \\  q' | \tilde{q}}}  L(1/2,f \otimes \chi_q)^3 \ll_{\kappa, \varepsilon} (qr)^{ 1+ \varepsilon}.
 \end{equation}
The estimate  holds for any even $\kappa \geq 2$, and depends polynomially on $\kappa$. \end{mytheo}
\begin{mycoro}
\label{coro:subconvexity}
For any  holomorphic  newform $f$ of square-free level $s$ and $\chi_q$ any real primitive character of conductor $q$ we have  \begin{equation}
 \label{eq:subconvexity}
  L(1/2,f \otimes \chi_q)\ll \left(\frac{sq}{(s,q)}\right)^{1/3+\varepsilon}.
 \end{equation}
\end{mycoro}
Remark.  The conductor of $L(1/2, f \otimes \chi_q)$ is $sq^2/(s,q)$.  
Therefore, the bound \eqref{eq:subconvexity} is a Weyl-type subconvexity bound in $q$-aspect, but does not reach the convexity bound in the $s$-aspect.
Note that the corollary holds without a relatively prime hypothesis on $s,q$.  

Corollary \ref{coro:subconvexity} gives a non-trivial bound when the root number $\epsilon_{f \otimes \chi_q} = +1$ (since otherwise $L(1/2, f \otimes \chi_q) = 0$).  See Section \ref{section:rootnumber}, specifically equation \eqref{eq:rootnumber} for a concrete formula for the root number.

Our work here is a generalization of the cubic moment studied by Conrey and Iwaniec \cite{CI}, who obtained, in our notation, the case $r=1$, $\kappa \geq 12$, and $q$ odd.  The extension of their work to $\kappa \geq 2$ was obtained by the first-named author \cite{Petrow}.  It may be somewhat surprising that, prior to Corollary \ref{coro:subconvexity},  a Weyl-type subconvex bound in the $q$-aspect was previously not known for any values of $r$ besides $1$ (nor for even $q$).  
The case $r=1$ has some pleasant simplifications; for one, the conductor of $L(1/2, f \otimes \chi_q)$ is $q^2$ for all $f$ of level dividing $q$.  Furthermore, the $n$th Fourier coefficient of $f \otimes \chi_q$ vanishes unless $(n,q) = 1$.  For these reasons, Conrey and Iwaniec could use a formula of
 Iwaniec, Luo, and Sarnak \cite{ILS}, who proved a Petersson formula that is applicable to \eqref{eq:cubicmoment} with $r=1$.   The case $r\neq 1$ lacks these simplifications, so in  order to approach the proof of Theorem \ref{thm:cubicmoment} we developed a more general form of the Petersson formula that is applicable to \eqref{eq:cubicmoment} with any square-free $r$.  This formula, which is of independent interest, is described in Section \ref{section:introPetersson}.

Corollary \ref{coro:subconvexity} improves on a hybrid subconvexity result of Blomer and Harcos  \cite[Theorem 2']{BlomerHarcos}, which holds more generally for $f$ of arbitrary level and nebentype character. In our notation (and assuming $(q,r) = 1$) the result of Blomer and Harcos takes the form 
\begin{equation}
\label{eq:BlomerHarcosSubconvexityBound}
 L(1/2, f \otimes \chi_q) \ll (r^{1/4} q^{3/8} + r^{1/2} q^{1/4}) (rq)^{\varepsilon}.
\end{equation}
One may check that Corollary \ref{coro:subconvexity} is superior to \eqref{eq:BlomerHarcosSubconvexityBound}, except in the range $r \asymp q^{1/2 + o(1)}$ where all the bounds are equalized.
This result of Blomer and Harcos is more general in that $\chi_q$ can be replaced by an arbitrary  primitive  Dirichlet character, $f$ may be a Maass form, and 
it is not restricted to the central point.  In addition, the Blomer-Harcos bound proceeds by bounding an amplified second moment, and is Burgess-quality in the $q$-aspect for $r$ fixed.  If $q$ is fixed and $r$ is large, then the cubic moment is not the appropriate moment to use, and both Corollary \ref{coro:subconvexity} and \eqref{eq:BlomerHarcosSubconvexityBound} are weaker than the convexity bound (specifically, \eqref{eq:subconvexity} is superior to the convexity bound of $(rq^2)^{1/4+\varepsilon}$ for $r \ll q^{2-\varepsilon}$).

The work of \cite{CI} treats both holomorphic forms and Maass forms, with similar proofs.  Provided one generalizes our newform Petersson formula to the setting of the Bruggeman-Kuznetsov formula, then our methods should carry over to the Maass case, as in \cite{CI}. 
Note added May 31, 2018: the Bruggeman-Kuznetsov formula for newforms has now appeared in \cite{YoungEisenstein}.

The type of sum appearing in Theorem \ref{thm:cubicmoment} may look somewhat unusual, but it is very important for the proof.  It is crucial in \cite{CI} that, after applying the Petersson formula, the moduli of the Kloosterman sums are all divisible by $q$.  The form of \eqref{eq:cubicmoment} is chosen to group together the terms with $q' \mid q$ to give a sum of Kloosterman sums with $c \equiv 0 \pmod{q}$.  As a rough sketch of what this means, and why it is important, one may consider the case of prime level $q$.  Very roughly, one naively expects the Petersson formula to say
\begin{equation}
\label{eq:fakeNewformFormula}
\sum_{f \text{ level } 1} \frac{a_f(m) a_f(n)}{\langle f, f \rangle} 
+ \sum_{f \text{ new of level } q} \frac{a_f(m) a_f(n)}{\langle f, f \rangle}
\\
\leftrightarrow \delta_{m,n} + 2 \pi i^{-\kappa} \sum_{c \equiv 0 \shortmod{q}} \frac{S(m,n;c)}{c}.
\end{equation}
This is not quite correct because there are other types of oldforms not appearing on the left hand side, but that does not affect the broader thrust of this discussion (the reader interested in the correct version of this formula will find abundant discussion throughout this paper!).    Meanwhile, the sum over $f$ of level $1$ has a Petersson formula in which \emph{all} $c \geq 1$ appear.  Thus, by rearranging these expressions, we see that a newform formula for $f$ of level $q$ should have all $c \geq 1$ present.  With the cubic moment, one also has a factor $\chi_q(mn)$, and one wishes to apply Poisson in these variables.  The total modulus of $\chi_q(mn) S(m,n;c)$ is $[q,c]$ which for $q \mid c$ is still $c$, but if $(q,c) = 1$ it is $qc$ which is much larger.  In this latter case, Poisson summation is practically ineffectual.  

Our proof of Theorem \ref{thm:cubicmoment} in fact shows a stronger asymptotic result of the form
\begin{equation}
\label{eq:cubicmomentasymptotic}
\sum_{\substack{f \in H_{\kappa}^*(rq') \\ q' | \tilde{q}}} \omega_f L(1/2, f \otimes \chi_q)^3 = R_{r,q} + O((qr)^{\varepsilon}(r^{-1/2} + q^{-1/2} r^{-1/4})),
\end{equation}
where $\omega_f$ are certain positive weights satisfying $\omega_f = (qr)^{-1+o(1)}$ and $R_{r,q}$ is a complicated main term arising from a residue calculation (see Section \ref{section:Diagonalterms} for details).  The error term here is seen to be $o(1)$ provided $r \gg q^{\delta}$ for some fixed $\delta > 0$.  Conrey and Iwaniec \cite{CI} express interest in finding the asymptotic of the cubic moment in their case $r=1$; it is perhaps surprising that deforming the problem slightly in the $r$-aspect allows us to solve this problem in a hybrid range.  In light of \eqref{eq:cubicmomentasymptotic}, perhaps it is possible to amplify the moment in the $r$-aspect, and thereby improve the exponent of $s$ in \eqref{eq:subconvexity}.

\subsection{Arithmetical applications of the cubic moment}
The bound from Corollary \ref{coro:subconvexity} implies a bound on the Fourier coefficients of half-integral weight cusp forms,  
as we now describe.  Suppose that $g(z) = \sum_n c(n) e(nz)$ is a weight $\frac{\kappa + 1}{2}$ Hecke eigenform of  nebentypus $\psi$ and  level $4r$ where $r$ is odd and square-free, and $\kappa$ is even.  The Shimura correspondence links $g$ to a form $f$ of weight $\kappa,$ level $2r$,  and nebentypus $\psi^2$ , and Waldspurger's formula gives  under some local conditions (see \cite[Th\'eor\`eme 1]{Waldspurger})  that $c(|D|)^2 = c_0 |D|^{\frac{\kappa-1}{2}} L(1/2, f \otimes \psi^{-1}\chi_D)$ where $D$ is a fundamental discriminant, and $c_0$ is some constant of proportionality depending on $g$.   Note $f \otimes \psi^{-1}$ has trivial nebentypus.  Since $2r$ is square-free, Corollary \ref{coro:subconvexity} applies, and we deduce:
\begin{mytheo}
\label{thm:FourierCoefficientBound}
 Assume that $\psi$ is a real character.  With notation as above, we have
\begin{equation}
c(|D|) \ll_g |D|^{\frac{\kappa -1}{4} + \frac16 + \varepsilon}.
\end{equation}
\end{mytheo}

Theorem \ref{thm:FourierCoefficientBound} has applications to the representation problem for ternary quadratic forms which has been studied by a number of authors, including \cite{DukeSchulze-Pillot}\cite{DukeTernary}  \cite{Kane}  \cite{BlomerUniformTheta} \cite{BlomerTernary} from which we have drawn some of the following background material.  Suppose that $Q$ is a  positive  ternary quadratic form with associated theta function $\theta_Q$ of level dividing $4N$ with $N$ odd and square-free.  For instance, any diagonal form $ax^2 + by^2 + cz^2$ with $abc$ odd and square-free satisfies these conditions.  Then $\theta_Q = E + U + S$ where $E$ is a linear combination of Eisenstein series, $U$ is a linear combination of unary theta functions, and $S$ is a linear combination of Hecke cusp forms.  According to this decomposition, write $r_Q(n) = c_E(n) + c_U(n) + c_S(n)$ where $r_Q(n)$ is the number of representations of $n$ by $Q$, and $c_*(n)$ is the $n$th Fourier coefficient of $* = E, U, S$.
For ease of exposition, suppose that $n$ is square-free and coprime to the level, which implies $c_U(n) = 0$.  If $n$ is locally represented everywhere by $Q$, then $c_E(n) \gg_{Q} n^{1/2-\varepsilon}$.  Theorem \ref{thm:FourierCoefficientBound} implies $c_S(n) \ll_Q n^{5/12 + \varepsilon}$, which is an improvement over that derived from the Burgess-quality subconvex bound of  \cite{BlomerHarcos} .

For some more advanced questions, one may desire to explicate the dependence on $g$ in Theorem \ref{thm:FourierCoefficientBound}.  Blomer \cite{BlomerTernary} remarks that in general this is difficult, and that Mao \cite[Appendix 2]{BlomerHarcosMichel} has done this but at the expense of relating the Fourier coefficients to twisted $L$-values of an auxiliary form  $f \otimes \psi^{-1} \chi_{-4}^{\kappa/2}$, which is of level dividing $16r^2$.  Our results here then may not apply to  this auxiliary form . 

However, if $g$ is in Kohnen's plus space, then the constant of proportionality is given explicitly by the Kohnen-Zagier formula, and our results apply, as we now explain.
We gather some notation from Kohnen's paper \cite{Kohnen}, paying careful attention to normalizations.  Let $g$ be as defined in this subsection, in Kohnen's plus space, and write
$f(z) = \sum_{n=1}^{\infty} n^{\frac{\kappa-1}{2}} \lambda_f(n) e(nz)$ with $\lambda_f(1) = 1$.  Define the Petersson inner product by
\begin{equation*}
 \langle f, f \rangle_{\text{Kohnen}} = \frac{1}{[\Gamma_0(1):\Gamma_0(r)]} \int_{\Gamma_0(r) \backslash \mh} y^{\kappa} |f(z)|^2 \frac{dx dy}{y^2}.
\end{equation*}
(For the rest of the paper we will mainly use a different normalization of the inner product.)  
Using \cite[Lemma 2.5]{ILS} \cite{HoffsteinLockhart} \cite{IwaniecSmallEigenvalues}, we have $\langle f, f \rangle_{\text{Kohnen}} = r^{o(1)}$.  
Let $D$ be a fundamental discriminant with $(-1)^{\kappa/2} D > 0$, coprime to $r$.  By \cite[Corollary 1]{Kohnen}, we have
\begin{equation*}
 \frac{|c(|D|)|^2}{\langle g, g \rangle_{\text{Kohnen}}} = 2^{\nu(r)} \frac{(\frac{\kappa}{2} -1)!}{\pi^{\kappa/2}} |D|^{\frac{\kappa-1}{2}}
 \frac{L(1/2, f \otimes \chi_D)}{\langle f, f \rangle_{\text{Kohnen}}},
\end{equation*}
under the assumption  $\chi_D(p) = \eta_p(f)$ for all $p \mid r$ (here $\eta_p(f)$ is the eigenvalue of the Atkin-Lehner operator).  If $\chi_D(p) = - \eta_p(f)$ for some $p  \mid  r$ then $c(|D|) = 0$ while the right hand side may not vanish.  As an aside, we mention that Baruch and Mao \cite{BaruchMao} have generalized the Kohnen-Zagier/Waldspurger formula by removing these conditions on $D$, relating the central value to a Fourier coefficient of a different half-integral weight cusp form.
By Corollary \ref{coro:subconvexity}, we have
\begin{equation}
 \frac{|c(q)|^2}{\langle g, g \rangle_{\text{Kohnen}}} \ll_{\kappa} 
 r^{\frac13 + \varepsilon} q^{\frac{\kappa-1}{2} + \frac13 + \varepsilon}.
\end{equation}
It is also natural to inquire into the normalization of the form $g$.  There is a slight difficulty here in that we cannot scale $g$ by taking $c(1) = 1$, since $c(1)$ may vanish.  
There exists a $D_0$, polynomially bounded in $r$, so that $L(1/2, f \otimes \chi_{D_0}) \gg r^{o(1)}$ (e.g., see \cite{HoffsteinKontorovich}).  Then we may choose the constant of normalization so that $|c(|D_0|)|^2 = |D_0|^{\frac{\kappa - 1}{2}}$.  Then with this normalization, $\langle g, g \rangle_{\text{Kohnen}} \ll r^{o(1)}$, and hence
\begin{equation}
|c(q)| \ll_{\kappa, \varepsilon} r^{\frac16+\varepsilon} q^{\frac{\kappa-1}{4} + \frac{1}{6} + \varepsilon}.
\end{equation}

Theorem \ref{thm:cubicmoment} itself can be used to improve many exponents in the results of \cite{LMY2}.  In particular, we improve the rate of equidistribution of the reductions of CM elliptic curves (see \cite{LMY2} for a full description of this arithmetical problem).  For brevity, we shall not repeat any material from \cite{LMY2}, but will instead indicate which exponents may be improved.  The bound $q^{1/8+\varepsilon} D^{7/16 + \varepsilon}$ in \cite[(1.5)]{LMY2} may be replaced by $q^{\varepsilon} D^{5/12+\varepsilon}$.  In \cite[Corollary 1.3]{LMY2}, the bound $D \gg q^{18+\varepsilon}$ may be replaced by $D \gg q^{12+\varepsilon}$.  In \cite[(1.10), (1.12)]{LMY2}, the bound $q^{7/8} D^{7/16}$ may be replaced by $q^{3/4} D^{5/12}$.  All these changes result from a use of Theorem \ref{thm:cubicmoment} to bound $M_2$ defined by \cite[(4.7), (3.1)]{LMY2} with 
\begin{equation}
 M_2 \ll \frac{D^{1/2+\varepsilon}}{q^{1/3-\varepsilon}} \Big(\sum_{f \in H_2^*(q)}
 L(1/2, f \otimes \chi_D)^3 \Big)^{1/3} \ll q^{\varepsilon} D^{5/6+\varepsilon}.
\end{equation}

If one can generalize Theorem \ref{thm:cubicmoment} (and hence Corollary \ref{coro:subconvexity}) to allow  $f$ to be a Hecke-Maass cusp form, then there are additional applications.  This is the setting required for equidistribution of integral points on ellipsoids \cite{DukeEquidistribution}.  The various exponents appearing in \cite{LMY1} would be updated similarly to the improvements to \cite{LMY2} described in the previous paragraph.
As another example in this vein, Folsom and  Masri \cite{FolsomMasri} \cite{Masri} have improved the error term in the asymptotic formula for the partition function which requires subconvexity for quadratic twists of a cusp form of level $6$; the previous bounds of \cite{CI} do not apply, and so the methods developed in this paper pave the way for further improvements.  

The second-named author \cite{Young} generalized the \cite{CI} method allowing for large weights (or spectral parameters, in the Maass case) giving a Weyl-type hybrid subconvexity bound.  This had applications to equidistribution problems on shrinking sets.  For simplicity, in this paper we have kept the weight $\kappa$ fixed but it seems likely that the methods of \cite{Young} could be combined with those in this paper to allow $\kappa$ to vary.  

\subsection{Petersson formula for newforms}
\label{section:introPetersson}
We begin with an expanded discussion on why a newform Petersson formula is relevant for Theorem \ref{thm:cubicmoment}.

One encounters a significant difficulty when attempting to generalize the \cite{CI} method to allow level structure of the base form, as we describe.  One begins by using an approximate functional equation of the $L$-function $L(1/2, f \otimes \chi_q)$, which has conductor $rq^2$ when $f$ is a newform of level $rq'$ with $(r,q) = 1$ and $q' \mid q$.  Next one would wish to apply the Petersson formula to average over an orthogonal basis of cusp forms.  The problem is that this basis consists of oldforms as well as newforms, which causes a variety of problems.  Firstly, it is not clear what Dirichlet series to attach to $f \otimes \chi_q$ when $f$ is an oldform.  One could take this to mean that if $f$ is induced from a newform $f^*$ of lower level, then we take $L(1/2, f^* \otimes \chi_q)$.  However, with this definition the conductor of this $L$-function may be a divisor of $rq^2$, in which case there is some dependence on the level of $f^*$ in the approximate functional equation.  The classical Petersson formula is unable to distinguish between these forms.

It is plausible that there is some trick that lets one set up the problem to prove Theorem \ref{thm:cubicmoment} using the classical Petersson formula, but 
the authors are not aware of one (if the moment was an even power, this would be easy because of positivity; the fact that the moment is an odd power in this application makes this more difficult).  

The robust solution is to prove a Petersson formula for the newforms only, similarly to the existence of averaging formulas for primitive Dirichlet characters of a given modulus (see \cite[(3.8)]{IK}).  Iwaniec, Luo, and Sarnak have proven a Petersson formula for newforms of square-free level \cite[Proposition 2.8]{ILS}, but with some coprimality conditions on the level and the Fourier coefficients of the modular forms, which in our application are crucial to avoid.  When working with $1$-level density of zeros of $L$-functions, it is easy to ensure coprimality because the log derivative of an $L$-function is a sum over  prime powers.  However, the $L$-function itself is not so easily treated, because altering a single Euler factor will ruin the functional equation.  For this reason, we have generalized the \cite{ILS} formula to hold with square-free level and arbitrary Fourier coefficients.

Suppose that $N$ is a positive integer, and let $B_\kappa(N)$ denote an orthogonal basis for the space of weight $\kappa$ cusp forms for $\Gamma_0(N)$.  For $f \in B_\kappa(N)$, write $f(z) = \sum_{n=1}^{\infty} a_f(n) e(nz)$, and $a_f(n) = \lambda_f(n) n^{\frac{\kappa-1}{2}}$.  
Let
\begin{equation}
\label{eq:DeltaDef}
 \Delta_{N}(m,n) = c_\kappa \sum_{f \in B_\kappa(N)} \frac{\overline{\lambda_f(m)} \lambda_f(n)}{\langle f, f \rangle_N}, \quad
 \text{where} 
 \quad
 c_\kappa = \frac{\Gamma(\kappa-1)}{(4\pi)^{\kappa-1}},
\end{equation}
and where
\begin{equation*}
 \langle f, g \rangle_N = \int_{\Gamma_0(N) \backslash \mh} y^\kappa f(z) \overline{g(z)} \frac{dx dy}{y^2}.
\end{equation*}
Since the main interest here is in the level aspect, we often suppress the dependence on the weight $\kappa$ in the notation.
The Petersson formula states
\begin{equation}
\label{eq:PeterssonFormulaWithKloosterman}
 \Delta_{N}(m,n) = \delta_{m=n} + 2 \pi i^{-\kappa} \sum_{c \equiv 0 \shortmod{N}} \frac{S(m,n;c)}{c} J_{\kappa-1}\Big(\frac{4 \pi \sqrt{mn}}{c} \Big).
\end{equation}

We have
\begin{mytheo}
\label{thm:PeterssonNewforms}
 Let $N$ be square-free, and let $H_\kappa^*(N)$ denote the set of Hecke-normalized newforms on $\Gamma_0(N)$  of trivial central character.  Let 
 \begin{equation}
 \Delta_{N}^*(m,n) = c_\kappa \sum_{f \in H_\kappa^*(N)} \frac{\overline{\lambda_f(m)} \lambda_f(n)}{\langle f, f \rangle_N}.
\end{equation}
Then with $\nu(L)$ defined to be the completely multiplicative function satisfying $\nu(p) = p+1$ for $p$ prime, we have
 \begin{multline}
 \label{eq:DeltakN*Formula}
 \Delta_{N}^*(m,n) = \sum_{LM=N}\frac{\mu(L)}{\nu(L)} \sum_{\ell | L^{\infty}}  \frac{\ell}{\nu(\ell)^2 } \sum_{d_1, d_2 | \ell} c_{\ell}(d_1) c_{\ell}(d_2) 
 \sum_{\substack{u|(m,L) \\ v | (n,L)}} \frac{u v }{(u, v)} \frac{\mu(\frac{u v}{(u, v)^2})}{\nu(\frac{u v}{(u, v)^2})}
 \\
\sum_{\substack{a | (\frac{m}{u}, \frac{u}{(u, v)}) \\ b | (\frac{n}{v}, \frac{v}{(u, v)})}}   
\sum_{\substack{e_1 | (d_1, \frac{m}{a^2 (u,v)}) \\ e_2 | (d_2, \frac{n}{b^2 (u,v)})} }
 \Delta_{M}\Big(
\frac{m d_1}{a^2 e_1^2 (u, v)}, 
 \frac{n d_2}{b^2 e_2^2 (u, v)}\Big).
\end{multline}
Here $c_{\ell}(d)$  with $d\mid \ell$ is  
jointly  multiplicative, and $c_{p^n}(p^j) = c_{j,n}$  where
\begin{equation}
 x^n = \sum_{j=0}^{n} c_{j,n} U_j\Big(\frac{x}{2}\Big),
\end{equation}
and $U_j(x)$ are the Chebyshev polynomials of the second kind.
\end{mytheo}
The constants $c_{j,n}$ arise from repeated application of the Hecke multiplicativity relations and we call them the \emph{Chebyshev coefficients}.  We describe some of their relevant properties in Section \ref{section:Chebyshev}, for instance, we shall show $c_{j,n} \geq 0$, and derive sharp bounds on $c_{j,n}$.  Many of the bounds on the Chebyshev coefficients appearing in Section \ref{section:Chebyshev} arose out of necessity for the proof of Theorem \ref{thm:cubicmoment}.

In Theorem \ref{thm:PeterssonNewformsTruncatedVersion}, we give an approximate version of \eqref{eq:DeltakN*Formula} with the additional restriction $\ell \leq Y$, which makes the right hand side a finite sum.  For our application to the cubic moment, we have found the approximate version most suitable.

The method of proof of \cite{ILS} is to explicitly choose a basis $B_k(N)$ (see \cite[Proposition 2.6]{ILS}) that relates the oldforms to the newforms, and thereby deduce an arithmetic formula for $\Delta_N(m,n)$ in terms of $\Delta_M^*(m',n')$'s, with $M  \mid  N$.  An inversion of this formula then gives their formula for $\Delta_N^*(m,n)$.
As mentioned in \cite{ILS}, there are many interesting choices of basis and it could be argued that their choice is ad-hoc.  Other authors have also constructed various bases.  Choie and Kohnen \cite[Proposition 2]{ChoieKohnen} use the same basis that we will use here.  Rouymi \cite{Rouymi} gave a basis for prime power level and derived a newform Petersson formula from it.  Building on Rouymi,  Ng \cite{NgBasis}  and Blomer and Mili{\'c}evi{\'c} \cite[(3.7)]{BlomerMili}  gave  a basis for arbitrary level, and   Ng \cite{NgBasis} and  five authors \cite{BBDDM}  have  used \ this basis  to give  newform Petersson formulas for arbitrary level (but with restrictive coprimality conditions on the level and the Fourier coefficients).  It is important for our work that there is no restriction on $m,n$ appearing in Theorem \ref{thm:PeterssonNewforms}.

Nelson \cite{Nelson} has described a method for proving a Petersson formula for newforms without explicitly choosing a basis, and gives such a formula when the level $N$ is divisible by the cube of each prime dividing it.  

Our proof takes a different path from \cite{ILS,Rouymi,NgBasis,BlomerMili,BBDDM} in that we choose our basis to be eigenfunctions of the Atkin-Lehner operators, which for square-free level is enough to diagonalize the basis.  This choice is natural and leads to many pleasant simplifications.
Our method of proof of Theorem \ref{thm:PeterssonNewforms} most naturally shows
\begin{multline}
\label{eq:PeterssonOldforms}
 \Delta_{N}(m,n) = \sum_{LM=N}\frac{1}{\nu(L)} \sum_{\ell | L^{\infty}}  \frac{\ell}{\nu(\ell)^2 } \sum_{d_1, d_2 | \ell} c_{\ell}(d_1) c_{\ell}(d_2) 
 \sum_{\substack{u|(m,L) \\ v | (n,L)}} \frac{u v }{(u, v)} \frac{\mu(\frac{u v}{(u, v)^2})}{\nu(\frac{u v}{(u, v)^2})}
 \\
\sum_{\substack{a | (\frac{m}{u}, \frac{u}{(u, v)}) \\ b | (\frac{n}{v}, \frac{v}{(u, v)})}}   
\sum_{\substack{e_1 | (d_1, \frac{m}{a^2 (u,v)}) \\ e_2 | (d_2, \frac{n}{b^2 (u,v)})} }
 \Delta_{M}^*\Big(
\frac{m d_1}{a^2 e_1^2 (u, v)}, 
 \frac{n d_2}{b^2 e_2^2 (u, v)}\Big).
\end{multline}
We deduce Theorem \ref{thm:PeterssonNewforms} from \eqref{eq:PeterssonOldforms} in Section \ref{section:sieve} below.

In this introduction, we have not presented the Petersson formula that is required for the proof of Theorem \ref{thm:cubicmoment}.  What we need is a kind of hybrid formula for modular forms of level $rq$ that in the $r$-aspect restricts to newforms of level $r$, and in the $q$-aspect groups together all the newforms of level dividing $q$, in accordance with the setup of Theorem \ref{thm:cubicmoment}.  This formula appears in Section \ref{section:PeterssonHybrid}.

The newform formula of \cite[Proposition 2.8]{ILS} has coprimality assumptions of the form $(m,N) = 1$ and $(n, N^2)  \mid  N$, which on the face of it is rather restrictive, however, one may reduce to this case as follows.  Firstly, using that $\lambda_f(d) \lambda_f(p) = \lambda_f(dp)$ for any $d \in \mathbb{N}$, and $p  \mid  N$, one may write $\lambda_f(m) \lambda_f(n) = \lambda_f(m') \lambda_f(n')$ where $mn = m' n'$, and $(m',N) = 1$.  
Secondly, we have $\lambda_f(p^2 d) = \lambda_f(p)^2 \lambda_f(d) = p^{-1} \lambda_f(d)$ (see \eqref{eq:etapfFormula}), which by repeated applications  allows one to reach the case $(n,N^2)  \mid  N$.   
It is not obvious how to use the \cite{ILS} newform formula to derive our hybrid version presented in Section \ref{section:PeterssonHybrid} below.
The problem is that the above factorizations of $m$ and $n$ depend on the ambient level, and so summing over different levels introduces some complications.

\subsection{Structure of the paper}
Sections \ref{section:AtkinLehnerTheory}--\ref{section:PeterssonApproximate} are devoted to proving a number of versions of the Petersson formula with newforms as well as some estimates for the Chebyshev coefficients.  This part of the paper is self-contained.

In Sections \ref{section:Structure}--\ref{section:recombination}, we prove the cubic moment bound, that is, Theorem \ref{thm:cubicmoment}.

\section{Atkin-Lehner theory}
\label{section:AtkinLehnerTheory}

\subsection{Construction of basis}
\label{section:AtkinLehner}
We briefly review some of the theory developed by Atkin and Lehner \cite{AtkinLehner}.  Throughout we assume that the level $N$ is square-free.
For a matrix in $GL_2^{+}(\mz)$, define
\begin{equation*}
 f_{\vert \left(\begin{smallmatrix} a & b \\ c & d \end{smallmatrix}\right)}(z) = (ad-bc)^{\kappa/2} (cz+d)^{-\kappa} f\Big(\frac{az+b}{cz+d}\Big).
\end{equation*}
Atkin and Lehner showed that
\begin{equation*}
 S_\kappa(N) = \bigoplus_{LM = N} \bigoplus_{f \in H_\kappa^*(M)} S_\kappa(L;f),
\end{equation*}
where $S_\kappa(L;f)$ is the span of forms $f_{\vert A_{\ell}}$, with $\ell  \mid  L$, where $A_{\ell} = \begin{pmatrix} \ell & 0 \\ 0 & 1 \end{pmatrix}$.  They call $S_\kappa(L;f)$ the \emph{oldclass} associated to $f$.
Observe
$f_{\vert A_\ell} (z) = \ell^{\kappa/2} f(\ell z)$, so $S_\kappa(L;f) = \text{span}\{ f(\ell z) : \ell  \mid  L\}$.
Our goal here is to construct an explicit orthogonal basis of $S_\kappa(L;f)$, in the case that $N$ is square-free.

We turn to the Atkin-Lehner operators $W_d$.  
Suppose that $d \mid N$, $N=LM$, and let
\begin{equation}
\label{eq:Wqdef}
 W_d = \begin{pmatrix} dx  & y \\ Nz  & d w \end{pmatrix},
\end{equation}
where $x,y,z,w \in \mz$ are chosen so that $\det(W_d) = d$ (such a choice exists because $(d,N/d) = 1$, since $N$ is square-free, and the forthcoming properties of $W_d$ are independent of the choices of $x,y,z,w$).  If $d\mid M$ and $f \in H_\kappa^*(M)$ then $f$ is an eigenfunction of $W_d$ (see \cite[Theorem 3 (iii)]{AtkinLehner}),   so suppose now that $d\mid L$.  Let
\begin{equation}
V = \begin{pmatrix} 1 & 0 \\ \frac{-LM}{d} & 1 \end{pmatrix} \in \Gamma_0(M).
\end{equation}
Note that taking $x=z=1$ in the definition of $W_d$ we have 
\begin{equation*}
V W_d = \begin{pmatrix} d & y \\ 0 & 1 \end{pmatrix}.
\end{equation*}
Therefore, if $f \in H_\kappa^*(M)$ and $d\mid L$, then
\begin{equation}
\label{eq:fWqcalculation}
f_{\vert W_d} = f_{\vert V}{}_{\vert W_d} = f_{\vert V W_d} = d^{\kappa/2} f(dz+y) = f_{\vert A_d}.
\end{equation}
This calculation easily shows that $S_\kappa(L;f)$ is preserved by all $W_d$ with $d  \mid  L$, that each $W_d$ is an involution, and that the $W_d$ commute with each other.  Therefore, the group of transformations of $S_\kappa(L;f)$ generated by the $W_d$ is isomorphic to $(\mz/2\mz)^{\omega(L)}$, where $\omega(n)$ is the number of prime divisors of $n$.  Note $2^{\omega(L)} = \tau(L)$.
Furthermore, the $W_d$ are Hermitian with respect to the Petersson inner product (see \cite[Lemma 25]{AtkinLehner}).
By some simple character theory, we can show that $S_\kappa(L;f)$ has an explicit orthogonal basis of common eigenfunctions of the $W_d$.

We briefly describe a more abstract statement. 
Let $G$ be a group isomorphic to $(\mz/ 2\mz )^n$, and let $\phi$ be a character on $G$, which we denote by $\phi \in \widehat{G}$.  There are $2^n$ such characters.  For each $g \in G$, suppose there exists an involution $W_g$ acting on some vector space of functions, and such that $W_{g_1} W_{g_2} = W_{g_1 g_2}$.  For each $f$ in the vector space and character $\phi \in \widehat{G}$, define
\begin{equation}
 f_{\phi} = \sum_{g \in G} \phi(g) W_g f.
\end{equation}
It is easy to see that
\begin{equation*}
 W_{g} f_{\phi} = \phi(g) f_{\phi}.
\end{equation*}
Therefore, each $f_{\phi}$ is an eigenfunction of all the $W_{g}$.  Also, the $f_{\phi}$ are distinct because any two choices of $f_{\phi}$ have a different eigenvalue for some $W_{g}$.  This also means that if the $W_g$ are Hermitian with respect to some inner product, then all the $f_{\phi}$ are orthogonal.  In the case of $S_\kappa(L;f)$, which has dimension $2^{\omega(L)} = \tau(L)$, there are $2^{\omega(L)}$ eigenfunctions $f_{\phi}$, so by dimension counting, the $f_{\phi}$ form a basis. 
Finally, we derive a useful formula for $\langle f_{\phi}, f_{\phi} \rangle$:
\begin{multline}
 \langle f_{\phi}, f_{\phi} \rangle = \sum_{g_1, g_2 \in G} \phi(g_1) \phi(g_2) \langle W_{g_1} f , W_{g_2} f \rangle 
 \\
 = \sum_{g_1, g_2 \in G} \phi(g_1 g_2) \langle W_{g_1 g_2} f ,  f  \rangle = |G| \sum_{g \in G} \phi(g) \langle W_g f, f \rangle.
\end{multline}

Returning to $S_\kappa(L;f)$, by \cite[Lemma 2.4]{ILS} (which in turn follows closely a proof in \cite{AbbesUllmo}), we have
\begin{equation*}
 \langle f_{\vert W_d}, f \rangle = \langle f_{\vert A_d}, f \rangle = \frac{\lambda_f(d)}{\nu(d)} d^{1/2} \langle f, f \rangle.
\end{equation*}
 We endow the set of divisors $d \mid L$ with the group structure $(\mz/ 2\mz )^{\omega(L)}$ and define characters on it by $\phi(d) = \prod_{p | d} \phi(p)$, where $\phi(p)$ is chosen to be $+1$ or $-1$ independently for each prime divisor of $L$.   In this way, we obtain
\begin{equation}
\label{eq:fphiPeterssonNorm}
 \langle f_{\phi}, f_{\phi} \rangle = |G| \langle f, f \rangle \sum_{d | L} \phi(d) \frac{\lambda_f(d)}{\nu(d)} d^{1/2} = \tau(L) \langle f, f \rangle \prod_{p | L} \Big(1 + \frac{\phi(p) \lambda_f(p)p^{1/2}}{\nu(p)} \Big).
\end{equation}
All of the above inner products are $\langle\, , \rangle_N$.

\subsection{Dirichlet series of the basis of oldforms}
To lend some support to the assertion that our choice of basis of $S_\kappa(L;f)$ given above is natural, here we describe some pleasant features of the Dirichlet series corresponding to these modular forms.  Let $\lambda_{f_\phi}(n)$ be the Fourier coefficients of $f_\phi$ and  
define 
\begin{equation}
D(s,f_{\phi}) = \sum_{n=1}^{\infty} \frac{\lambda_{f_{\phi}}(n)}{n^s}.
\end{equation}
The reader should beware that this is not a character twist of $f$, because $\phi$ is not a Dirichlet character (in fact $\phi$ is only defined on the divisors of $L$).  We show here that $D(s, f_{\phi})$ satisfies a functional equation similar to that of a level $N$ newform.

By a direct calculation with the Fourier expansion, we have
\begin{equation}
\label{eq:fphiFourierCoefficients}
 \lambda_{f_{\phi}}(m) = \sum_{u  | (m,L)} \phi(u) u^{1/2} \lambda_f(m/u),
\end{equation}
  Therefore, we have
\begin{equation}
D(s, f_{\phi})  =  L(s,f) \prod_{p | L} \Big(1 + \frac{\phi(p)}{p^{s-1/2}} \Big).
\end{equation}
Then define the ``completed" Dirichlet series
\begin{equation}
N^{s/2} \Gamma_{f}(s) D(s, f_{\phi}) = 
\Lambda(s,f) L^{s/2} \prod_{p | L} \Big(1 + \frac{\phi(p)}{p^{s-1/2}} \Big),
\end{equation}
where $$\Gamma_f(s) = \pi^{-s} \Gamma\left( \frac{s+(\kappa-1)/2}{2}\right)\Gamma\left( \frac{s+(\kappa+1)/2}{2}\right)$$ is the gamma factor associated to  $L(s,f)$  and $ \Lambda(s,f)  =  M^{s/2} \Gamma_f(s)  L(s,f) $.  This satisfies the functional equation  $ \Lambda(s,f)  = \epsilon_f  \Lambda(1-s,f) $.  Meanwhile, the secondary factor satisfies
\begin{equation*}
g(s) := L^{s/2} \prod_{p | L} \Big(1 + \frac{\phi(p)}{p^{s-1/2}} \Big) = \prod_{p | L} \Big( p^{s/2} + \phi(p) p^{\frac{1-s}{2}} \Big) = \phi(L) g(1-s).
\end{equation*}
Therefore, $D(s, f_{\phi})$ satisfies the functional equation
\begin{equation}
N^{\frac{s}{2}} \Gamma_{f}(s) D(s, f_{\phi}) = \epsilon_f \phi(L) N^{\frac{1-s}{2}} \Gamma_{f}(1-s) D(1-s,f_{\phi}).
\end{equation}

\section{Manipulations with sums of Fourier coefficients}
\label{section:FourierManipulations}
The goal of this section is to prove \eqref{eq:PeterssonOldforms}.

We begin by describing \eqref{eq:DeltaDef} for the basis chosen in Section \ref{section:AtkinLehner}.  We have
\begin{equation}
\label{eq:PeterssonWithPhis}
\Delta_{N}(m,n) = c_\kappa \sum_{LM=N} \sum_{f \in H_\kappa^*(M)} \sum_{\phi}\
 \frac{\overline{\lambda_{f_{\phi}}(m)} \lambda_{f_{\phi}}(n)}{\langle  f_{\phi}, f_{\phi} \rangle_N}.
\end{equation}
We therefore need to evaluate the inner sum over $\phi$, namely
\begin{multline}
\label{eq:chiinnerproductsum}
T(m,n):= \sum_{\phi} \frac{\overline{\lambda_{f_{\phi}}(m)} \lambda_{f_{\phi}}(n)}{\langle  f_{\phi}, f_{\phi} \rangle_N} \\ = \frac{1}{\tau(L) \langle f, f \rangle}_N \sum_{\phi} \overline{\lambda_{f_{\phi}}(m)} \lambda_{f_{\phi}}(n) \prod_{p | L} \Big(1 + \frac{\phi(p) \lambda_f(p)p^{1/2}}{\nu(p)} \Big)^{-1},
\end{multline}
where we have used \eqref{eq:fphiPeterssonNorm}.
We multiply and divide by $\prod_{p | L} (1-\frac{\phi(p)\lambda_f(p) p^{1/2}}{\nu(p)})$, giving that
\begin{equation*}
T(m,n)= \frac{1}{\tau(L) \rho_f(L) \langle f, f \rangle}_N \sum_{\phi} \overline{\lambda_{f_{\phi}}(m)} \lambda_{f_{\phi}}(n) \sum_{t | L} \frac{\mu(t) \phi(t) \lambda_f(t) t^{1/2}}{\nu(t)},
\end{equation*}
where as in \cite{ILS}, we define
\begin{equation}\label{eq:rhof}
 \rho_f(L) = \prod_{p  |  L} \Big(1 - p\frac{ \lambda_f(p)^2}{(p+1)^2}\Big).
\end{equation}

The formula \eqref{eq:fphiFourierCoefficients} implies
\begin{multline*}
T(m,n) =  \frac{1}{\tau(L) \rho_f(L) \langle f, f \rangle}_N \sum_{\phi} \sum_{u  | (m,L)}
 \sum_{v  | (n,L)} 
 \\
 \sum_{t | L} 
 \phi(u) u^{1/2} \lambda_f(m/u) \phi(v) v^{1/2} \lambda_f(n/v)  \frac{\mu(t) \phi(t) \lambda_f(t) t^{1/2}}{\nu(t)},
\end{multline*}
where we have used that $\lambda_f(n)$ is real to remove the complex conjugate symbols.
The sum over $\phi$ detects if $uvt$ is a square, precisely
\begin{equation*}
 \sum_{\phi} \phi(u) \phi(v) \phi(t) = \begin{cases}
                                         \tau(L), \qquad &\text{if } uvt = \square, \\
                                         0, \qquad &\text{otherwise}.
                                       \end{cases}
\end{equation*}
Since $u$ and $v$ are square-free, and so is $t$, the condition $uvt = \square$ determines $t$ uniquely, namely
\begin{equation*}
 t = \frac{uv}{(u, v)^2} = \frac{u}{(u , v)} \frac{v}{(u , v)}.
\end{equation*}
One may compare this with Lemma 2.4 of \cite{ILS}.  Therefore, $T(m,n)$ equals
\begin{equation}
\label{eq:TmnNiceFormula}
 \frac{1}{\rho_f(L) \langle f, f \rangle}_N  \sum_{u   | (m,L)}
 \sum_{v   | (n,L)}
  u^{1/2} \lambda_f(m/u) v^{1/2} \lambda_f(n/v)  \\ \frac{\mu(\frac{u v}{(u, v)^2})  \lambda_f(\frac{u v}{(u , v)^2}) (\frac{u  v}{(u , v)^2})^{1/2}}{\nu(\frac{uv}{(u , v)^2})}.
\end{equation}

To check this against \cite{ILS}, suppose that $(m,n,N) = 1$.  This means $((m,L), (n,L)) = (m,n,L)  \mid  (m,n,N) = 1$, so in particular, $(u, v) = 1$.  Hence 
\begin{equation}
T(m,n) = \\ \frac{1}{\rho_f(L) \langle f, f \rangle}_N  \sum_{u   | (m,L)} \frac{\mu(u ) u   }{\nu(u )} \lambda_f(u ) \lambda_f(m/u )
 \sum_{v  | (n,L)} 
 \frac{\mu(v) v}{\nu(v)}
\lambda_f(v) \lambda_f(n/v),
\end{equation}
which equals
\begin{equation*}
 \frac{A_f(m,L) A_f(n,L)}{\rho_f(L) \langle f, f \rangle_N}, \quad \text{where}  \quad A_f(m,L) := \sum_{u | (m,L)}  \frac{\mu(u) u}{\nu(u)} \lambda_f(u) \lambda_f(m/u),
\end{equation*}
as in \cite[p.76]{ILS}.   Since $f$ is on $\Gamma_0(M)$, we have $\langle f, f \rangle_N = \frac{\nu(N)}{\nu(M)} \langle f, f \rangle_M = \nu(L) \langle f, f \rangle_M$. So,  if $(m,n,N) = 1$, then
\begin{equation}
\label{eq:coprimeFormulaFromILS}
 \Delta_{N}(m,n) = c_\kappa \sum_{LM=N} \sum_{f \in H_\kappa^*(M)} \frac{A_f(m,L) A_f(n,L)}{\rho_f(L) \nu(L) \langle f, f \rangle_M}.
\end{equation}
From this we may quickly derive (2.48) of \cite{ILS}.

We continue with the calculation of $\Delta_N(m,n)$, without the assumption $(m,n,N) = 1$. The formula \eqref{eq:TmnNiceFormula} shows
\begin{multline*}
 \Delta_{N}(m,n) = c_\kappa \sum_{LM=N} \sum_{f \in H_\kappa^*(M)} \frac{1}{\nu(L) \rho_f(L) \langle f, f\rangle_M} 
 \\
 \sum_{\substack{u | (m,L) \\ v  |  (n,L)}} \frac{u v }{(u, v)} \frac{\mu(\frac{uv}{(u, v)^2})}{\nu(\frac{uv}{(u, v)^2})}
 \lambda_f\Big(\frac{m}{u}\Big) \lambda_f\Big(\frac{u}{(u,v)}\Big)
 \lambda_f\Big(\frac{n}{v}\Big) \lambda_f\Big(\frac{v}{(u,v)}\Big).
\end{multline*}

Recall that the Hecke relation for a newform of level $M$ with trivial nebentypus is
\begin{equation*}
 \lambda_f(m) \lambda_f(n) = \sum_{\substack{d | (m,n) \\ (d,M) = 1}} \lambda_f\Big(\frac{mn}{d^2}\Big).
\end{equation*}
In our desired application, $u$ and $v$ divide $L$ and $(M,L) = 1$, so any divisor of $u$ or $v$ is automatically coprime to $M$.
Using the Hecke relation, we then deduce
\begin{multline}
\label{eq:DeltakmnformulaWithRho}
 \Delta_{N}(m,n) = c_\kappa \sum_{LM=N} \sum_{f \in H_\kappa^*(M)} \frac{1}{\nu(L) \rho_f(L) \langle f, f\rangle_M} 
 \\
 \sum_{\substack{u | (m,L) \\ v  |  (n,L)}} \frac{u v }{(u, v)} \frac{\mu(\frac{u v}{(u, v)^2})}{\nu(\frac{u v}{(u, v)^2})}
\sum_{\substack{a  |  (\frac{m}{u}, \frac{u}{(u, v)}) \\ b  |  (\frac{n}{v}, \frac{v}{(u, v)})}}   
 \lambda_f\Big(\frac{m}{a^2 (u, v)}\Big)
 \lambda_f\Big(\frac{n}{b^2 (u, v)}\Big).
\end{multline}

The tricky part of our analysis of $\Delta_N(m,n)$ is to express $\rho_f(L)^{-1}$ in terms of Fourier coefficients of $f$.
We have
\begin{equation}
\label{eq:rhoinverseformula1}
 \frac{1}{\rho_f(L)} = \prod_{p  |  L} \Big(1- p \frac{\lambda_f(p)^2}{(p+1)^2}\Big)^{-1} = \sum_{\ell  |  L^{\infty}} \ell \frac{\lambda_f^*(\ell)^2}{\nu(\ell)^2},
\end{equation}
where $\lambda_f^*(\ell)$ is the completely multiplicative version of $\lambda_f(n)$, that is,
\begin{equation*}
 \lambda_f^*(\ell) = \prod_{p^n || \ell} \lambda_f(p)^n.
\end{equation*}
Using only the weak bound $|\lambda_f(p)| \leq p^{\theta} + p^{-\theta}$ with some $\theta < 1/2$ shows that the product and sum in \eqref{eq:rhoinverseformula1}
converge absolutely.

Define the \emph{Chebyshev coefficients} $c_{j,n}$ by
\begin{equation}
 \lambda_f(p)^n = \sum_{j=0}^{n} c_{j,n} \lambda_f(p^j),
\end{equation}
where $p$ is coprime to the level of $f$.
Let $U_k(x)$ denote the degree $k$ Chebyshev polynomial of the second kind  (defined below). Then 
\begin{equation*}
 x^n = \sum_{j=0}^{n} c_{j,n} U_j\Big(\frac{x}{2}\Big),
\end{equation*}
where the $c_{j,n}$ can be written in various ways using that the $U_j$ form a system of orthogonal polynomials. Here the $U_j$ can be defined concisely by the generating function
\begin{equation*}
 (1 - 2yx + x^2)^{-1} = \sum_{j=0}^{\infty} U_j(y) x^j.
\end{equation*}
For instance, since $U_4(x/2) = x^4 - 3x^2 + 1$, $U_2(x/2) = x^2-1$, and $U_0(x) = 1$, we get
\begin{equation*}
 x^4 = U_4(x/2) + 3 U_2(x/2) + 2 U_0(x/2).
\end{equation*}
An alternative formula is $U_j(\cos(\theta)) = \frac{\sin((j+1)\theta)}{\sin \theta}$.  
The orthogonality of the $U_j$ implies that
\begin{equation}
\label{eq:chebyshevCoefficientIntegralFormula}
 c_{j,n} = \int_0^{\pi} U_j(\cos \theta) (2 \cos \theta)^n \tfrac{2}{\pi} \sin^2 \theta d \theta.
\end{equation}
We will develop some properties of the Chebyshev coefficients in Section \ref{section:Chebyshev}.

With this notation in hand, we have for $f$ a newform of level $M$ with $(\ell, M) = 1$, that
\begin{equation}
 \lambda_f^*(\ell) = \prod_{p^n || \ell} (\sum_{j=0}^{n} c_{j,n} \lambda_f(p^j)) =: \sum_{d | \ell} c_{\ell}(d) \lambda_f(d),
\end{equation}
where 
\begin{equation*}
 c_{\ell}(d) = \prod_{\substack{p^j || d \\ p^n || \ell}} c_{j, n}.
\end{equation*} 
Moreover, we have
\begin{equation}
\label{eq:rhoinverseformula}
 \frac{1}{\rho_f(L)} = \sum_{\ell | L^{\infty}}  \frac{\ell}{\nu(\ell)^2} \sum_{d_1, d_2 | \ell} c_{\ell}(d_1) c_{\ell}(d_2) \lambda_f(d_1) \lambda_f(d_2).
\end{equation}

Inserting \eqref{eq:rhoinverseformula} into \eqref{eq:DeltakmnformulaWithRho}, we get
\begin{multline*}
 \Delta_{N}(m,n) = c_\kappa \sum_{LM=N} \sum_{\ell | L^{\infty}}  \frac{\ell}{\nu(\ell)^2 \nu(L)} \sum_{d_1, d_2 | \ell} c_{\ell}(d_1) c_{\ell}(d_2) \sum_{f \in H_\kappa^*(M)} \frac{1}{\langle f, f\rangle_M} \\ \lambda_f(d_1) \lambda_f(d_2)
 \sum_{\substack{u|(m,L) \\ v | (n,L)}} \frac{u v }{(u, v)} \frac{\mu(\frac{u v}{(u, v)^2})}{\nu(\frac{u v}{(u, v)^2})}
\sum_{\substack{a | (\frac{m}{u}, \frac{u}{(u, v)}) \\ b | (\frac{n}{v}, \frac{v}{(u, v)})}}   
 \lambda_f\Big(\frac{m}{a^2 (u, v)}\Big)
 \lambda_f\Big(\frac{n}{b^2 (u, v)}\Big).
\end{multline*}
Now we can use the Hecke relations one final time (again the divisors are coprime to $M$), to give
\begin{multline*}
 \Delta_{N}(m,n) = c_\kappa \sum_{LM=N} \sum_{\ell | L^{\infty}}  \frac{\ell}{\nu (\ell)^2 \nu(L)} \sum_{d_1, d_2 | \ell} c_{\ell}(d_1) c_{\ell}(d_2) \sum_{f \in H_\kappa^*(M)} \frac{1}{\langle f, f\rangle_M}   
 \\
 \sum_{\substack{u|(m,L) \\ v | (n,L)}} \frac{ u v }{(u, v)} \frac{\mu(\frac{u v}{(u, v)^2})}{\nu(\frac{u v}{(u, v)^2})}
\sum_{\substack{a | (\frac{m}{u}, \frac{u}{(u, v)}) \\ b | (\frac{n}{v}, \frac{v}{(u, v)})}}   
\sum_{\substack{e_1 | (d_1, \frac{m}{a^2 (u,v)}) \\ e_2 | (d_2, \frac{n}{b^2 (u,v)})} }
 \lambda_f\Big(\frac{m d_1}{a^2 e_1^2 (u, v)}\Big)
 \lambda_f\Big(\frac{n d_2}{b^2 e_2^2 (u, v)}\Big).
\end{multline*}
This is precisely the desired formula \eqref{eq:PeterssonOldforms}, after a rearrangement.

\section{Inversion}
\label{section:sieve}
In this section, we show how to deduce \eqref{eq:DeltakN*Formula} from \eqref{eq:PeterssonOldforms}.  We work in greater generality than what is immediately required, which will be useful in Section \ref{section:PeterssonHybrid}.
Suppose that $F$ and $G$ are two arithmetic functions that we write in the form $F(m,N), G(m,N)$ where $N$ is a positive square-free integer, and $m = (m_1, \dots, m_d)$ is a tuple of integers.
We assume there is a relation of the form
\begin{equation}
\label{eq:InitialRelation}
 F(m, N) = \sum_{LM=N} \sum_{a| L^{\infty}} A(a, m_L) G(\alpha(a, m_L) \frac{m}{m_L}, M),
\end{equation}
where: $A$ is some multiplicative arithmetical function, $m_L$ denotes the part of $m$ having common factors with $L$, so $m_L  \mid  L^{\infty}$ and $(m/m_L, L) = 1$, and $\alpha$ is some integer-valued multiplicative function having the property that $\alpha(p^{i_1}, \dots, p^{i_k})  \mid  p^{\infty}$ for all primes $p$.  Furthermore, $a$ is shorthand for some tuple $(a_1, \dots, a_J)$, and the condition $a \mid L^{\infty}$ means that $a_i  \mid  L^{\infty}$ for all $i=1,\dots, J$.

We can derive that 
$F(m,1) = G(m,1)$ for all $m$, by taking $N=1$ in \eqref{eq:InitialRelation}.  The main topic of this section is to prove
\begin{mylemma} 
\label{lemma:inversion}
 For square-free $N$  we have the inversion formula
\begin{equation}
\label{eq:inversion}
 G(m, N) = \sum_{LM=N} \mu(L) \sum_{a | L^{\infty}} A(a, m_L) F(\alpha(a, m_L) \frac{m}{m_L}, M).
\end{equation}
 \end{mylemma}
 Lemma \ref{lemma:inversion} implies Theorem \ref{thm:PeterssonNewforms}, since \eqref{eq:InitialRelation} encompasses \eqref{eq:PeterssonOldforms}, as we now explain.  The tuple $(m_1, \dots m_d)$ in \eqref{eq:InitialRelation} takes the form $(m,n)$ in \eqref{eq:PeterssonOldforms}.
 The tuple $a$ appearing in \eqref{eq:InitialRelation} is of the form $(\ell, d_1, d_2, u, v, a, b, e_1, e_2)$, where note that all these entries divide $L^{\infty}$.  The arithmetical function $A(a,m_L)$ accounts for $\frac{1}{\nu(L)}, \frac{\ell}{\nu(\ell)^2}, \dots, \frac{1}{\nu(uv/(u,v)^2)}$, as well as all the summation conditions in \eqref{eq:PeterssonOldforms}, in which $(m, n)$ may be replaced by $(m_L, n_L)$.  Finally, we have $\alpha(a,(m_L, n_L)) = (\frac{m_L d_1}{a^2 e_1^2 (u,v)}, \frac{n_L d_2}{b^2 e_2^2 (u,v)})$, and $G(x,y,M) = \Delta_M^*(x,y)$.

\begin{proof}
 If $N=1$, then \eqref{eq:inversion} is true, by an easy calculation. 
 
 Now induct on the number of prime factors of $N$.  We replace $N$ by $NP$ with $P$ a prime (whence $(P,N) = 1$), giving
\begin{equation*}
 F(m,NP) = G(m,NP) +  \sum_{\substack{LM=NP \\ L \neq 1}} \sum_{a| L^{\infty}} A(a, m_L) G(\alpha(a, m_L) \frac{m}{m_L}, M).
\end{equation*}
Since $M$ has fewer prime factors than $NP$, we can use the induction hypothesis to give
\begin{multline*}
 F(m,NP) = G(m,NP) +  \Big[ \sum_{\substack{LM=NP \\ L \neq 1}} \sum_{a | L^{\infty}} A(a, m_L) 
 \\
  \times \sum_{CD = M} \mu(C) \sum_{b | C^{\infty}} A(b, m_C) F(\alpha(b, m_C) \frac{\alpha(a, m_L) \frac{m}{m_L}}{m_C}, D) \Big].
\end{multline*}
Here we have used that 
\begin{equation*}
 \Big( \alpha(a, m_L) \frac{m}{m_L} \Big)_C = m_C,
\end{equation*}
which follows from $(L,C) = 1$.
Next we put back $L=1$ and subtract it off again, giving
\begin{multline*}
 F(m,NP) = G(m,NP)  
 - \sum_{CD = NP} \mu(C) \sum_{b| C^{\infty}} A(b, m_C) F(\alpha(b, m_C) \frac{m}{m_C}, D) 
 \\
 + \Big[
 \sum_{\substack{LM=NP}} \sum_{a | L^{\infty}} A(a, m_L) 
  \sum_{CD = M} \mu(C) \sum_{b | C^{\infty}} A(b, m_C) F(\alpha(b, m_C) \frac{\alpha(a, m_L) \frac{m}{m_L}}{m_C}, D)  \Big].
\end{multline*}
We need to show that the term in square brackets equals $F(m,NP)$, since we can then solve for $G(m,NP)$, giving \eqref{eq:inversion}.
We have
\begin{equation*}
 [ \dots ] = \sum_{CDL = NP} \mu(C)
 \sum_{a | L^{\infty}} A(a, m_L)  \sum_{b | C^{\infty}} A(b, m_C)
 F(\alpha(b, m_C) \frac{\alpha(a, m_L) \frac{m}{m_L}}{m_C}, D) .
\end{equation*}
Using multiplicativity of $A$ and $\alpha$, and that $(C,L) = 1$, we get
\begin{equation*}
 [ \dots ] = \sum_{CDL = NP} \mu(C)
 \sum_{c | (LC)^{\infty}} A(c, m_{CL}) 
 F(\alpha(c, m_{CL}) \frac{m}{m_{CL}}, D) .
\end{equation*}
We can write this as
\begin{equation*}
 [ \dots ] = \sum_{D | NP}  \sum_{c | (NP/D)^{\infty}} A(c, m_{NP/D}) 
 F(\alpha(c, m_{NP/D}) \frac{m}{m_{NP/D}}, D) 
 \sum_{CL = NP/D} \mu(C).
\end{equation*}
The inner sum over $C$ gives $D= NP$, which simplifies as $[ \dots ] = F(m, NP)$, as desired.
 \end{proof}

\section{Hybrid formulas}
\label{section:PeterssonHybrid}
We desire a formula that is  intermediate between $\Delta_{N}(m,n)$ and $\Delta^*_N(m,n)$, in order to capture the weights appearing in Theorem \ref{thm:cubicmoment}.  See the discussion surrounding \eqref{eq:fakeNewformFormula} for motivation for this goal.

For square-free $AB$, 
define
\begin{equation}
\label{eq:DeltaTildeDefinition}
\widetilde{\Delta}_{A,B}(m,n) = c_\kappa \sum_{CD=B} \sum_{f \in H_\kappa^*(AD)} \frac{1}{\nu(C)} \frac{1}{\rho_f(C)} \frac{\lambda_f(m) \lambda_f(n)}{\langle f, f\rangle_{AD}}.
\end{equation}
Note that $\widetilde{\Delta}_{A,1}(m,n) = \Delta_{A}^*(m,n)$.  Provided $(mn,B) = 1$, we have $A_f(m,C) = \lambda_f(m)$ for $m$ coprime to $C$, and then by \eqref{eq:coprimeFormulaFromILS} and we get $\widetilde{\Delta}_{1,B}(m,n) = \Delta_{B}(m,n)$.  Hence $\widetilde{\Delta}_{A,B}(m,n)$ interpolates between $\Delta_{AB}(m,n)$ and $\Delta^*_{AB}(m,n)$, provided $(mn,B) = 1$.

Because of our application to the cubic moment problem, we are interested in the case where $qN$ is square-free, and $(mn,q) = 1$.  
In \eqref{eq:DeltakmnformulaWithRho} we substitute $N \rightarrow Nq$, and factor $L = L_N L_q$ with $L_N  \mid  N$ and $L_q  \mid  q$, giving
\begin{multline*}
 \Delta_{Nq}(m,n) = c_\kappa
\sum_{L_N | N}  \frac{1}{\nu(L_N)} 
\sum_{\substack{u|(m,L_N) \\ v | (n,L_N)}} \frac{u v}{(u, v)} \frac{\mu(\frac{u v}{(u, v)^2})}{\nu(\frac{u v}{(u, v)^2})}
\sum_{\substack{a | (\frac{m}{u}, \frac{u}{(u, v)}) \\ b | (\frac{n}{v}, \frac{v}{(u, v)})}}     
\\
\sum_{\substack{L_N L_q M=Nq \\ L_q | q}} 
  \sum_{f \in H_\kappa^*(M)} \frac{\lambda_f(\frac{m}{a^2 (u, v)})
 \lambda_f(\frac{n}{b^2 (u, v)})}{\nu(L_q) \rho_f(L_q) \rho_f(L_N) \langle f, f\rangle_M} 
 .
\end{multline*}
Here we used that $(mn, L_q) = 1$ to simplify the divisiblity conditions.  Next we use \eqref{eq:rhoinverseformula} on $\frac{1}{\rho_f(L_N)}$, and use the Hecke relation again, giving now
\begin{multline*}
\Delta_{Nq}(m,n) = 
\sum_{L_N | N} \sum_{\ell | L_N^{\infty}} \frac{\ell}{\nu(L_N) \nu(\ell)^2} 
\sum_{d_1, d_2 | \ell} c_{\ell}(d_1) c_{\ell}(d_2) 
\sum_{\substack{u|(m,L_N) \\ v | (n,L_N)}} \frac{u v }{(u, v)} \frac{\mu(\frac{u v}{(u, v)^2})}{\nu(\frac{u v}{(u, v)^2})}
\\
\sum_{\substack{a | (\frac{m}{u}, \frac{u}{(u, v)}) \\ b | (\frac{n}{v}, \frac{v}{(u, v)})}}     
\sum_{\substack{e_1 | (d_1, \frac{m}{a^2 (u,v)}) \\ e_2 | (d_2, \frac{n}{b^2 (u,v)})} }
c_\kappa
\sum_{L_q |q} 
  \sum_{f \in H_\kappa^*(\frac{N}{L_N} \frac{q}{L_q})} \frac{\lambda_f(\frac{m d_1}{a^2 e_1^2 (u, v)})
 \lambda_f(\frac{n d_2}{b^2 e_2^2 (u, v)})}{\nu(L_q) \rho_f(L_q)  \langle f, f\rangle_{\frac{Nq}{L_N L_q }}} 
 .
\end{multline*}
With $A= N/L_N$, $B = q$, $C=L_q$, and $D= q/L_q$, we can write the sum over $L_q  \mid  q$ above as $\widetilde{\Delta}_{N/L_N, q}(m', n')$, for $m'$ and $n'$ the obvious integers.  Therefore, this shows
\begin{multline*}
\Delta_{Nq}(m,n) = 
\sum_{L_N | N} \sum_{\ell | L_N^{\infty}} \frac{\ell}{\nu(L_N) \nu(\ell)^2} 
\sum_{d_1, d_2 | \ell} c_{\ell}(d_1) c_{\ell}(d_2) 
\sum_{\substack{u|(m,L_N) \\ v | (n,L_N)}} \frac{uv }{(u, v)} \frac{\mu(\frac{u v}{(u, v)^2})}{\nu(\frac{u v}{(u, v)^2})}
\\
\sum_{\substack{a | (\frac{m}{v}, \frac{u}{(u, v)}) \\ b | (\frac{n}{v}, \frac{v}{(u, v)})}} 
\sum_{\substack{e_1 | (d_1, \frac{m}{a^2 (u,v)}) \\ e_2 | (d_2, \frac{n}{b^2 (u,v)})} }
\widetilde{\Delta}_{\frac{N}{L_N}, q}\Big(
\frac{m d_1}{a^2 e_1^2 (u, v)}, 
 \frac{n d_2}{b^2 e_2^2 (u, v)}\Big).
\end{multline*}

At this point, we replace the condition $L_N  \mid  N$ by $LM = N$ (i.e. rename $L_N$ by just $L$, and then call $M$ the complementary divisor).  This gives
\begin{multline}
\label{eq:DeltaTildeRelationship}
\Delta_{Nq}(m,n) = 
\sum_{L M =  N} \sum_{\ell | L^{\infty}} \frac{\ell}{\nu(L) \nu(\ell)^2} 
\sum_{d_1, d_2 | \ell} c_{\ell}(d_1) c_{\ell}(d_2) 
\sum_{\substack{u|(m,L) \\ v | (n,L)}} \frac{uv }{(u, v)} \frac{\mu(\frac{u v}{(u, v)^2})}{\nu(\frac{u v}{(u, v)^2})}
\\
\sum_{\substack{a | (\frac{m}{u}, \frac{u}{(u, v)}) \\ b | (\frac{n}{v}, \frac{v}{(u, v)})}} 
\sum_{\substack{e_1 | (d_1, \frac{m}{a^2 (u,v)}) \\ e_2 | (d_2, \frac{n}{b^2 (u,v)})} }
\widetilde{\Delta}_{M, q}\Big(
\frac{m d_1}{a^2 e_1^2 (u, v)}, 
 \frac{n d_2}{b^2 e_2^2 (u, v)}\Big).
\end{multline}

Now we fix $q$, and suppress it in the following notation.  Set $F(m,n,N) = \Delta_{Nq}(m,n)$, and likewise $G(m,n,N) = \widetilde{\Delta}_{N, q}(m,n)$.  Then the relation \eqref{eq:DeltaTildeRelationship} is essentially the same formula as \eqref{eq:PeterssonOldforms}, and, more precisely, is encompassed by \eqref{eq:InitialRelation}.  By Lemma \ref{lemma:inversion}, we therefore have (recall $(mn,q) = 1$, and $Nq$ is square-free)
\begin{multline}
\widetilde{\Delta}_{N,q}(m,n) = 
\sum_{L M =  N} \frac{\mu(L)}{\nu(L)} \sum_{\ell | L^{\infty}} \frac{\ell}{ \nu(\ell)^2} 
\sum_{d_1, d_2 | \ell} c_{\ell}(d_1) c_{\ell}(d_2) 
\sum_{\substack{u|(m,L) \\ v | (n,L)}} \frac{u v }{(u, v)} \frac{\mu(\frac{u v}{(u, v)^2})}{\nu(\frac{u v}{(u, v)^2})}
\\
\sum_{\substack{a | (\frac{m}{u}, \frac{u}{(u, v)}) \\ b | (\frac{n}{v}, \frac{v}{(u, v)})}} 
\sum_{\substack{e_1 | (d_1, \frac{m}{a^2 (u,v)}) \\ e_2 | (d_2, \frac{n}{b^2 (u,v)})} }
\Delta_{M q}\Big(
\frac{m d_1}{a^2 e_1^2 (u, v)}, 
 \frac{n d_2}{b^2 e_2^2 (u, v)}\Big).
\end{multline}

\section{Formulas and estimates with Chebyshev coefficients}
\label{section:Chebyshev}

We begin with a combinatorial evaluation of
$c_{j,n}$.  From \eqref{eq:chebyshevCoefficientIntegralFormula}, combined with the formula $U_j(\cos \theta) = \frac{\sin((j+1) \theta)}{\sin \theta}$, we
have
\begin{equation}
 c_{j,n} = \tfrac12 \int_{-\pi}^{\pi} \sin((j+1)\theta) (2 \cos(\theta))^n \tfrac{2}{\pi} \sin \theta d \theta.
\end{equation}
Writing everything in terms of $e^{i\theta}$, we get
\begin{equation*}
 c_{j,n} = \frac{1}{\pi}  \frac{1}{(2i)^2} \int_{-\pi}^{\pi} \Big(\sum_{r=0}^{n} \binom{n}{r} e^{i \theta(n-2r)} ) (e^{i(j+2) \theta} - e^{ij\theta} - e^{-ij \theta} + e^{-i(j+2) \theta}\Big) d \theta.
\end{equation*}
At this point it is clear that $c_{j,n} = 0$ if $j \not \equiv n \pmod{2}$, so assume $j \equiv n \pmod{2}$.  We also see that $c_{j,n} = 0$ if $j > n$.  For $j=n$, it is easy to check that $c_{n,n} = 1$, since the only values of $r$ for which the integral does not vanish are $r=0$, and $r=n$.  If $j \leq n-2$, then we deduce that
\begin{equation}
\label{eq:cjnformula} 
 c_{j,n} = \binom{n}{\frac{n+j}{2}} - \binom{n}{\frac{n+j}{2} + 1},
\end{equation}
which in fact agrees with $c_{n,n}$ too, since $\binom{n}{n+1} = 0$.  One may find this sequence of Chebyshev coefficients in the OEIS \cite{OEIS} which thereby leads to interesting connections.  
For instance, the list (where we omit terms with $j \not \equiv n \pmod{2}$) $$c_{0,0}, c_{1,1}, c_{0,2}, c_{2,2}, c_{1,3}, c_{3,3}, c_{0,4}, c_{2,4}, c_{4,4}, \dots = 1,1,1,1,2,1,2,3,1,5,4,1, 5\dots$$ is the same as Catalan's triangle ordered along diagonals in reverse order.

From \eqref{eq:cjnformula}, we deduce $c_{j,n} \geq 0$ for all $j,n$, and we have
\begin{equation}
\label{eq:cjnsumTelescoping}
 \sum_{j=0}^{n} c_{j,n} = 
 \begin{cases}
   \binom{n}{n/2}, \qquad &n \text{ even}, \\
   \binom{n}{(n+1)/2}, \qquad &n \text{ odd},
 \end{cases}
\end{equation}
since the sum telescopes.  Let $\delta \in \{0,1\}$, $\delta \equiv n \pmod{2}$.
Note that Stirling's formula gives
\begin{equation}
\label{eq:binomialStirling}
 \binom{n}{\frac{n+\delta}{2}} \sim \frac{2^n}{\sqrt{\pi n/2}}. 
\end{equation}

\begin{mylemma}
\label{lemma:ChebyshevBound}
Let $c_{\ell}(d)$ be as defined in Section \ref{section:FourierManipulations}, and suppose $\gamma \geq 0$.  Then
\begin{equation*}
\sum_{d|\ell} c_{\ell}(d) d^{\gamma} \leq \prod_{p^n || \ell} (p^{-\gamma} + p^{\gamma})^{n}.
\end{equation*}
\end{mylemma}
Remark.  For $\gamma=0$, this bound is slightly worse than that implied by
\eqref{eq:cjnsumTelescoping}, in view of \eqref{eq:binomialStirling}.
\begin{proof}
We have 
\begin{equation}
\sum_{d|\ell} c_{\ell}(d) d^{\gamma} = \prod_{p^n || \ell} \sum_{j=0}^{n} p^{\gamma j} c_{j,n},
\end{equation}
so it suffices to show for $x > 0$ that 
\begin{equation}
\label{eq:cjnxupperbound}
\sum_{j=0}^{n} c_{j,n} x^j \leq (x^{-1} + x)^n.
\end{equation}

From \eqref{eq:cjnformula}, we obviously have that
\begin{equation}
\label{eq:cjnupperbound}
 c_{j,n} 
\leq  \binom{n}{\frac{n+j}{2}}.
\end{equation}
From this, we shall deduce \eqref{eq:cjnxupperbound},
as we presently explain.  On the one hand, we have from \eqref{eq:cjnupperbound} that
\begin{equation}
\label{eq:cjnxupperbound2}
\sum_{j=0}^{n} c_{j,n} x^j \leq x^n + \binom{n}{n-1} x^{n-2} + \dots + \binom{n}{\frac{n+\delta}{2}} x^{\delta},
\end{equation}
and on the other hand, we have
\begin{equation}
\label{eq:binomial}
(x^{-1} + x)^n = x^n + \binom{n}{n-1} x^{n-2} + \dots + \binom{n}{\frac{n+\delta}{2}} x^{\delta} + \binom{n}{\frac{n+\delta}{2}-1} x^{\delta-2}
+ \dots + 
x^{-n}.
\end{equation}
In words, the positive powers of $x$ appearing in \eqref{eq:binomial} precisely match the upper bound on \eqref{eq:cjnxupperbound2}.
 \end{proof}

For later use, we require an estimate for the following expression:
\begin{equation}
\label{eq:SLYdefinition}
S(L,Y) := \sum_{\substack{\ell | L^{\infty} \\ \ell \leq Y}} \frac{\ell}{\nu(\ell)^2} (\sum_{d | \ell} c_{\ell}(d) d^{1/2})^2. 
\end{equation}
\begin{mylemma}
\label{lemma:SLYbound}
We have
\begin{equation}
\label{eq:SLYbound}
S(L,Y) 
\ll_{\varepsilon}  Y^{\varepsilon} \tau(L).
\end{equation}
\end{mylemma}
\begin{proof}
Without the restriction $\ell \leq Y$, the estimate in Lemma \ref{lemma:ChebyshevBound} would barely fail to show the sum converges, since
\begin{equation*}
\frac{p(p^{-1/2}+p^{1/2})^2}{(1+p)^2} = 1.
\end{equation*}
However, using Rankin's trick and Lemma \ref{lemma:ChebyshevBound} we obtain
\begin{equation*}
S(L,Y) 
\leq \sum_{\ell | L^{\infty}} \Big(\frac{Y}{\ell}\Big)^{\varepsilon}
= 
Y^{ \varepsilon} \prod_{p|L} (1-p^{-\varepsilon})^{-1} 
\leq C(\varepsilon) Y^{\varepsilon} \tau(L),
\end{equation*}
where we have used the following:
\begin{equation}
\label{eq:miscellaneousproductOfPrimeDivisorsOfL}
 \prod_{p|L} (1-p^{-\varepsilon})^{-1} \leq \Big(\prod_{\substack{p|L \\ p^{\varepsilon} \geq 2}} 2 \Big) \Big(  \prod_{\substack{p|L \\ p^{\varepsilon} < 2}} (1-p^{-\varepsilon})^{-1} \Big) \leq \tau(L) C(\varepsilon),
\end{equation}
where $C(\varepsilon) = \prod_{\substack{p: p^{\varepsilon} < 2}} (1-p^{-\varepsilon})^{-1}$.
 \end{proof}

 We will additionally need the following bound. 
\begin{mylemma}
\label{lemma:cnubound}
 Let $c_{\ell}(d)$ be as defined in Section \ref{section:FourierManipulations}.  Then
 \begin{equation}
 \label{eq:celldsumbound}
 \frac{\sum_{d|\ell} c_{\ell}(d)}{\nu(\ell)} \leq \ell^{-\frac{\log(3/2)}{\log 2}} = \ell^{-0.5849\dots}.
\end{equation}
\end{mylemma}
\begin{proof}
By Lemma \ref{lemma:ChebyshevBound}, we have
 \begin{equation}
 \label{eq:miscellaneousproduct}
 \ell^{\frac{\log(3/2)}{\log 2}} \frac{\sum_{d|\ell} c_{\ell}(d)}{\nu(\ell)} \leq \prod_{p^n || \ell}  \Big( \frac{2}{(p+1)} p^{\frac{ \log(3/2)}{\log 2}} \Big)^n.
 \end{equation}
Let
\begin{equation*}
 f_{\delta}(x) = \frac{2 x^{\delta}}{x+1},
\end{equation*}
where $x \geq 2$ and $0 < \delta < 1$.  If $\delta = \frac{\log(3/2)}{\log{2}}$ is such that $f_{\delta}(x)$ is decreasing for $x \geq 2$, and $f_{\delta}(2) \leq 1$, then this will show that the product on the right hand side of \eqref{eq:miscellaneousproduct} is $\leq 1$, which suffices to prove the desired bound.
It is easy to check with basic calculus that the desired properties hold for $f_{\delta}(x)$.
 \end{proof}
Remark.  The exponent occurring in \eqref{eq:celldsumbound} is mainly controlled by the powers of $2$ dividing $\ell$.  If $\ell = 2^n$, and $n$ is even, then in fact $\sum_{d | \ell} c_{\ell}(d) = \binom{n}{n/2} \asymp 2^n/n^{1/2}$, while $\nu(\ell) = 3^n$, so if $\ell = 2^n$ then \eqref{eq:celldsumbound} is sharp up to the factor $n^{-1/2} = (\log_2{\ell})^{-1/2}$.  If $\ell$ has no small prime divisors  then the exponent can be improved.

\begin{mycoro}
\label{coro:ellsumTailBound}
 Let $\gamma_0 = \frac{\log(3/2)}{\log 2} - \frac12 =0.0849625\dots$, and suppose $\varepsilon > 0$ is small.  Then
 \begin{equation}
  \sum_{\substack{ \ell |L^{\infty} \\ \ell > Y}} \frac{\ell^{1+\varepsilon}}{\nu(\ell)^2} (\sum_{d | \ell} c_{\ell}(d))^2 \ll_{\varepsilon} Y^{-2\gamma_0 + 2\varepsilon} \tau(L).
 \end{equation}
\end{mycoro}
\begin{proof}
By Lemma \ref{lemma:cnubound}, we have
\begin{equation*}
  \frac{\ell^{1+\varepsilon}}{\nu(\ell)^2} (\sum_{d | \ell} c_{\ell}(d))^2 \leq \ell^{\varepsilon - 2 \gamma_0}.
\end{equation*}
Then, we have by Rankin's trick and \eqref{eq:miscellaneousproductOfPrimeDivisorsOfL}, that
\begin{equation*}
 \sum_{\substack{\ell | L^{\infty} \\ \ell > Y}} \ell^{-2\gamma_0+\varepsilon} \leq Y^{-2\gamma_0 + 2\varepsilon} \sum_{\substack{\ell | L^{\infty}}} \ell^{-\varepsilon} = Y^{-2\gamma_0 + 2\varepsilon} \prod_{p | L} (1-p^{-\varepsilon})^{-1} \\ \ll_{\varepsilon} Y^{-2\gamma_0+2\varepsilon} \tau(L). 
\end{equation*}
\end{proof}

\section{Approximate Petersson formulas}
\label{section:PeterssonApproximate}
Our main Petersson formula, \eqref{eq:DeltakN*Formula}, has a technical problem arising from the fact that the sum over $\ell  \mid  L^{\infty}$ is not a finite sum.  
Iwaniec, Luo, and Sarnak encountered a similar difficulty in \cite{ILS}.  The idea is to truncate the sum at some large parameter $Y$, and estimate the tail trivially.  

To this end, we begin with some simple bounds.  
Throughout this section we assume the weight $\kappa$ is fixed, and do not display any $\kappa$-dependence in the error terms.
First, we claim the crude bound
\begin{equation}
\label{eq:DeltaNCrudeBound}
 |\Delta_N(m,n)| \ll_{\kappa} (m,N)^{1/2} (n,N)^{1/2} \tau_3(m) \tau_3(n),
\end{equation}
for fixed weight $\kappa$.
\begin{proof}
We use \eqref{eq:PeterssonWithPhis}.  Using \eqref{eq:fphiFourierCoefficients} and the Deligne bound, we have
\begin{equation*}
 |\lambda_{f_{\phi}}(m)| \leq \sum_{u|(m,L)} u^{1/2} \tau(m/u) \leq (m,L)^{1/2} \tau_3(m) \leq (m,N)^{1/2} \tau_3(m).
\end{equation*}
Therefore, by the fact that $\lambda_{f_{\phi}}(1) = \lambda_f(1) = 1$, we have
\begin{equation*}
 |\Delta_N(m,n)| \leq (m,N)^{1/2} (n,N)^{1/2} \tau_3(m) \tau_3(n) \Delta_N(1,1).
\end{equation*}
One can then apply the Kloosterman sum formula for $\Delta_N(1,1)$ to show (e.g.~see \cite[Corollary 2.3]{ILS})
\begin{equation*}
 \Delta_N(1,1) = 1 + O\Big(\frac{\tau(N)}{N^{3/2}}\Big) \ll 1. 
\end{equation*}
\end{proof}

Now we state an approximate version of Theorem \ref{thm:PeterssonNewforms}.
\begin{mytheo}
\label{thm:PeterssonNewformsTruncatedVersion}
Let $\gamma_0 = \frac{\log(3/2)}{\log 2} - \frac12 =0.0849625\dots$, and suppose $\varepsilon > 0$ is small. We have
  \begin{multline}
 \label{eq:DeltakN*Formula2}
 \Delta_{N}^*(m,n) = \sum_{LM=N}\frac{\mu(L)}{\nu(L)} \sum_{\substack{\ell | L^{\infty} \\ \ell \leq Y}}  \frac{\ell}{\nu(\ell)^2 } \sum_{d_1, d_2 | \ell} c_{\ell}(d_1) c_{\ell}(d_2) 
 \sum_{\substack{u|(m,L) \\ v | (n,L)}} \frac{u v }{(u, v)} \frac{\mu(\frac{u v}{(u, v)^2})}{\nu(\frac{u v}{(u, v)^2})}
 \\
\sum_{\substack{a | (\frac{m}{u}, \frac{u}{(u, v)}) \\ b | (\frac{n}{v}, \frac{v}{(u, v)})}}   
\sum_{\substack{e_1 | (d_1, \frac{m}{a^2 (u,v)}) \\ e_2 | (d_2, \frac{n}{b^2 (u,v)})} }
 \Delta_{M}\Big(
\frac{m d_1}{a^2 e_1^2 (u, v)}, 
 \frac{n d_2}{b^2 e_2^2 (u, v)}\Big)
 + O((mnNY)^{\varepsilon} N Y^{-2\gamma_0}).
\end{multline}
\end{mytheo}
\begin{proof}
 It suffices to bound the tail of the sum over $\ell$, namely the terms with $\ell > Y$.  Using \eqref{eq:DeltaNCrudeBound}, we have that the difference between $\Delta_N^*(m,n)$ and the main term sum on the right hand side of \eqref{eq:DeltakN*Formula2} is
 \begin{multline}
 \label{eq:PeterssonNewformErrorTerm1}
  \ll \sum_{LM=N} \frac{1}{\nu(L)} \sum_{\substack{\ell | L^{\infty} \\ \ell > Y}}  \frac{\ell}{\nu(\ell)^2 } \sum_{d_1, d_2 | \ell} c_{\ell}(d_1) c_{\ell}(d_2) 
 \sum_{\substack{u|(m,L) \\ v | (n,L)}} \frac{u v }{(u, v)} \frac{1}{\nu(\frac{u v}{(u, v)^2})}
 \\
\sum_{\substack{a | (\frac{m}{u}, \frac{u}{(u, v)}) \\ b | (\frac{n}{v}, \frac{v}{(u, v)})}}   
\sum_{\substack{e_1 | (d_1, \frac{m}{a^2 (u,v)}) \\ e_2 | (d_2, \frac{n}{b^2 (u,v)})} }
\Big(M, \frac{m d_1}{a^2 e_1^2 (u, v)}\Big)^{1/2} 
\Big(M, \frac{n d_2}{b^2 e_2^2 (u, v)} \Big)^{1/2}  \tau_3(md_1) \tau_3(nd_2).
 \end{multline}
We use the weak bound $(M,m') \leq M$ (for any integer $m'$), and use $\tau_3(md_1) \ll (m \ell)^{\varepsilon}$, and similarly for $\tau_3(nd_2)$, and trivially estimate the sums over $u$, $v$,  $a$,  $b$, $e_1$, $e_2$ to give that \eqref{eq:PeterssonNewformErrorTerm1} is
\begin{equation*}
 \ll (mn)^{\varepsilon} N \sum_{L | N} \frac{L^{\varepsilon}}{\nu(L)}  \sum_{\substack{\ell | L^{\infty} \\ \ell > Y}}  \frac{\ell^{1+\varepsilon}}{\nu(\ell)^2 } (\sum_{d | \ell} c_{\ell}(d) )^2.
\end{equation*}
The desired bound then follows from Corollary \ref{coro:ellsumTailBound}. 
 \end{proof}

The same method of proof applies verbatim to $\widetilde{\Delta}_{N,q}(m,n)$:
\begin{mytheo}
Suppose $(mnN,q) = 1$.
We have
 \begin{multline}
 \label{eq:DeltaTildePeterssonFormula}
\widetilde{\Delta}_{N,q}(m,n) = 
\sum_{L M =  N} \frac{\mu(L)}{\nu(L)} \sum_{\substack{\ell | L^{\infty} \\ \ell \leq Y}} \frac{\ell}{ \nu(\ell)^2} 
\sum_{d_1, d_2 | \ell} c_{\ell}(d_1) c_{\ell}(d_2) 
\sum_{\substack{u|(m,L) \\ v | (n,L)}} \frac{u v }{(u, v)} \frac{\mu(\frac{u v}{(u, v)^2})}{\nu(\frac{u v}{(u, v)^2})}
\\
\sum_{\substack{a | (\frac{m}{u}, \frac{u}{(u, v)}) \\ b | (\frac{n}{v}, \frac{v}{(u, v)})}} 
\sum_{\substack{e_1 | (d_1, \frac{m}{a^2 (u,v)}) \\ e_2 | (d_2, \frac{n}{b^2 (u,v)})} }
\Delta_{M q}\Big(
\frac{m d_1}{a^2 e_1^2 (u, v)}, 
 \frac{n d_2}{b^2 e_2^2 (u, v)}\Big)
 + O((mnNY)^{\varepsilon} N Y^{-2\gamma_0}).
\end{multline}
\end{mytheo}

In our desired application, we shall take $Y$ to be a very large power of the level, in which case the error term is very small.  For this reason, we made no attempt to optimize the error term.

\section{Initial Structural Steps}
\label{section:Structure}
\subsection{Invariants of the twisted $L$-functions}
\label{section:rootnumber}
We begin by calculating the root number and conductor of $L(s,f \otimes \chi_q)$, which is apparently somewhat difficult to locate in the literature.  Our proof of Theorem \ref{thm:cubicmoment} does not require any formula for the root number of the twisted $L$-function, but it is helpful for interpreting Corollary \ref{coro:subconvexity}.

More generally, suppose that $g$ is a weight $\kappa$ newform of level $N$ with trivial central character.  Also recall the definition of the Atkin-Lehner operators \eqref{eq:Wqdef}.  
Then $g$ is an eigenform for the $W_d$, and we write its eigenvalue as $g\vert_{W_d} = \eta_d(g) g$. Then the sign of the functional equation for $\Lambda(s,g)$ is given by $i^\kappa \eta_N(g)$.  Since the Atkin-Lehner operators
satisfy $g|_{W_{d_1}} |_{W_{d_2}} = g|_{W_{d_1 d_2}}$ for $(d_1, d_2) = 1$, it suffices to consider the eigenvalues of Atkin-Lehner operartors $\eta_{Q}(g)$ where $Q$ is a prime power dividing the conductor of $g$.  

Let $\chi_q$ be a primitive quadratic Dirichlet character of conductor $q=q_oq_e$ with $q_o$ odd and $q_e$ a power of $2$. Explicitly, $\chi_q(n) = \legen{n}{q_o} \chi_{q_e}(n)$ where $\legen{n}{q_o}$ is the Jacobi symbol and $\chi_{q_e}(n)$ is either $1,\chi_4$, or one of the two primitive quadratic characters of conductor $8$. 
Recall we set $\tilde{q} = \operatorname{rad}(q)$ the largest square-free divisor of $q$.  Let $f$ be a newform of square-free level $rq'$, where $(r,q)=1$ and $q' \mid \tilde{q}$.  We also take $q''   \mid q$ to be such that $q'' \mid q'^\infty$ and $(q/q'',q')=1$. Let us write $f \otimes \chi_q = (f \otimes \chi_{q''}) \otimes \chi_{q/q''}$.  We have by \cite[Theorem 4.1]{AL78} and e.g. \cite[Proposition 14.20]{IK} that $f \otimes \chi_{q''}$ and $f \otimes \chi_q$ are newforms of conductors $rq''^2$ and $rq^2$, respectively. 

We have by \cite[(5.5.1)]{De73} that for each $p \mid r$ that \est{ \eta_p(f\otimes \chi_{q''}) = \chi_{q''}(p)\eta_p(f)} where in Deligne's notation $a(V)=1$ by our square-free hypothesis on $r$ and $\dim(V)=2$.  We found the exposition by Pacetti \cite{Pacetti} particularly helpful for these calculations.    
For each $p \mid q''$ we write $P$ for the power of $p$ dividing $q''$.  Now we have 
\est{ \eta_{P^2}(f \otimes \chi_{q''}) = \chi_P(-1),} 
by Atkin-Li Theorem 4.1 \cite{AL78}, and where we have written $\chi_{q''} = \prod_{p | q'} \chi_P$.  Therefore we have shown that 
\est{\eta_{rq''^2}(f \otimes \chi_{q''}) = \eta_r(f) \chi_{q''}(r) \chi_{q''}(-1).}  Now by Section 3 of Li \cite{L75} or Proposition 14.20 of \cite{IK} we have since $\chi_{q/q''}$ is real that 
\est{\eta_{rq^2} (f\otimes \chi_q) & = \chi_{q/q''}(-rq''^2)\eta_{rq''^2}(f \otimes \chi_{q''}) \\ & =  \chi_{q/q''}(-rq''^2) \chi_{q''}(r) \chi_{q''}(-1) \eta_r(f) \\ & =  \chi_q(-r)\eta_r(f).} 
 Note $\eta_r(f)$ is the eigenvalue of the Atkin-Lehner operator $W_r$ on $f$.  In our case, $f$ is of trivial central character and square-free conductor $rq'$.  In this case one can compute for each $p \mid rq'$ that \begin{equation}
 \label{eq:etapfFormula}
 \eta_p(f) = -p^{1/2}\lambda_f(p),
 \end{equation} 
for which see the proof of Theorem 2.1 of \cite{AL78}.  

In summary, if we let $\epsilon_g$ denote the root number of a newform $g$, this shows
\begin{equation}
\label{eq:rootnumber}
\begin{split}
 \epsilon_{f \otimes \chi_q} = & i^\kappa \chi_q(-r) \mu(r)r^{1/2}\lambda_f(r) \\ 
 = & \chi_q(-r) \mu(q') {q'}^{1/2} \lambda_f(q')\epsilon_f,
 \end{split}
 \end{equation}
 where recall $\lambda_f(n)$ is normalized to be bounded by the divisor function of $n$.

Now let 
\begin{equation}
\omega_f := c_\kappa \frac{1}{\nu(\tilde{q}/q') \rho_f(\tilde{q}/q') \langle f,f \rangle_{rq'}}, 
\end{equation}
where $\rho_f$ was defined in \eqref{eq:rhof}, and $\omega_f$ in particular satisfies 
\begin{equation*}
\omega_f = (rq)^{-1 + o(1)},
\end{equation*}
since by \cite[Lemma 2.5]{ILS} \cite{HoffsteinLockhart} \cite{IwaniecSmallEigenvalues} we have
\begin{equation*}
 \langle f, f \rangle_{rq'} = (rq')^{1+o(1)}.
\end{equation*}

Note that with these weights we have 
\es{\label{Deltatildeformula} \sum_{\substack{f \in H^*_\kappa(rq') \\ q' | \tilde{q}}} \omega_f \lambda_f(m)\lambda_f(n) = \widetilde{\Delta}_{r,\tilde{q}}(m,n),}
where recall $\widetilde{\Delta}_{r,q}(m,n)$ was defined in \eqref{eq:DeltaTildeDefinition}.

\subsection{Approximate functional equation}
Recall that our goal is the bound \eqref{eq:cubicmoment}, which we write as
\begin{equation}
\mathcal{M}(r,q) := \sum_{\substack{ f \in H^*_\kappa(rq') \\ q' | \tilde{q}}} \omega_f L(1/2,f \otimes \chi_q)^3 \ll (qr)^{\varepsilon}.
\end{equation}

We have for $\real(s)>1$ that $L(s,f \otimes \chi_q) = \sum_{n \geq 1} \frac{\chi_q(n)\lambda_f(n)}{n^s}$ and
\begin{multline*}
 L(s,f\otimes \chi_q)^2  = \sum_{m \geq 1} \sum_{n \geq 1} \frac{\lambda_f(m)\lambda_f(n)  \chi_q(mn)}{(mn)^s} = \sum_{m \geq 1} \sum_{n \geq 1} \frac{ \chi_q(mn)}{(mn)^s} 
  \sum_{ \substack{d | (m,n) \\ (d,qr)=1} } \lambda_f\Big(\frac{mn}{d^2}\Big)
 \\ = \sum_{\substack{ (d,qr)=1}} \frac{1}{d^{2s}} \sum_{m \geq 1} \sum_{n \geq 1}\frac{\lambda_f(mn)\chi_q(mn)}{(mn)^s} = \sum_{ \substack{ (d,qr)=1} } \sum_{n \geq 1} \frac{ \tau(n)\chi_q(n) \lambda_f(n)}{(d^2n)^s} .
\end{multline*}
Then we have by standard approximate functional equations 
that
\est{L(1/2,f\otimes \chi_q) = (1+\epsilon_{f\otimes \chi_q} ) \sum_{n\geq 1} \frac{\lambda_f(n) \chi_q(n)}{n^{1/2}} V_1\left(\frac{n}{q \sqrt{r}}\right),} and \est{ L(1/2,f \otimes \chi_q)^2 = 2 \sum_{ \substack{(d,qr)=1}}\sum_{m\geq 1} \frac{\lambda_f(m) \tau(m) \chi_q(m)}{d\sqrt{m}}V_2\left(\frac{d^2m}{rq^2}\right),} where $V_1$ and $V_2$ are certain  bounded  smooth functions of rapid decay (see \eqref{eq:Videf}, \eqref{eq:Widef} below for formulas).   
Therefore,
\begin{equation}
\label{eq:Mrqprotoformula}
 \mathcal{M}(r,q) = 2\sum_{\substack{ f \in H^*_\kappa(R) \\ r | R | r\tilde{q}}} \omega_f   (1+\epsilon_{f\otimes \chi_q}) \sum_{ \substack{d,m,n \geq 1 \\ (d,qr)=1} }     \frac{\lambda_f(m) \tau(m) \lambda_f(n) \chi_q(mn) }{d\sqrt{mn}} 
 V_1\left(\frac{n}{q \sqrt{r}}\right)
 V_2\left(\frac{d^2m}{ q^2 r}\right).
\end{equation}
In \eqref{eq:Mrqprotoformula}, we may replace $(1+ \epsilon_{f \otimes \chi_q})$ by $2$, because if $\epsilon_{f \otimes \chi_q} = -1$, then the other factor $L(1/2, f \otimes \chi_q)^2$ vanishes anyway.  Using this and \eqref{Deltatildeformula}, we derive
\begin{equation}
\label{eq:MrqInTermsOfDelta}
 \mathcal{M}(r,q) = 4 \sum_{\substack{ (d,qr)=1} }\frac{1}{d} \sum_{n\geq 1}  \sum_{m\geq 1} \frac{ \tau(m) \chi_q(mn)}{\sqrt{m n}}  V_1\left(\frac{n}{q \sqrt{r}}\right)  V_2\left(\frac{d^2m}{ q^2 r}\right) \widetilde{\Delta}_{r,\tilde{q}}(m,n).
\end{equation}
The contribution from $m \gg r^{1+\eps}q^{2+\varepsilon} d^{-2}$ or $n \gg r^{1/2+\eps} q^{1+\varepsilon}$ is very small by trivial bounds.

\subsection{Exercises with arithmetical functions}
Equation \eqref{eq:MrqInTermsOfDelta} gives
\begin{equation}
\label{eq:MrqFormula}
 \mathcal{M}(r,q) = 4 \sum_{ \substack{ (d,qr)=1}} \frac{1}{d} 
 \mathcal{B}_{r,q}(\alpha, \beta) + O((qr)^{-A}),
\end{equation}
where
\est{
 \mathcal{B}_{r,q}(\alpha, \beta)  
 = 
 \sum_{\substack{m, n \geq 1}} 
 \alpha_m \beta_n \widetilde{\Delta}_{r,\tilde{q}}(m,n),
} 
with
\begin{equation}
 \alpha_m = \frac{\tau(m) \chi_q(m)}{\sqrt{m}} V_2\Big(\frac{d^2 m}{rq^2}\Big), \qquad \beta_n = \frac{\chi_q(n)}{\sqrt{n}} V_1\Big(\frac{n}{q \sqrt{r}} \Big).
\end{equation}
Now we work in more generality than what is required.
\begin{myprop}\label{prop:setup}
 Let $\alpha$ and $\beta$ be two sequences of complex numbers 
 of rapid decay,  
 and let $Y$ be some large power of $qr$.  Then
 \begin{multline} \label{eq:BtoB'}
\mathcal{B}_{r,q}(\alpha,\beta) = 
  \sum_{L R =  r} \frac{\mu(L)}{\nu(L)} \sum_{\substack{\ell | L^{\infty} \\ \ell \leq Y}} \frac{\ell}{ \nu(\ell)^2} 
\sum_{d_1, d_2 | \ell} c_{\ell}(d_1) c_{\ell}(d_2) 
\sum_{\substack{u|L \\ v |L }} \frac{u v }{(u, v)} \frac{\mu(\frac{u v}{(u, v)^2})}{\nu(\frac{u v}{(u, v)^2})}
\\
\sum_{\substack{a |  \frac{u}{(u, v)} \\ b | \frac{v}{(u, v)} }} 
\sum_{\substack{e_1 | d_1 \\ e_2 | d_2} } \mathcal{B}'
+ O(\|\alpha_m m^{\varepsilon} \|_1 \| \beta_n n^{\varepsilon} \|_1 (rq)^{-100}),
\end{multline}
where
\begin{equation}
 \mathcal{B}' = \sum_{ m \geq 1} \sum_{ n \geq 1 } \alpha_{au \frac{e_1}{(e_1,\frac{u}{a(u,v)})}m} \beta_{bv \frac{e_2}{(e_2,\frac{v}{b(u,v)})}n}
 \Delta_{R \tilde{q}}\Big(  \frac{d_1}{e_1} \frac{u}{(u, ae_1(u, v))}m ,  \frac{d_2}{e_2}  \frac{v}{(v,be_2(u, v))}n \Big)
\end{equation}
\end{myprop}

\begin{proof}
Using \eqref{eq:DeltaTildePeterssonFormula}, and pulling the sums over $m$ and $n$ to the inside, we obtain  
\begin{multline*}
\mathcal{B}_{r,q}(\alpha,\beta) = 
\sum_{L R =  r} \frac{\mu(L)}{\nu(L)} \sum_{\substack{\ell | L^{\infty} \\ \ell \leq Y}} \frac{\ell}{ \nu(\ell)^2} 
\sum_{d_1, d_2 | \ell} c_{\ell}(d_1) c_{\ell}(d_2) 
\sum_{\substack{u|L \\ v |L }} \frac{u v }{(u, v)} \frac{\mu(\frac{u v}{(u, v)^2})}{\nu(\frac{u v}{(u, v)^2})}
\\
\sum_{\substack{a |  \frac{u}{(u, v)} \\ b | \frac{v}{(u, v)} }} 
\sum_{\substack{e_1 | d_1 \\ e_2 | d_2} }
\sum_{\substack{ m \equiv 0 \shortmod{ au \frac{e_1}{(e_1,\frac{u}{a(u,v)})}} \\ n \equiv 0 \shortmod{ bv \frac{e_2}{(e_2,\frac{v}{b(u,v)})} } }}  \alpha_m \beta_n \Delta_{R \tilde{q}}\Big(
\frac{m d_1}{a^2 e_1^2 (u, v)}, 
 \frac{n d_2}{b^2 e_2^2 (u, v)}\Big) 
\\ 
 + O\Big(\sum_{m \geq 1} \sum_{n \geq 1} |\alpha_m m^{\varepsilon}|  |\beta_n n^{\varepsilon}|  (rq)^{1+\eps}  Y^{-2\gamma_0 + \varepsilon}\Big),  
 \end{multline*}
where we have used the following elementary observations: 
 we have \est{ \begin{cases} u | m \\ a | \frac{ m}{u} \end{cases} \Leftrightarrow m \equiv 0 \pmod{au} } and for any integers $a,b,x$ we have $ax \equiv 0 \pmod{b}$ if and only if $x \equiv 0 \pmod{\frac{b}{(a,b)}}$ so that 
 \begin{gather*}
  \frac{m}{a^2(u,v)} \equiv 0 \pmod {e_1}  \Leftrightarrow  \frac{m}{au}\frac{u}{a(u,v)} \equiv 0 \pmod {e_1} \\ \Leftrightarrow  \frac{m}{au} \equiv 0 \pmod {\frac{e_1}{(e_1,\frac{u}{a(u,v)})}}  \Leftrightarrow  m \equiv 0 \pmod{ au \frac{e_1}{(e_1,\frac{u}{a(u,v)})} }.
 \end{gather*}

We now make the change of variables \est{ m \mapsto au \frac{e_1}{(e_1,\frac{u}{a(u,v)})}m \,\,\,\,\,\,\, n \mapsto bv \frac{e_2}{(e_2,\frac{v}{b(u,v)})}n,}
which gives the desired formula. \end{proof}

Continuing with our more general set-up, let $\gamma_1, \gamma_2, \delta_1, \delta_2$ be positive integers that divide $L^{\infty}$, and set 
\begin{equation}
 \mathcal{B}'_{\gamma_1, \gamma_2, \delta_1, \delta_2} = 
 \sum_{m,n \geq 1} \alpha_{\gamma_1 m} \beta_{\gamma_2 n} \Delta_{R\tilde{q}}(\delta_1 m, \delta_2 n).
\end{equation}
In our application of interest, we have
\begin{equation}\label{eq:gamma12def}
 \gamma_1 = au \frac{e_1}{(e_1,\frac{u}{a(u,v)})}, \qquad \gamma_2 = bv\frac{e_2}{(e_2,\frac{v}{b(u,v)})},
\end{equation}
and
\begin{equation} \label{eq:delta12def}
 \delta_1 = \frac{d_1}{e_1} \frac{u}{(u, ae_1(u, v))}, \qquad \delta_2 =   \frac{d_2}{e_2}  \frac{v}{(v,be_2(u, v))}.
\end{equation}

We now use the $\alpha$ and $\beta$ specific to our situation.  In anticipation of some future maneuvers, we use a Hecke relation on the divisor function implicit in $\alpha$, namely $\tau(\gamma m) = \sum_{h| (\gamma, m)} \tau(\gamma/h) \tau(m/h) \mu(h)$.  This gives (for an arbitrary  function $f$ such that the sums converge absolutely)
\begin{align*}
\sum_{m \geq 1} \alpha_{\gamma m} f(m) &= \sum_{m \geq 1} \frac{\tau(\gamma m) \chi_q(\gamma m)}{\sqrt{\gamma m}} V_2\Big(\frac{d^2 \gamma m}{q^2 r}\Big) f(m)
\\
&= \frac{\chi_q(\gamma)}{\sqrt{\gamma}} \sum_{h | \gamma} \frac{\tau(\frac{\gamma}{h}) \mu(h) \chi_q( h)}{\sqrt{ h}} \sum_{m \geq 1} \frac{\chi_q(m) \tau(m)}{\sqrt{m}} 
V_2\Big(\frac{d^2 \gamma h m}{q^2 r}\Big) f(hm).
\end{align*}
With this, and by inserting the definition of $\beta$, we have
\begin{equation}\label{eq:BtoB''}
\mathcal{B}'_{\gamma_1, \gamma_2, \delta_1, \delta_2} = 
\frac{\chi_q(\gamma_1 \gamma_2)}{\sqrt{\gamma_1 \gamma_2}} \sum_{\gamma_3 | \gamma_1} \frac{\tau(\frac{\gamma_1}{\gamma_3}) \mu(\gamma_3) \chi_q( \gamma_3)}{\sqrt{ \gamma_3}} 
\mathcal{B}'',
\end{equation}
where
\begin{equation}
\mathcal{B}'' = \sum_{m,n \geq 1}  \frac{\chi_q(mn) \tau(m)}{\sqrt{mn}} 
V_1\left(\frac{  \gamma_2 n}{q \sqrt{r}}\right)  
V_2\Big(\frac{d^2 \gamma_1 \gamma_3 m}{q^2 r}\Big) \Delta_{R\tilde{q}}(\delta_1 \gamma_3 m,\delta_2 n).
\end{equation}

Applying the Petersson formula \eqref{eq:PeterssonFormulaWithKloosterman}, we obtain
\begin{equation*}
 \mathcal{B}'' = \mathcal{D}'' + 2\pi i^{-\kappa} \mathcal{S},
\end{equation*}
where 
\begin{equation}\label{eq:D''}
 \mathcal{D}'' = \sum_{\delta_1 \gamma_3 m = \delta_2 n}  \frac{\chi_q(mn) \tau(m)}{\sqrt{mn}} 
V_1\left(\frac{\gamma_2 n}{q \sqrt{r}}\right)  
V_2\Big(\frac{d^2 \gamma_1 \gamma_3 m}{q^2 r}\Big) ,
\end{equation}
and
\begin{multline}
\label{eq:SmathcalDef}
 \mathcal{S} =  \sum_{m,n \geq 1}  \frac{\chi_q(mn) \tau(m)}{\sqrt{mn}} 
V_1\left(\frac{ \gamma_2 n}{q \sqrt{r}}\right)  
V_2\Big(\frac{d^2 \gamma_1 \gamma_3 m}{q^2 r}\Big) \\ \times
\sum_{c \equiv 0 \shortmod{\tilde{q}R}} \frac{S(Am, Bn;c)}{c} J_{\kappa-1}\Big(\frac{4 \pi \sqrt{ABmn}}{c}\Big),
\end{multline}
with
\begin{equation*}
 A = \delta_1 \gamma_3, \qquad B = \delta_2.
\end{equation*}
According to this, we write $\mathcal{B}' = \mathcal{D}' + \mathcal{B}'_{\mathcal{S}}$, and similarly, $\mathcal{B} = \mathcal{D} + \mathcal{B}_{\mathcal{S}}$. 
It may be helpful to record that $c \equiv 0 \pmod{\tilde{q}R}$, that $(qR,L) = 1$, and that $AB | L^{\infty}$, so that $(AB,qR) = 1$.

The main technical result proved in the rest of the paper is the following
\begin{myprop}
\label{prop:B''bound}
 With $\alpha$ and $\beta$ as above, we have
 \begin{equation*}
  \mathcal{S} \ll \Big(\frac{\sqrt{AB}}{\sqrt{R}} + \frac{r^{3/4}}{\sqrt{q} R} \Big)(qr)^{\varepsilon}.
 \end{equation*}
\end{myprop} 
From Proposition \ref{prop:B''bound}, we deduce bounds on $\mathcal{B}'_{\mathcal{S}}$, then $\mathcal{B}_{\mathcal{S}}$.  We have
\begin{equation*}
\mathcal{B}'_{\mathcal{S}} \ll \frac{1}{\sqrt{\gamma_1 \gamma_2}} \sum_{\gamma_3 | \gamma_1} \frac{1}{\sqrt{\gamma_3}} 
\Big(\frac{\sqrt{\delta_1 \delta_2 \gamma_3}}{\sqrt{R}} + \frac{r^{3/4}}{q^{1/2} R}\Big) (qr)^{\varepsilon} 
\ll 
\Big(\frac{\sqrt{\delta_1 \delta_2}}{\sqrt{R}} + \frac{r^{3/4}}{q^{1/2} R}\Big) \frac{(qr)^{\varepsilon}}{\sqrt{\gamma_1 \gamma_2}}.
\end{equation*}
Therefore, we get the following bound on $\mathcal{B}_{\mathcal{S}}$:
 \begin{multline*} 
\mathcal{B}_{\mathcal{S}} \ll (qr)^{\varepsilon}
  \sum_{L R =  r} \frac{1}{\nu(L)} \sum_{\substack{\ell | L^{\infty} \\ \ell \leq Y}} \frac{\ell}{ \nu(\ell)^2} 
\sum_{d_1, d_2 | \ell} c_{\ell}(d_1) c_{\ell}(d_2) 
\\
\times
\sum_{\substack{u|L \\ v |L }} \frac{u v }{(u, v)} \frac{|\mu(\frac{u v}{(u, v)^2})|}{\nu(\frac{u v}{(u, v)^2})}
\sum_{\substack{a |  \frac{u}{(u, v)} \\ b | \frac{v}{(u, v)} }} 
\sum_{\substack{e_1 | d_1 \\ e_2 | d_2} } 
\Big(\frac{\sqrt{\delta_1 \delta_2}}{\sqrt{R}} + \frac{r^{3/4}}{q^{1/2} R}\Big) 
\frac{1}{\sqrt{\gamma_1 \gamma_2}}.
\end{multline*}

 Note that
\begin{equation*}
 \frac{\delta_1}{\gamma_1} = \frac{d_1}{e_1} \frac{u}{(u,ae_1 (u,v))} \frac{(e_1, \frac{u}{a(u,v)})}{aue_1} = \frac{d_1}{e_1^2 a^2 (u,v)},
\end{equation*}
and so by symmetry
\begin{equation*}
 \frac{\delta_2}{\gamma_2} =  \frac{d_2}{e_2^2 b^2 (u,v)},
\end{equation*}
and thus
\begin{equation*}
 \Big(\frac{\delta_1 \delta_2}{\gamma_1 \gamma_2}\Big)^{1/2} = \frac{(d_1 d_2)^{1/2}}{e_1 e_2 ab (u,v)}.
\end{equation*}
We also use $\frac{1}{\sqrt{\gamma_1 \gamma_2}} \leq \frac{1}{\sqrt{uv}} \leq \frac{1}{(u,v)}$.
With these observations, we have 
\begin{equation*}
\sum_{\substack{a |  \frac{u}{(u, v)} \\ b | \frac{v}{(u, v)} }} 
\sum_{\substack{e_1 | d_1 \\ e_2 | d_2} } 
\Big(\frac{\sqrt{\delta_1 \delta_2}}{\sqrt{R}} + \frac{r^{3/4}}{q^{1/2} R}\Big) \frac{1}{\sqrt{\gamma_1 \gamma_2}}
\ll \frac{(qr)^{\varepsilon}}{(u,v)} 
\Big( \frac{(d_1 d_2)^{1/2}}{R^{1/2}} + \frac{r^{3/4}}{q^{1/2} R}\Big).
\end{equation*}

The inner sum over $u$ and $v$ gives a divisor bound, so now we get
\begin{equation*} 
\mathcal{B}_{\mathcal{S}} \ll (qr)^{\varepsilon}
  \sum_{L R =  r} \frac{1}{\nu(L)} \sum_{\substack{\ell | L^{\infty} \\ \ell \leq Y}} \frac{\ell}{ \nu(\ell)^2} 
\sum_{d_1, d_2 | \ell} c_{\ell}(d_1) c_{\ell}(d_2)  \Big(\frac{(d_1 d_2)^{1/2}}{R^{1/2}} + \frac{r^{3/4}}{q^{1/2} R} \Big)
.
\end{equation*}
By Lemma \ref{lemma:cnubound}, we bound the second part of this sum by
\begin{multline*}
(qr)^{\varepsilon} \frac{r^{3/4}}{q^{1/2}} \sum_{L R =  r} \frac{1}{\nu(L) R} \sum_{\substack{\ell | L^{\infty} \\ \ell \leq Y}} \frac{\ell}{ \nu(\ell)^2} 
\sum_{d_1, d_2 | \ell} c_{\ell}(d_1) c_{\ell}(d_2) 
\\
\ll 
(qr)^{\varepsilon} \frac{r^{3/4}}{q^{1/2}}
\sum_{L R =  r} \frac{1}{\nu(L)R} \sum_{\substack{\ell | L^{\infty} \\ \ell \leq Y}} \ell^{-2 \gamma_0} \ll q^{-1/2+\varepsilon} r^{-1/4+\varepsilon}.
\end{multline*}
Recalling the definition \eqref{eq:SLYdefinition} and using Lemma \ref{lemma:SLYbound} we have
 \begin{equation*}
  \mathcal{B}_{\mathcal{S}} \ll (qr)^{\varepsilon}\Big(\sum_{LR=r} \frac{S(L,Y)}{\nu(L) R^{1/2}} +  r^{-1/4} q^{-1/2}  \Big) 
  \ll (qr)^{\varepsilon} (r^{-1/2} + r^{-1/4} q^{-1/2}).
 \end{equation*}

Finally, we have from \eqref{eq:MrqFormula} that
\begin{equation*}
 \mathcal{M}(r,q) = \mathcal{M}_0(r,q) + O((qr)^{\varepsilon}(  r^{-1/4} q^{-1/2}  + r^{-1/2})),
\end{equation*}
where $\mathcal{M}_0(r,q)$ is the contribution to $\mathcal{M}(r,q)$ from the diagonal term $\mathcal{D}$.  It is easy to see that $\mathcal{M}_0(r,q) \ll (rq)^{\varepsilon}$, following the proof of the bounds on $\mathcal{B}_S$.

We summarize this discussion with
\begin{mycoro}
 Proposition \ref{prop:B''bound} implies Theorem \ref{thm:cubicmoment}.
\end{mycoro}
This is appealing because it lets us reduce the number of variables to consider from this point onward.

\subsection{Diagonal terms}
\label{section:Diagonalterms}
In this section, we evaluate $\mathcal{M}_0(r,q)$ which along with Proposition \ref{prop:B''bound} leads to \eqref{eq:cubicmomentasymptotic}.  

The functions $V_1$ and $V_2$, are given explicitly by
\begin{equation}
\label{eq:Videf}
 V_i(y) = \frac{1}{2\pi i }\int_{(2)} W_i(u_i) y^{-u_i}\,du_i,
\end{equation}
where 
\begin{equation}
\label{eq:Widef}
W_1(u) = (2\pi)^{-u} \frac{\Gamma(u+\frac{\kappa}{2})}{\Gamma(\frac{\kappa}{2})u}, \qquad  W_2(u) = (2\pi)^{-2u} \frac{\Gamma(u+\frac{\kappa}{2})^2}{\Gamma(\frac{\kappa}{2})^2u}.
\end{equation}
Then recalling \eqref{eq:MrqFormula}, \eqref{eq:BtoB'}, \eqref{eq:BtoB''}, and \eqref{eq:D''} 
we have that 
\begin{equation*}
 \mathcal{M}_0(r,q)  =     \int_{(1)} \int_{(1)} W_1(u_1)W_2(u_2) (q^2 r)^{\frac{u_1}{2}+u_2} \zeta(1+2u_2) \zeta(1+u_1 +u_2)^2  F_{r,q,Y}(u_1, u_2) \frac{4 du_1 du_2}{(2\pi i )^2}  ,
\end{equation*}
where 
\begin{multline*}
 F_{r,q,Y}(u_1,u_2)=\sum_{L R =  r} \frac{\mu(L)}{\nu(L)}  \sum_{\substack{\ell | L^{\infty} 
 \ell \leq Y}} \frac{\ell}{ \nu(\ell)^2} 
\sum_{d_1, d_2 | \ell} c_{\ell}(d_1) c_{\ell}(d_2) \\
\sum_{\substack{u|L \\ v |L }} \frac{u v }{(u, v)} \frac{\mu(\frac{u v}{(u, v)^2})}{\nu(\frac{u v}{(u, v)^2})}
\sum_{\substack{a |  \frac{u}{(u, v)} \\ b | \frac{v}{(u, v)} }}  
\sum_{\substack{e_1 | d_1 \\ e_2 | d_2} }
\frac{\chi_q(\gamma_1 \gamma_2)}{\gamma_1^{1/2+u_2} \gamma_2^{1/2+u_1}} 
\sum_{\gamma_3 | \gamma_1} \frac{\tau(\frac{\gamma_1}{\gamma_3}) 
\mu(\gamma_3) \chi_q( \gamma_3)}{\gamma_3^{1/2+u_2}} \\
\frac{\prod_{p|qr} (1-p^{-1-2u_2})}{\zeta(1+u_1+u_2)^{2}} 
\sum_{\delta_1 \gamma_3 m = \delta_2 n}  \frac{\chi_q(mn) \tau(m)}{m^{1/2+u_2} n^{1/2+u_1}} 
,
\end{multline*}
and 
 where recall \eqref{eq:gamma12def} and \eqref{eq:delta12def} for the definitions of $\gamma_1,\gamma_2, \delta_1, \delta_2$, which depend on $a,b,u,v,e_1, e_2, d_1, d_2$.

 Our plan is to shift the contours past the poles.  We claim $F_{r,q,Y}(u_1, u_2)$ is holomorphic in the region $\text{Re}(u_i) = \sigma_i \geq -1/2$, for $i=1,2$, and satisfies the bound
 \begin{equation}
 \label{eq:FrqBound}
 |F_{r,q,Y}(u_1, u_2)| \ll (qr)^{\varepsilon}.
\end{equation}
\begin{proof}
By a simple argument with Euler factors, it is not hard to see that we have the bound
\begin{equation}
\label{eq:lamp}
\zeta(1+u_1+u_2)^{-2}
\sum_{\delta_1 \gamma_3 m = \delta_2 n}  \frac{\chi_q(mn) \tau(m)}{m^{1/2+u_2} n^{1/2+u_1}} \ll (qr)^{\varepsilon},
\end{equation}
and that the left hand side of \eqref{eq:lamp}, and hence $F_{r,q,Y}$, is holomorphic in the desired region.

 Using divisor-type bounds on the inner sums, we have
 \begin{equation*}
  |F_{r,q, Y}(u_1,u_2)| \ll (qr)^{\varepsilon} 
  \sum_{L R =  r} \frac{1}{\nu(L)}  \sum_{\substack{\ell | L^{\infty} 
 \\ 
 \ell \leq Y}} \frac{\ell}{ \nu(\ell)^2} 
\sum_{d_1, d_2 | \ell} c_{\ell}(d_1) c_{\ell}(d_2).
 \end{equation*}
By Lemma \ref{lemma:cnubound}, we have
\begin{equation*}
 \sum_{\substack{\ell | L^{\infty}  \\
 \ell \leq Y}} \frac{\ell}{ \nu(\ell)^2} 
\sum_{d_1, d_2 | \ell} c_{\ell}(d_1) c_{\ell}(d_2) \ll L^{\varepsilon},
\end{equation*}
and hence \eqref{eq:FrqBound} follows.
 \end{proof}
The proof given above, combined with Corollary \ref{coro:ellsumTailBound}, shows that $$F_{r,q,Y} = \lim_{Y \rightarrow \infty} F_{r,q,Y} + O((qr)^{\varepsilon} Y^{-2\gamma_0 + 2\varepsilon}),$$ so for the rest of the calculation of $\mathcal{M}_0(r,q)$ we take $Y = \infty$, and define $F_{r,q} = \lim_{Y \rightarrow \infty} F_{r,q,Y}$.

Rather than attempting to obtain the strongest error term, we take the easiest path that gives some power saving.  We begin by taking $\sigma_1 = 1/2 + \varepsilon$, and $\sigma_2 = 1/2$. Next we shift $u_2$ to the line $\sigma_2 = -1/2$, crossing a double pole at $u_2 = 0$ only.
On the new line, we have
\begin{multline*}
\int_{(\sigma_1)} \int_{(\sigma_2)} |W_1(u_1)W_2(u_2)(q^2 r)^{u_1/2+u_2} \zeta(1+u_1+u_2)^2 \zeta(1+2u_2) \\
 F_{r,q}(u_1,u_2) |
|du_1 du_2 |
\ll (r q^2)^{-1/4 + \varepsilon}.
 \end{multline*}
Some thought shows that 
\begin{multline}
\text{Res}_{u_2 = 0} W_2(u_2) (q^2 r)^{u_2} \zeta(1+u_1 + u_2)^2 \zeta(1+2u_2) F_{r,q}(u_1, u_2) 
\\
= 
F_{r,q}(u_1,0) P_1(\log q^2 r) \zeta^2(1+u_1) + c \zeta'(1+u_1) \zeta(1+u_1) F_{r,q}(u_1,0) \\ + c' \zeta^2(1+u_1) F_{r,q}^{(0,1)}(u_1, 0),
\end{multline} 
where $c,c'$ are constants and $P_1$ is a degree $1$ polynomial.

The residue is now a single integral over $u_1$, and we shift this contour to $\sigma_1 = -1/2 + \varepsilon$.   The new integral is bounded by $(q^2 r)^{-1/4 + \varepsilon}$, again.  The residue at $u_1 =0$ takes the form
\begin{equation}
R_{r,q} := \sum_{\substack{0 \leq i \leq 2 \\ 0 \leq j \leq 1}} P_{i,j}(\log q^2 r) F_{r,q}^{(2-i,1-j)}(0,0),
\end{equation}
where $P_{i,j}$ is a polynomial of degree $\leq i + j$.

Gathering this discussion together, we have shown
\begin{equation*}
\mathcal{M}_0(r,q) = R_{r,q} + O((q^2 r)^{-1/4+\varepsilon}).
\end{equation*}
It would be better to study the main terms in the style of \cite{CFKRS} using shifts, which for the sake of brevity we leave for another occasion.

\subsection{Dyadic subdivisions}
We return to estimating $\mathcal{S}$ defined by \eqref{eq:SmathcalDef}.
Next, we open the divisor function $\tau(m) = \sum_{n_1 n_2 = m} 1$ and apply a dyadic partition of unity to the sums over $n_1, n_2, n=n_3$, and $c$.  This gives
\begin{equation*}
 \mathcal{S}   =   \sum_{N_1, N_2, N_3, C} \frac{1}{(N_1 N_2 N_3)^{1/2} C} \mathcal{S}_{N_1, N_2, N_3, C} +  O((qr)^{-10}),
\end{equation*}
where 
$N_1, N_2, N_3, C$ run over dyadic numbers  
and where
\begin{multline}
\label{eq:Sdef}
 \mathcal{S}_{N_1, N_2, N_3, C}
 =
 \sum_{c \equiv 0 \shortmod{\tilde{q}R}} w_C(c)
 \sum_{n_1, n_2,n_3 \geq 1}   \chi_q(n_1 n_2 n_3) \\
S(An_1 n_2, Bn_3;c) J_{\kappa-1}\Big(\frac{4 \pi \sqrt{ABn_1 n_2 n_3}}{c}\Big)
w_{N_1, N_2, N_3}(n_1, n_2, n_3).
\end{multline}
Here the weight functions $w_C(x)$ and $w_{N_1, N_2, N_3}(x_1, x_2, x_3)$ satisfy
\begin{equation*}
  w_C^{(j)}(x) \ll x^{-j}, \qquad w_{N_1, N_2, N_3}^{(j_1, j_2, j_3)}(x_1, x_2, x_3) \ll x_1^{-j_1} x_2^{-j_2} x_3^{-j_3},
\end{equation*}
and are supported on $x \asymp C$, $x_i \asymp N_i$, $i=1,2,3$.

By the Weil bound, and using $J_{\kappa - 1}(x) \ll x $, the contribution to $\mathcal{S}$ from $c \geq C$ is
\begin{equation*}
 \ll \frac{(qr)^{\varepsilon}}{\sqrt{N_1 N_2 N_3} C} \Big( \frac{\sqrt{AB N_1 N_2 N_3}}{C} \Big) \frac{C^{3/2}}{qR} N_1 N_2 N_3  = \frac{\sqrt{AB} N_1 N_2 N_3}{C^{1/2} q R} (qr)^{\varepsilon}.
\end{equation*}
This is satisfactory for Proposition \ref{prop:B''bound}  for $C \gg \frac{(N_1 N_2 N_3)^2}{q^2 R}$.  Thus we may restrict the variables by
\begin{equation}
\label{eq:Csize}
qR \ll C \ll 
\frac{(N_1 N_2 N_3)^2}{q^2 R},  
\qquad N_1 N_2 \ll \frac{(q^2 r )^{1+\varepsilon}}{d^2 \gamma_1 \gamma_3}, \qquad N_3 \ll \frac{(q r^{1/2})^{1+\varepsilon}}{\gamma_2}.
\end{equation}

Let us also write  \eqref{eq:Sdef} as 
\begin{equation}
\label{eq:S'def}
 \mathcal{S}_{N_1, N_2, N_3, C}  = \sum_{c \equiv 0 \shortmod{\tilde{q}R}} w_C(c) \mathcal{S}'_{N_1, N_2, N_3, c}.
\end{equation}

\subsection{Poisson summation}
Let $[c,q] = \lcm(c,q)$.  We have
\begin{equation}
\label{eq:S'Poisson}
 \mathcal{S}'_{N_1, N_2, N_3, c} = 
 \sum_{m_1, m_2, m_3 \in \mathbb{Z}} G_{A,B}(m_1, m_2, m_3;c) K(m_1, m_2, m_3;c),
\end{equation}
where
\begin{multline}\label{eq:GABdef}
 G_{A,B}(m_1, m_2, m_3;c)  =  \frac{1}{[c,q]^3} \sum_{x_1, x_2, x_3 \shortmod{[c,q]}} \chi_q(x_1 x_2 x_3) \\ S(A x_1 x_2, Bx_3;c)  e\Big(\frac{x_1 m_1 + x_2 m_2 + x_3 m_3}{[c,q]}\Big),
\end{multline}
and
\begin{multline}
\label{eq:Kdef}
 K(m_1, m_2, m_3;c) 
 = \int_{\mathbb{R}^3} J_{\kappa-1}\Big(\frac{4 \pi \sqrt{AB t_1 t_2 t_3}}{c}\Big) e\Big(\frac{-m_1 t_1 - m_2 t_2 - m t_3}{[c,q]}\Big) \\
 w_{N_1, N_2, N_3}(t_1, t_2, t_3)  dt_1 dt_2 dt_3.
\end{multline}
When $A=B=1$ and $q$ is odd and square-free, this is precisely as in \cite{CI} (though the reader should be aware of our slightly different normalization of $G$ by $[c,q]^{-3}$), so this   appears at first glance to be a fairly minor generalization of their work, however the calculations become rather intricate.

\section{Arithmetic Part}
Let $(\epsilon_1,\epsilon_2,\epsilon_3) \in \{\pm 1\}^3$, $\delta = 1,2,4$, or $8$, and write $q_e$ for the even part of $q$.  Let \begin{equation}
\label{eq:ZRqdefinition}
 Z_{\delta,R,q}^{(\epsilon_1,\epsilon_2,\epsilon_3)}(s_1, s_2, s_3, s_4) =  \sum_{\substack {\epsilon_i m_i \geq 1 \\ i=1,2,3}} \sum_{\substack{c \equiv 0 \shortmod{R\tilde{q}} \\ (c,q_e)=\delta}} \frac{cq G_{A,B}(m_1, m_2, m_3;c) e_{AB[c,q]^3/c^2}(-m_1 m_2 m_3)}{|m_1|^{s_1} |m_2|^{s_2} |m_3|^{s_3} (\frac{c}{R\tilde{q}})^{s_4}}.
\end{equation}
One of the key ingredients of the Conrey-Iwaniec method (when $A=B=1$ and $q$ is odd) is that the additive character $e_{AB[c,q]^3/c^2}(-m_1 m_2 m_3)$ nicely combines with $G(m_1, m_2, m_3;c)$, allowing for an efficient decomposition into multiplicative characters.  

To avoid over-burdening the already burdened notation we only give proofs in the case $(\epsilon_1,\epsilon_2,\epsilon_3)= (1,1,1)$ and denote this case simply $Z_{\delta,R,q}$, the other cases being treated similarly.  Note that we have $(AB) \ll (qr)^L$ for some fixed but possibly large $L > 0$ (see Proposition \ref{prop:setup}) so that $(AB)^{\varepsilon} \ll (qr)^{\varepsilon'}$.

The main goal of this section is the following proposition.
 
\begin{myprop}
 \label{prop:Zproperties}
For each choice of $(\epsilon_1,\epsilon_2,\epsilon_3)$ and $\delta$ there is a decomposition $Z_{\delta,R,q} = Z^{(\epsilon_1,\epsilon_2,\epsilon_3)}_{\delta,R,q}(s_1, s_2, s_3, s_4) = Z_0 + Z'$, where $Z_0$ and $Z'$ have the following properties.
Here
$Z_0$ is analytic in $\text{Re}(s_i) \geq 1 + \sigma$ for $i=1,2,3,4$, $\sigma > 0$ and in this region 
it satisfies the bound
\begin{equation}
 \label{eq:Z00boundRightOfCriticalStrip}
Z_{0} \ll_{\sigma, \varepsilon}  \frac{(qr)^\eps}{AB}.
\end{equation}
The function $Z'$ is analytic for $\text{Re}(s_i) \geq \frac12 + \sigma$ for $i=1,2,3,4$, any $\sigma > 0$, and in this region satisfies the bound 
  \begin{equation}
|Z'| \ll_{\sigma, \varepsilon} q^{3/2} (AB)^{1/2} (qr)^{\varepsilon} 
\prod_{j=1}^{4} (1+|s_j|)^{1/4+\varepsilon}.
\end{equation}  
  Moreover, if $s_j = 1/2 + \varepsilon + i(y_j + t)$ for $j=1,2,3$, $s_4 = 1/2 + \varepsilon + i(y_4 -t)$, and $y_j \ll (qr)^{\varepsilon}$ for $j=1,2,3,4$, then we have
  \begin{multline}
  \label{eq:Z'integralbound}
   \int_{|t| \leq T} |Z'(1/2 + \varepsilon + i(y_1 + t), 1/2 + \varepsilon + i(y_2 + t), 1/2 + \varepsilon + i(y_3 + t), 1/2 + \varepsilon + i(y_4 - t))| dt 
   \\
   \ll q^{3/2} (AB)^{1/2} T (qrT)^{\varepsilon}.
  \end{multline}
 \end{myprop} 
 
We begin by reducing the evaluation of $G_{A, B}$ into cases.  First, write $c=c_1c_2$ with $c_2 | (AB)^{\infty}$, and $(c_1,AB)=1$.  As $r=RL$ is square-free and $(r,q)=1$ we have $(qR,AB)=1$, hence $(qR,c_2)=1$.  By a calculation with the Chinese remainder theorem, we have
\begin{multline}
\label{eq:GABchineseremaindertheorem}
G_{A,B}(m_1, m_2, m_3 ;c_1 c_2) = \chi_q(AB) G_{1,1}(m_1, m_2, \overline{AB c_2} m_3;c_1) \\ G_{A,B}(m_1, m_2, \overline{[c_1, q]}^3c_1^2 m_3 ;c_2),
\end{multline}
where in the definition of $G_{A,B}(m_1,m_2,m_3;c_2)$ for $c_2 \mid (AB)^\infty$ we implicitly take $q=1$.
Write $q = q_o q_e$ where $q_o$ is odd and $q_e \in \{1,4,8\}$.  
We further decompose $c_1$ by $c_1 = c_o c_e$ where $c_o$ is odd and $c_e$ is a power of $2$.  Another short calculation with the Chinese remainder theorem shows
\begin{equation}\label{eq:G11crt}
 G_{1,1}(\ell_1, \ell_2, \ell_3 ;c_o c_e) = 
 G_{1,1}(\ell_1, \ell_2,  \overline{[c_e,q_e]}^3c_e^2  \ell_3 ;c_o )
G_{1,1}(\ell_1, \ell_2, \overline{c_o} \ell_3 ;c_e ).
 \end{equation}

 Next we evaluate the three types of $G_{A,B}$ sums in a form most relevant for our further calculations.  The case with modulus $c_o$ was derived by \cite[\S 10]{CI}.  Following the notation found in \cite{CI}, write
\begin{equation*}
c_o = q_o s_o. 
\end{equation*} 
 \begin{mylemma}[Conrey and Iwaniec]
 \label{lemma:GevalConreyIwaniec}
 We have for $q_o,s_o \in \mn$ with $(q_o,2)=1$ that 
\begin{multline}
\label{eq:GevalConreyIwaniec}
c_o q_o e\Big(\frac{-a_1 a_2 a_3}{c_o} \Big) G_{1,1}(a_1, a_2, a_3 ;c_o) = 
 \sum_{\substack{D_1 D_2 hk = q_o \\ h = (q_o,s_o) \\ k = (a_1 a_2 a_3, q_o) \\ (h, a_1 a_2) = 1}} 
 \frac{1}{\varphi(D_2)} \\ \sum_{\psi \shortmod{D_2}} g_{D_1, D_2, h, k, \psi} \psi(a_1 a_2 a_3) \overline{\psi}(s_o) R_k(a_1) R_k(a_2) R_k(a_3),
\end{multline}
where $R_k(n) = S(n,0;k)$ is the Ramanujan sum,
and $g$ is some function satisfying
\begin{equation}
 |g_{D_1, D_2, h, k, \psi}| \ll D_2^{3/2+\varepsilon},
\end{equation}
and where in addition we must have $(a_3, s_o) = 1$, otherwise $G$ vanishes.
In case $\psi$ is principal, then $|g_{D_1, D_2, h, k, \psi}| \ll D_2^{\varepsilon}$.
 \end{mylemma}
Conrey and Iwaniec in fact give a more precise formula that we describe within the proof.

Next we evaluate the case with modulus $c_2 | (AB)^{\infty}$.  
\begin{mylemma}
\label{lemma:GABevalSimplifiedForm}
Suppose $c_2 | (AB)^{\infty}$.  Suppose that $a_1a_2a_3 \neq 0$ and write $a_i = u_i v_i$ where $(u_i, AB) =1$ and $v_i | (AB)^{\infty}$.  
Then
\begin{multline}
\label{eq:GABevalSimplifiedForm}
c_2 e\Big(\frac{-a_1 a_2 a_3}{c_2 AB}\Big) G_{A,B}(a_1, a_2, a_3; c_2) 
=  
\delta((A,c_2) | v_1) \delta((A,c_2) | v_2) \delta((B,c_2) | v_3)
\\
\times
\sum_{\substack{g_1 g_2 | \frac{c_2}{(A,c_2)} \\ g_1 = (\frac{v_1}{(A,c_2)}, \frac{c_2}{(A,c_2)}) \\ g_2 = (\frac{v_2}{(A,c_2)}, \frac{c_2}{g_1 (A,c_2)})}}
 \sum_{D | \frac{c_2 AB}{(c_2, A)^2 (c_2, B)}} \frac{1}{\varphi(D)}
 \sum_{\eta \shortmod{D}} \gamma_\eta \eta(u_1 u_2 u_3)
,
\end{multline}
were $\gamma_\eta = \gamma_{v_1, v_2, g_1, g_2, c_2, A, B, D, \eta}$ is some function satisfying the bound
\begin{equation}
\label{eq:gammabound}
 |\gamma_{v_1, v_2, g_1, g_2, c_2, A, B, D, \eta}| 
 \ll  (A,c_2)(B,c_2)  D^{1/2}.
\end{equation}
In case $\eta$ is principal then with $A = (A,c_2) A'$ and $B = (B,c_2) B'$, we have
\begin{equation}
\label{eq:gammaboundPrincipalCharacterCase}
 \frac{|\gamma_\eta|}{D} \ll (qr)^{\varepsilon} (A,c_2)(B,c_2) \frac{(\frac{v_1 v_2 v_3}{(A,c_2)^2 (B,c_2)}, A'B')}{A' B'}.
\end{equation}
\end{mylemma} 
Again the point is that we get a short linear combination of multiplicative functions.

Finally, we consider the case of $c_{e}$.  For this, we have
\begin{mylemma}
\label{lemma:GevalPowerof2}
Suppose $c_{e}$ is a power of $2$.  Suppose $a_1a_2a_3\neq 0$ and write each $a_i = e_i f_i$ where $e_i$ is a power of $2$, and $f_i$ is odd. Then
\begin{equation}\label{eq:GABevalEven}
q_e c_e e\Big(\frac{-a_1 a_2 a_3}{[c_e, q_e]^3/c_e^2}\Big) G_{1,1}(a_1, a_2, a_3; c_e)
=\sum_{\Delta | 64} \frac{1}{\varphi(\Delta)} \sum_{\chi \shortmod{\Delta}} \mathfrak{g}_{\chi} \chi(f_1f_2f_3), 
\end{equation}
where 
$\mathfrak{g}_\chi =  \mathfrak{g}_{e_1, e_2, e_3, q_e, c_e, \chi,\Delta}$ is bounded by an absolute constant.
\end{mylemma}
As in the previous two cases, we have a much more precise formula for $G_{1,1}$, which we shall describe within the proof.

\begin{proof}[Proof of Lemma \ref{lemma:GevalConreyIwaniec}]
First we note that our $G_{1,1}$ is scaled differently from $G$ defined by \cite{CI}, precisely $G_{1,1}(a_1, a_2, a_3 ;c_o) = c_o^{-3} G(a_1, a_2, a_3;c_o)$, as in \cite[(8.2)]{CI}.  In the notation of \cite{CI}, make the definitions $c_o = q_o s_o$, $h = (q_o, s_o)$, $k = (a_1 a_2 a_3, q_o)$, $D = \frac{q_o}{hk}$.  The sum $G_{1,1}$ vanishes unless $(h, a_1 a_2 ) = 1$ and $(s_0,  a_3 ) = 1$, in which case 
by \cite[Lemma 10.2]{CI}, we have
\begin{equation}
\label{eq:G11CIformula}
 G_{1,1}(a_1, a_2,  a_3 ;c_o) = e\Big(\frac{a_1 a_2 a_3}{c_o}\Big)  \frac{ h \chi_{kD}(-1)}{c_o q_o \phi(k)} 
 R_k(a_1) R_k(a_2) R_k(a_3)  H(\overline{s_o hk} a_1 a_2 a_3;D).
\end{equation}
We do not need the exact formula for $H$, but rather the 
fact that it essentially depends on the variables as a block, and the
decomposition into character sums.  
Specifically, Conrey and Iwaniec \cite[(11.7), (11.9)]{CI} showed
\begin{equation}\label{eq:hdef}
 H(w;D) = \sum_{D_1 D_2 = D} 
 \mu(D_1) \chi_{D_1}(-1) H^*(\overline{D_1} w;D_2),
\end{equation}
and
\begin{equation}
 H^*(w;D_2) = \frac{1}{\varphi(D_2)} \sum_{\psi \shortmod{D_2}} \tau(\overline{\psi}) g(\chi_{D_2}, \psi) \psi(w).
\end{equation}
The crucial fact about $g(\chi_{D_2},\psi)$ is that
\begin{equation}
\label{eq:gchipsiBound}
|g(\chi_{D_2}, \psi)| \ll D_2^{1+\varepsilon},
\end{equation}
which requires the Riemann Hypothesis for varieties, i.e., Deligne's bound.

From here it is a matter of bookkeeping to derive \eqref{eq:GevalConreyIwaniec}.

In case $\psi = \psi_0$ is the principal character, then $|g(\chi_{D_2}, \psi_0)| \leq d(D_2)$ (the divisor function) and $\tau(\psi_0) = \mu(D_2)$.
Indeed, one may show that if $\psi = \psi_0$ is the principal character modulo an odd prime $p$, then $g(\chi_p, \psi_0) = 2$ if $p \equiv 1 \pmod{4}$, and $=0$ if $p \equiv 3 \pmod{4}$.  Furthermore, $g(\chi_q, \psi_0)$ is multiplicative in $q$.\footnote{This corrects a claimed formula for $g(\chi_q, \psi_0)$ of \cite[p.1212]{CI}.}
 \end{proof}

\begin{proof}[Proof of Lemma \ref{lemma:GABevalSimplifiedForm}]
We will evaluate $G_{A,B}$ in precise terms.  We will not use the assumption $a_1 a_2 a_3 = 0$ until indicated later in the proof.
Since $c_2 | (AB)^{\infty}$ and $(q,AB) = 1$, the quadratic character is not present in the sum, and specifically we have
\begin{equation*}
 G_{A,B}(a_1, a_2,  a_3;c_2)  =  \frac{1}{c_2^3} \thinspace \sumstar_{u \shortmod{c_2}} \sum_{x_1, x_2, x_3 \shortmod{c_2}}   e\Big(\frac{A x_1 x_2 \overline{u} + B x_3 u + x_1 a_1 + x_2 a_2 + x_3  a_3}{c_2}\Big).
\end{equation*}
Summing over $x_1$, we detect the congruence $A x_2 \overline{u} \equiv - a_1 \pmod{c_2}$, while the sum over $x_3$ detects $B u \equiv -  a_3 \pmod{c_2}$.  Therefore,
\begin{equation*}
G_{A,B}(a_1, a_2, a_3;c_2) = \frac{1}{c_2}
\thinspace
\sumstar_{B u \equiv - a_3 \shortmod{c_2}} \sum_{A x_2 \overline{u} \equiv -a_1 \shortmod{c_2}}   e\Big(\frac{ x_2 a_2}{c_2}\Big).
\end{equation*}
Note that $Bu \equiv -a_3 \pmod{c_2}$ 
and 
$Ax_2\overline{u} \equiv -a_1 \pmod{c_2}$
are solvable if and only if
\begin{equation}
\label{eq:Bc2Ac2coprimalityConditions}
 (B,c_2) = (a_3, c_2), \qquad \text{ and } \qquad (A,c_2) | a_1.
\end{equation}
By symmetry, we expect that in addition that we will require $(A,c_2) | a_2$, and indeed we will recover this condition later in the analysis.  From now on, we assume the conditions \eqref{eq:Bc2Ac2coprimalityConditions} hold, otherwise the sum is $0$. 

Next we make the definitions
\begin{gather}
\label{eq:A'B'def}
 A = (A,c_2) A', \qquad B = (B, c_2) B', \qquad c_2 = (B,c_2) c_2', \qquad  a_3 = (a_3, c_2)  \widetilde{a_3}.
\end{gather}

Now the congruence $Bu \equiv -a_3 \pmod{c_2}$ is equivalent to
$u \equiv -\overline{B' }\widetilde{a_3} \pmod{c_2'}$.  
Next write $x_2 = -\overline{A'} \frac{a_1}{(A,c_2)} u + \frac{c_2}{(A,c_2)}t$ with $t$ running mod $(A,c_2)$.  Inserting this into the exponential sum, we obtain
\begin{equation*}
G_{A,B}(a_1, a_2, a_3;c_2) =  \frac{1}{c_2}\sumstar_{\substack{u \shortmod{c_2} \\ u \equiv - \overline{B'} \widetilde{a_3} \shortmod{c_2'}}} \sum_{t \shortmod{(A,c_2)}}   e\Big(\frac{ a_2(- \overline{A'}  \frac{a_1}{(A,c_2)} u + \frac{c_2}{(A,c_2)} t)}{c_2}\Big).
\end{equation*}
The sum over $t$ vanishes unless $(A,c_2) | a_2$, in which case we obtain
\begin{equation*}
 G_{A,B}(a_1, a_2, a_3;c_2) = \frac{(A,c_2)}{c_2} \sumstar_{\substack{u \shortmod{c_2} \\ u \equiv - \overline{B' }\widetilde{a_3} \shortmod{c_2'}}} e\Big(\frac{ -\overline{A'}  \frac{a_1}{(A,c_2)} \frac{a_2}{(A,c_2)} u }{\frac{c_2}{(A,c_2)}}\Big).
\end{equation*}

To proceed further, we make some additional definitions, namely
\begin{gather*}
g_1 = \Big(\frac{a_1}{(A,c_2)}, \quad  \frac{c_2}{(A,c_2)}\Big), \quad a_1 = g_1 (A,c_2) \widetilde{a_1}, \quad g_2 = \Big(\frac{a_2}{(A,c_2)}, \frac{c_2}{g_1 (A,c_2)}\Big),
\\
 a_2 = g_2 (A,c_2) \widetilde{a_2},
 \qquad 
 c_2'' = \frac{c_2}{(A,c_2) g_1 g_2}.
\end{gather*}
Thus $(\widetilde{a_1}, \frac{c_2}{g_1 (A,c_2)}) = 1$, and $ (A' \widetilde{a_1} \widetilde{a_2}, c_2'') = 1$, and with this notation the formula becomes
\begin{equation*}
 G_{A,B}(a_1, a_2, a_3;c_2) = \frac{(A,c_2)}{c_2} \sumstar_{\substack{u \shortmod{c_2} \\ u \equiv - \overline{B' } \widetilde{a_3} \shortmod{c_2'}}} e\Big(\frac{ -\overline{A'} \widetilde{a_1} \widetilde{a_2} u }{c_2''}\Big).
\end{equation*}

The tricky part in the analysis is that there is no apparent divisibility relationship between $c_2' = \frac{c_2}{(B,c_2)}$ and $c_2'' = \frac{c_2}{(A,c_2)g_1 g_2}$, and so it is necessary to proceed by cases.  Although it is not globally true that $c_2' | c_2''$, or vice versa, we may factor the moduli corresponding to which prime power of $c_2'$ or $c_2''$ is larger, which motivates the forthcoming factorization.
For $p$ a prime and $n$ a nonzero integer, define $\nu_p(n) = d$ if $p^d || n$.
Then we set  $c_2 = c_z c_f c_g$ where
\begin{equation*}
c_z =  \prod_{\substack{p^{\beta} || c_2 \\ 1 \leq \nu_p(c_2') < \nu_p(c_2'')}} p^{\beta},
\qquad 
 c_f = \prod_{\substack{p^{\beta} || c_2 \\ \nu_p(c_2'') \leq \nu_p(c_2')}} p^{\beta}, \qquad c_g = \prod_{\substack{p^{\beta} || c_2 \\ \nu_p(c_2') = 0 \\  \nu_p(c_2'') \geq 1}} p^{\beta}.
\end{equation*}
According to this factorization, we also write $c_2' = c_z' c_f' c_g'$ and $c_2'' = c_z'' c_f'' c_g''$ where $c_{*}' = (c_*, c_2')$ and $c_{*}'' = (c_*, c_2'')$ with $* = z,f,g$.  Note from the definitions that $c_z,c_f,c_g$ are all pairwise relatively prime, and also that $c_f'' | c_f'$, and $c_g'=1$.

Using the Chinese remainder theorem, we  factor $G_{A,B}$ as
\begin{multline*}
 G_{A,B}(a_1, a_2, a_3;c_2) = \frac{(A,c_2)}{c_2}
 \sumstar_{\substack{u_e \shortmod{c_z} \\ u_e \equiv - \overline{B' }\widetilde{a_3} \shortmod{c_z'}}} e\Big(\frac{ -\overline{A'} \widetilde{a_1} \widetilde{a_2} u_e (c_f c_g) \overline{c_fc_g} }{c_z''c_f'' c_g'' }\Big) 
 \\
 \sumstar_{\substack{u_f \shortmod{c_f} \\ u_f \equiv - \overline{B' } \widetilde{a_3} \shortmod{c_f'}}} e\Big(\frac{ -\overline{A'} \widetilde{a_1} \widetilde{a_2} u_f (c_z c_g) \overline{c_z c_g} }{c_z''c_f'' c_g'' }\Big) \\
 \sumstar_{\substack{u_g \shortmod{c_g} \\ u_g \equiv - \overline{B' } \widetilde{a_3} \shortmod{c_g'}}} e\Big(\frac{ -\overline{A'} \widetilde{a_1} \widetilde{a_2} u_g (c_z c_f) \overline{c_z c_f} }{c_z''c_f'' c_g'' }\Big).
\end{multline*}
Let us examine each of these three sums in turn.  We begin by writing the sum over $u_e$ more suggestively as follows:
\begin{equation*}
 \sumstar_{\substack{u \shortmod{c_z} \\ u \equiv - \overline{B' } \widetilde{a_3} \shortmod{c_z'}}} e\Big(\frac{ -\overline{A'} \widetilde{a_1} \widetilde{a_2} u (\frac{c_f c_g}{c_f'' c_g''}) \overline{c_f c_g} }{ c_z''}\Big).
\end{equation*}
For each prime $p$ dividing $c_z$, we have $p|c_z'$, and $p|\frac{c_z''}{c_z'}$.  Therefore, we may write $u = -\overline{B' } \widetilde{a_3} + c_z' t$, where $t$ runs over \emph{all} residue classes modulo $c_z''/c_z'$.  But then the sum over $t$ vanishes, because the factor in the numerator is relatively prime to the denominator.  Thus we obtain that $G$ vanishes unless $c_z=1$.  We henceforth make this assumption in the next computations of the $c_f$ and $c_g$ moduli sums.  

For the sum over $u_f \pmod{c_f}$, since $c_f'' | c_f'$, the congruence uniquely determines $u_f$ modulo $c_f''$, so we get, using $(c_g/c_g'') \overline{c_g} = \overline{c_g''}$:
\begin{equation*}
 \sumstar_{\substack{u \shortmod{c_f} \\ u \equiv - \overline{B'}\widetilde{a_3} \shortmod{c_f'}}} e\Big(\frac{ -\overline{A'} \widetilde{a_1} \widetilde{a_2} u c_g \overline{c_g} }{c_f'' c_g'' }\Big) = e\Big(\frac{ \overline{A' B' } \widetilde{a_1} \widetilde{a_2} \widetilde{a_3}   \overline{c_g''} }{c_f'' }\Big) \sumstar_{\substack{u \shortmod{c_f} \\ u \equiv - \overline{B' } \widetilde{a_3} \shortmod{c_f'}}} 1.
\end{equation*}
We also have
\begin{equation*}
 \sumstar_{\substack{u \shortmod{c_f} \\ u \equiv - \overline{B' } \widetilde{a_3} \shortmod{c_f'}}} 1 = \frac{\varphi(c_f)}{\varphi(c_f')},
\end{equation*}
as can be checked as follows.  Firstly, we see that both sides of the purported identity are multiplicative, so it suffices to check this on prime powers.  If $c_f' = 1$, then the identity follows easily.  If $c_f' = p^{\beta'}$, and $c_f = p^{\beta}$, and $\beta' \geq 1$, then the left hand side is $p^{\beta-\beta'} = \frac{\varphi(p^{\beta})}{\varphi(p^{\beta'})}$, as desired.

Finally, we examine the sum modulo $c_g$.  We have $c_g'=1$ (directly from the definition, as remarked earlier), so the congruence condition is vacuous.  The sum then simplifies as $R_{c_g}(\frac{c_g}{c_g''})$, using that $c_f$ and everything else in the numerator in the exponential is relatively prime to $c_g$.  

Putting all these calculations together, we have shown that if $c_2=c_fc_g$, $(B,c_2)=(a_3,c_2)$, and $(A,c_2)\mid (a_1,a_2)$, then 
\begin{equation}
\label{eq:GABformulaWithTildeVariables}
G_{A,B}(a_1, a_2, a_3;c_2) = \frac{(A,c_2)}{c_2} R_{c_g}\Big(\frac{c_g}{c_g''}\Big) \frac{\varphi(c_f)}{\varphi(c_f')} e\Big(\frac{ \overline{A' B' } \widetilde{a_1}\widetilde{a_2}\widetilde{a_3}   \overline{ c_g''} }{c_f'' }\Big), 
\end{equation}
and otherwise $G_{A,B}$ vanishes.  So far the assumption $a_1 a_2 a_3 \neq 0$ was not used.

To estimate this expression for $G$, we have  $c_f' = (c_2', c_f) =  (\frac{c_2}{(B,c_2)},\frac{c_2}{c_g})$ and $c_g \mid (B,c_2)$, so in fact $c_f'=\frac{c_2}{(B,c_2)}=c_2'$.  Then
\begin{equation}
\label{eq:RcgBound}
\Big|R_{c_g}\left( \frac{c_g}{c_g''}\right) \Big| \frac{\varphi(c_f)}{\varphi(c_f')}\leq \varphi(c_g)\frac{\varphi(c_f)}{\varphi(c_f')} = \frac{\varphi(c_2)}{\varphi(c_2/(c_2,B))} \leq (B, c_2). 
\end{equation}

Our goal now is to use Dirichlet characters to decompose $$e_{c_2}(-a_1 a_2 a_3) G_{A,B}(a_1, a_2, a_3;c_2).$$  Switching to the new notation used in \eqref{eq:GABformulaWithTildeVariables}, we have
\begin{equation*}
e\Big(\frac{-a_1 a_2 a_3}{ABc_2}\Big) = e\Big(\frac{-\widetilde{a_1} \widetilde{a_2} \widetilde{a_3}}{A'B'c_f'' c_g''}\Big),
\end{equation*}
and by reciprocity, we have
\begin{equation*}
c_2 e\Big(\frac{-a_1 a_2 a_3}{ABc_2}\Big) G_{A,B}(a_1, a_2, a_3;c_2)
=
(A,c_2) R_{c_g}\Big(\frac{c_g}{c_g''}\Big) \frac{\varphi(c_f)}{\varphi(c_f')} e\Big(\frac{ - \widetilde{a_1}\widetilde{a_2}\widetilde{a_3}   \overline{ c_f''} }{c_g'' A' B'}\Big).
\end{equation*}
Let
\begin{equation*}
g_3 = (\widetilde{a_1}\widetilde{a_2}\widetilde{a_3}, c_g'' A' B').
\end{equation*}
Then we have
\begin{multline*}
c_2 e\Big(\frac{-a_1 a_2 a_3}{ABc_2}\Big) G_{A,B}(a_1, a_2, a_3 ; c_2)
\\
=
(A,c_2) R_{c_g}\Big(\frac{c_g}{c_g''}\Big) \frac{\varphi(c_f)}{\varphi(c_f')} 
\frac{1}{ \varphi\Big(\frac{c_g'' A' B'}{g_3}\Big)}\sum_{\eta \shortmod{\frac{c_g'' A' B'}{g_3}}} \tau(\overline{\eta})
\eta( - \widetilde{a_1}\widetilde{a_2}\widetilde{a_3} \overline{ c_f''}/g_3 ).
\end{multline*}

Finally, we argue that this expression is of the desired form for Lemma \ref{lemma:GABevalSimplifiedForm}.  Recall we write $a_i = u_i v_i$ where $(u_i, AB) =1$ and $v_i |(AB)^{\infty}$.  As originally written, the $g_i$ depend on the $a_i$, but in fact they only depend on the $v_i$ since the $g_i$ are divisors of $c_2$, and $c_2 | (AB)^{\infty}$.  By writing the dependence of the $g_i$ on the $v_i$ explicitly as summation conditions, we see the presence of the first sum in \eqref{eq:GABevalSimplifiedForm}.  A careful scrutiny of the changes of variables throughout the proof shows that the variables $c_g, c_g'', c_f, c_f', A', B', g_3$ are functions of $c_2$, the $v_i$ and $g_1, g_2$, and are independent of the $u_i$.  We may also extract from $\widetilde{a_1} \widetilde{a_2} \widetilde{a_3}$ the factor $u_1 u_2 u_3$.  We obtain the bound on $\gamma_\eta$ by \eqref{eq:RcgBound}, and using the standard bound on the Gauss sum.  We note that
\begin{equation*}
\frac{c_g'' A' B'}{g_3}  | c_2'' A' B', \quad \text{and} \quad c_2'' A' B' | \frac{c_2 AB}{(A,c_2)^2 (B,c_2)},
\end{equation*}
which gives the divisibility condition on $D$.  

The only remaining statement to prove is \eqref{eq:gammaboundPrincipalCharacterCase}.  In this case, the Gauss sum is bounded by $1$, and by \eqref{eq:RcgBound} we have
\begin{equation*}
 (A,c_2) \Big| R_{c_g}\Big(\frac{c_g}{c_g''}\Big) \Big| \frac{\varphi(c_f)}{\varphi(c_f')} 
\frac{1}{ \varphi\Big(\frac{c_g'' A' B'}{g_3}\Big)} \ll (qr)^{\varepsilon} (A,c_2) (B,c_2) \frac{g_3}{c_g'' A' B'}.
\end{equation*}
By tracing back the definitions, we see that \begin{equation*} g_3 = (\frac{v_1 v_2 v_3}{g_1 g_2 (A,c_2)^2 (B,c_2)}, c_g'' A' B') \leq c_g'' (\frac{v_1 v_2 v_3}{(A,c_2)^2 (B,c_2)},   A'B'), \end{equation*} which implies the bound \eqref{eq:gammaboundPrincipalCharacterCase}.
 \end{proof}

\begin{proof}[Proof of Lemma \ref{lemma:GevalPowerof2}]
Now we evaluate $G_{1,1}(a_1,a_2, a_3;c_e)$.  To do so we break into cases.  First assume that $q_e \nmid c_e$.  Under the condition $q_e \nmid c_e$ there are only finitely many possibilities for $q_e,c_e, a_1,a_2,a_3$.  A brute force computation with Sage \cite{SAGE} then shows that if $c_e=1,2$ then
\es{\label{eq:qenmidce1}G_{1,1}(a_1,a_2,a_3;c_e) = \frac{1}{q_e^2} \chi_{q_e}(-1) \tau(\chi_{q_e}) \chi_{q_e}(a_1a_2a_3),} 
and if $q_e=8$ and $c_e=4$ then 
\es{\label{eq:qenmidce2}G_{1,1}(a_1,a_2,a_3;c_e) = \frac{1}{32} \chi_{q_e}(-1) \tau(\chi_{q_e}\chi_4) \chi_{q_e}\chi_4(a_1a_2a_3).}
Now we assume that $q_e \mid c_e$.  We write $c_e=q_e s_e$ where $s_e $ is a power of $2$.  Following the same steps as \cite[Section 10]{CI} we have that if $q_e \mid c_e$ then 
\begin{equation}\label{eq:G11evenprep}
G_{1,1}(a_1,a_2,a_3;c_e) =  \frac{\delta_{(a_3,s_e)=1} \chi_{q_e}(-1)}{q_e^2 s_e} e\left(\frac{a_1 a_2 a_3  }{c_e}\right) H_{s_e}(a_1, a_2,a_3,q_e), 
\end{equation}
where $H_s$ is defined in \cite[(10.2)]{CI}.  Assume now both $q_e\mid c_e$ and $q_e \mid s_e$ (so that in fact $q_e^2 \mid c_e$).  Following the proof of Lemma 10.1 of \cite{CI} we find in this case that \es{H_{s_e}(a_1,a_2,a_3,q_e) = \chi_{q_e}(a_1 a_2 a_3)^2\tau(\chi_{q_e})^2.}  Having dealt with this case, we may now assume that $q_e\mid c_e$ and $q_e \nmid s_e$.  Now there are only finitely many choices for  $q_e,c_e, a_1,a_2,a_3$ which permits us to conclude the following lemma by another 
Sage computation.
\begin{mylemma}\label{lem:CI10.1even} Suppose that $q_e \mid c_e$ and let $s_e = c_e/q_e$.   
If $s_e=1$ we have \es{\label{eq:Hs1}H_{s_e}(a_1,a_2,a_3,q_e) = 
\begin{cases} 
\frac{1}{2} R_{q_e}(a_1)R_{q_e}(a_2)R_{q_e}(a_3) & \text{ if } q_e=4 
\\ 
\frac{1}{4} R_{q_e}(a_1)R_{q_e}(a_2)R_{q_e}(a_3)  & 
\text{ if } q_e = 8 \text{ and } 4| a_1, a_2, a_3, 
\\
 16i \chi_4\left(\frac{a_1 a_2 a_3 }{8}\right) & \text{ if } 
 q_e=8 \text{ and } 2 || a_1, a_2,a_3, 
 \\ 
 0 & \text{ otherwise}. 
\end{cases} } 
If $2 \mid s_e$ then 
\es{\label{eq:Hs2} H_s(a_1,a_2,a_3,q_e) = \begin{cases} \chi_{q_e}(a_1 a_2 a_3)^2\tau(\chi_{q_e})^2 & \text{ if } q_e \mid s_e \\ -\chi_{q_e}(a_1 a_2 a_3)^2\tau(\chi_{q_e})^2 & \text{ if } 2s_e = q_e \\ i \tau(\chi_{q_e})^2\chi_4(a_1 a_2 a_3) & \text{ if } s_e=2 \text{ and } q_e=8 . \end{cases}}
\end{mylemma}  
 If $q_e \mid c_e$ then the additive character on the left hand side of Lemma \ref{lemma:GevalPowerof2} cancels identically with the additive character appearing in \eqref{eq:G11evenprep}.  On the other hand, if $q_e \nmid c_e$ then the additive character $e_{[c_e,q_e]^3/c_e^2}(-a_1a_2a_3)= e_{q_e^3/c_e^2}(-a_1a_2a_3)$ must be expressed in terms of multiplicative characters. Recall, if $a_1a_2a_3\neq 0$ we factor $a_i=e_if_i$ with $e_i$ a power of $2$ and $f_i$ odd. We have 
\begin{equation} 
 \label{eq:qeexp} e_{q_e^3/c_e^2}(-a_1a_2a_3) =   \frac{1}{\varphi(q_e^3/c_e^2(q_e^3/c_e^2,a_1a_2a_3))} \sum_{\theta \shortmod{\frac{q^3_e}{c_e^2(q_e^3/c_e^2,a_1a_2a_3)}}} \tau(\overline{\theta})  \theta\left(-\frac{a_1a_2a_3}{(q_e^3/c_e^2,a_1a_2a_3)}\right), 
 \end{equation} where by convention we take $\theta \pmod{1}$ to be identically $1$.
 
Having computed $G_{1,1}(a_1,a_2, a_3;c_e)$ we now argue that the resulting expressions are of the desired form for Lemma \ref{lemma:GevalPowerof2}.  In similar fashion to the proof of Lemma \ref{lemma:GABevalSimplifiedForm}, note that for each fixed $q_e,c_e$ we have that $c_eq_e G'(a_1,a_2,a_3;c_e)$ is of the form \eqref{eq:GABevalEven} by inspecting \eqref{eq:qeexp}, \eqref{eq:qenmidce1}, \eqref{eq:qenmidce2}, \eqref{eq:Hs1} and \eqref{eq:Hs2}.  Now allowing the function $\mathfrak{g}$ to depend on $q_e,c_e$ and seeing that $\frac{q^3_e}{c_e^2(q_e^3/c_e^2,a_1a_2a_3)} \mid 64$ in all cases, we conclude the statement of Lemma \ref{lemma:GevalPowerof2}.
 \end{proof}

\begin{proof}[Proof of Proposition \ref{prop:Zproperties}]
Let 
\begin{equation*}
 G'_{A,B}(m_1, m_2, m_3;c) = G_{A,B}(m_1, m_2, m_3;c) e\Big(\frac{-m_1 m_2 m_3}{ AB [c,q]^3/c^2}\Big).
\end{equation*}
Then by \eqref{eq:GABchineseremaindertheorem} and \eqref{eq:G11crt} we have  
\begin{multline*}
 cq G_{A,B}'(m_1, m_2, m_3;c) = \chi_q(AB)
 c_o  q_o G_{1,1}'(m_1, m_2, \overline{AB c_2} \overline{[c_e,q_e]}^3c_e^2  m_3;c_o)   
\\ 
\times c_2 G_{A,B}'(m_1, m_2, \overline{c_o}\overline{[c_e,q_e]}^3c_e^2 m_3 ;c_2)
 c_eq_eG_{1,1}'(m_1,m_2,\overline{ABc_2c_o}m_3;c_e).
\end{multline*}

Now we factor $m_i = m_i'm_i''$, where $(m_i',AB)=1$ and $m_i'' \mid (AB)^\infty$, and then further factor $m_i' = m_i^o m_i^e$ where $m_i^o$ is odd and $m_i^e$ is a power of $2$.  By Lemmas \ref{lemma:GevalConreyIwaniec},  \ref{lemma:GABevalSimplifiedForm}, and \ref{lemma:GevalPowerof2}, we have 
\begin{multline}
\label{eq:G'GreatBigFormula}
cq G_{A,B}'(m_1,m_2,m_3;c) = \\ \delta((A,c_2) \mid  (m_1'',m_2''))\delta((B,c_2) \mid m_3'') \sum_{\substack{g_1 g_2 | \frac{c_2}{(A,c_2)} \\ g_1 = (\frac{m_1''}{(A,c_2)}, \frac{c_2}{(A,c_2)}) \\ g_2 = (\frac{m_2''}{(A,c_2)}, \frac{c_2}{g_1 (A,c_2)})}} \sum_{\substack{D_1 D_2 hk = q_o  \\ h = (q_o,s_o) \\ k = (m^o_1 m_2^o m_3^o, q) \\ (h, m_1^o m_2^o) = 1 \\ (m_3^o, c_o/q_o)=1}}   
 \frac{1}{\varphi(D_2)} \\
  \sum_{\psi \shortmod{D_2}} \sum_{\Delta | 64 } \frac{1}{\varphi(\Delta)} \sum_{\chi \shortmod{\Delta}} \sum_{D | \frac{c_2 AB}{(c_2, A)^2 (c_2, B)}} \frac{1}{\varphi(D)} \sum_{\eta \shortmod{D}} G_* \\
 (\psi\chi\eta)(m_1^om_2^om_3^o\overline{s_o})R_k(m_1^o) R_k(m_2^o) R_k(m_3^o),
\end{multline}
where $G_*$ is the product of the $g,\gamma$, and $\mathfrak{g}$ arising in Lemmas \ref{lemma:GevalConreyIwaniec}, \ref{lemma:GABevalSimplifiedForm}, and \ref{lemma:GevalPowerof2} along with various miscellaneous factors of unit size, such as $\chi_q(AB)$.  The exact form of $G_*$ is not important.  Rather, all that matters is a bound on its absolute value, and the fact that it does \emph{not} depend on $m_1^o,$ $m_2^o,$ $m_3^o,$ $s_o,$ $q_o,$ $c_o$.  Specifically, we have it is of size 
\es{
\label{eq:G*bound1}
G_* \ll D^{1/2} (A,c_2)(B,c_2) D_2^{3/2+\eps} 
} 
and if 
$\psi,$ $\chi,$ and $\eta$ are all the principal character then with $A = (A,c_2) A'$ and $B = (B,c_2) B'$ 
\es{
\label{eq:G*bound2}
\frac{G_*}{D} \ll (qr)^\eps (A,c_2)(B,c_2) \frac{\Big(\frac{m_1'' m_2'' m_3''}{(A,c_2)^2 (B,c_2)}, A'B'\Big)}{A' B'} .}  

Now we are ready to sum $G_{A,B}'$ over the $m_i$ and $c$.  We must break into two cases to handle the condition $(c,q_e)=\delta$.  

First suppose that $q$ is odd i.e., $q_e=1$.  Then since $(c,q_e)=1$ we have that the sum over $c$ is empty unless $\delta =1$, and if $\delta =1$ the the condition $(c,q_e)=\delta$ is true for all $c$.  We factor $R=R_eR_o$ where $R_o$ is odd and $R_e$ is a power of $2$ and write $c_o = t_oqR_o$ and $c_e = t_eR_e$.  Then for any function $f$ for which the sums converge absolutely we have \est{ \sum_{\substack{c \equiv 0 \shortmod{\tilde{q}R} \\ (c,q_e)=\delta}}f(c) = \delta(\delta=1)  \sum_{c_2 | (AB)^\infty} \sum_{\substack{t_e | 2^\infty \\ (t_e,AB)=1}} \sum_{(t_o,2AB)=1} f(c_2t_oqR_ot_eR_e).} Now we suppose that $q_e=4$  or $8$.  Then $R$ must be odd as $(R,q)=1$.  Recall that $(AB,qR)=1$ so that $(c_2,R\tilde{q})=1$, and also that $\tilde{q}/q_o = 2$.  Then we write $c_o = t_oq_oR$ and $c_e=2t_e$.  Then we have \est{  \sum_{\substack{c \equiv 0 \shortmod{\tilde{q}R} \\ (c,q_e)=\delta}}f(c) = \delta(\delta \mid q_e)\sum_{c_2 | (AB)^\infty} \sum_{\substack{t_e | 2^\infty \\ (t_e,AB)=1 \\ \delta | 2 t_e \\ (\frac{2t_e}{\delta},\frac{q_e}{\delta})=1}} \sum_{(t_o,2AB)=1} f(2c_2t_oq_oRt_e),} 
for any $f$ for which the sums converge absolutely.   
We treat only this last case for the remainder of the section, as the other cases are strictly simpler.  
Applying this decomposition to $G'_{A,B}$ we find if $\delta \mid q_e$ that 
\begin{multline}
\label{eq:greatbigZformula}
Z_{\delta,R,q} = 
\Big[
\sum_{c_2 | (AB)^\infty} \sum_{\substack{t_e | 2^\infty \\ (t_e,AB)=1 \\ \delta | 2 t_e \\ (\frac{2t_e}{\delta},\frac{q_e}{\delta})=1}} \sum_{\substack{m_1'',m_2'',m_3'' | (AB)^\infty }} \sum_{\substack{m_1^e,m_2^e,m_3^e | 2^\infty \\ (m_1^em_2^em_3^e,AB)=1}} 
\\
\sum_{(t_o,2AB)=1} \sum_{(m_1^om_2^om_3^o,2AB)=1}\frac{cq G'_{A,B}(m_1,m_2,m_3;c) }{m_1^{s_1}m_2^{s_2}m_3^{s_3}(c_2t_et_o)^{s_4}}
\Big]
 \\
=
\Big[
\sum_{\substack{c_2, t_e, m_1'', m_2'', m_3'', m_1^e, m_2^e, m_3^e \\ (A,c_2) | (m_1'',m_2''), \thinspace (B,c_2)|m_3'', \thinspace (\dots)}}   
 \sum_{\substack{g_1 g_2 | \frac{c_2}{(A,c_2)} \\ g_1 = (\frac{m_1''}{(A,c_2)}, \frac{c_2}{(A,c_2)}) \\ g_2 = (\frac{m_2''}{(A,c_2)}, \frac{c_2}{g_1 (A,c_2)})}} 
 \sum_{\Delta | 64 } \frac{1}{\varphi(\Delta)} \sum_{\chi \shortmod{\Delta}}
\\
\times 
 \sum_{D_1 D_2 hk = q_o }  \frac{1}{\varphi(D_2)} \sum_{\psi \shortmod{D_2}}  \sum_{D | \frac{c_2 AB}{(c_2, A)^2 (c_2, B)}} \frac{1}{\varphi(D)} \sum_{\eta \shortmod{D}} G_*' Y
\Big] 
 ,
\end{multline}
where $G_*'$ satisfies \eqref{eq:G*bound1} and \eqref{eq:G*bound2}, the conditions $(\dots)$ in the first (large) summand following the second equals sign are the same conditions as on the first line of \eqref{eq:greatbigZformula}, and 
\est{Y= \sum_{\substack{ (m_1^om_2^om_3^ot_o,2AB)=1 \\ h = (q_o,R_ot_o) \\ k=(m_1^om_2^om_3^o,q_o) \\ (h,m_1^om_2^o)=1 \\ (m_3^o,R_ot_o)=1}} \frac{(\psi\chi\eta)(m_1^om_2^om_3^o\overline{t_o}) R_k(m_1^o)R_k(m_2^o)R_k(m_3^o)}{(m_1^o)^{s_1} (m_2^o)^{s_2} (m_3^o)^{s_3}t_o^{s_4}}.}

Our next goal is to obtain meromorphic continuation of $Y$ inside the critical strip, and a bound on $Y$ both slightly to the right of the critical lines, and slightly to the right of the edge of absolute convergence.
First we note the following formal combinatorial identity:
\begin{equation*}
\sum_{(n_1 n_2 n_3, q_o) =k} f(n_1, n_2, n_3) =  \sum_{k_1 k_2 k_3 = k} \sum_{(n_1, \frac{q_o}{k_1})=1} \sum_{(n_2, \frac{q_o}{k_1 k_2}) = 1} \sum_{(n_3, \frac{q_o}{k_1 k_2 k_3}) = 1} f(k_1 n_1, k_2 n_2, k_3 n_3).
\end{equation*}
With this, we have (with some minor simplifications arising from $(q_o,AB) = 1$ which means for instance that $(k,AB) = 1$)
\begin{multline}
\label{eq:ZtildeSomekindofFormula}
Y = 
\sum_{\substack{k_1 k_2 k_3 = k }} \frac{( \psi \chi \eta)(k_1 k_2 k_3)}{k_1^{s_1} k_2^{s_2} k_3^{s_3} }
\sum_{\substack{(t_o, k_3 AB) = 1   \\ (t_o,q_o ) = h}}  
\\
\sum_{\substack{(m_1^o, hAB \frac{q_o }{k_1}) = 1 \\ (m_2^o, hAB \frac{q_o }{k_1 k_2}) = 1 \\ (m_3^o, R_ot_o AB \frac{q_o }{k}) = 1 }}
\frac{     (\psi \chi \eta)(m_1^o m_2^o m_3^o \overline{t_o}) }{(m_1^o)^{s_1} (m_2^o)^{s_2} (m_3^o)^{s_3} t_o^{s_4}   }
  R_k(k_1 m_1^o) R_k(k_2 m_2^o) R_k(k_3 m_3^o).
\end{multline}
The condition $(k_1 k_2, h) = 1$ is automatic, because $k_1 k_2 k_3 =k$, $(s_o,q_o) = h$, $hk \mid  q_o$, and $q_o$ is square-free, so $(h,k) = 1$.

Now let $a,b\in \mathbb{N}$ and suppose that $a''  \mid  a' \mid  a$, with $a'$ square-free, and $\chi$ is a Dirichlet character mod $a$.  Then for $\real(s)>1$ we have \es{\label{eq:aDS} \sum_{(n,b)=1} \frac{R_{a'}(a''n)\chi(n)}{n^s} = \mu(a'/a'')\phi( a'' )L_b(s,\chi),} where $L_b(s,\chi)$ is the Dirichlet $L$-function with Euler factors at primes dividing $b$ omitted.  
To see this, observe that if $a'$ is square-free then $\mu(a')R_{a'}(n)$ is a multiplicative function of $n$, and that the summand on the left side of \eqref{eq:aDS} vanishes whenever $(n,a'')\neq 1$ because $a'' \mid a$.  One can then factor out a Ramanujan sum from the left hand side and use the fact that $\mu(a')R_{a'}(n)=1$ if $(n,a')=1$.  

From this, we 
easily get the meromorphic continuation of $Y$ to, say $\text{Re}(s_i) > 1/2$, $i=1,2,3,4$.  Moreover, $Y$ is analytic except for possible poles at $s_i = 1$ in case $\eta \psi \chi$ is the principal character (which then implies all of $\chi$, $\eta$, and $\psi$ are principal, since their respective moduli are coprime).  Assuming $\text{Re}(s_i) = \sigma > 1/2$ for all $i=1,2,3,4$, and $\sigma \neq 1$, we have
\begin{equation}
|Y| \ll_{\sigma} k^{1-\sigma} h^{-\sigma} (qr)^{\varepsilon} |L(s_1, \psi \chi \eta) L(s_2, \psi \chi \eta) L(s_3, \psi \chi \eta) L(s_4, \overline{\psi \chi \eta})|.
\end{equation}

Now let $Z_{\delta,R,q} = Z_0 + Z'$ where $Z_0$ corresponds to the terms with $\eta \chi \psi$ principal, and $Z'$ corresponds to the terms with $\eta \chi \psi$ nonprincipal.

Taking $\sigma = 1 + \varepsilon$, we bound $Z_0$ as follows:
\begin{multline*}
|Z_0| \ll (qr)^{\varepsilon}
\sum_{\substack{ m_1'', m_2'', m_3'', c_2 | (AB)^{\infty} \\ (A,c_2) | (m_1'',m_2'') \\ (B,c_2) | m_3''}} \frac{1}{(m_1''  m_2''  m_3''  c_2 )^{\sigma}}
\sum_{\substack{D_1 D_2 hk = q_o   }} \sum_{\substack{g_1 g_2 | \frac{c_2}{(A,c_2)} \\ g_1 = (\frac{m_1''}{(A,c_2)}, \frac{c_2}{(A,c_2)}) \\ g_2 = (\frac{m_2''}{(A,c_2)}, \frac{c_2}{g_1 (A,c_2)})}}
  \\
  \times 
\frac{1}{\varphi(D_2)} 
\sum_{D | \frac{c_2 AB}{(c_2, A)^2 (c_2, B)}} 
(A,c_2)(B,c_2) \frac{(\frac{m_1'' m_2'' m_3''}{(A,c_2)^2 (B,c_2)}, A'B')}{A' B'} k^{1-\sigma} h^{-\sigma}.
\end{multline*}
Next we change variables $m_i'' = n_i (A,c_2)$ for $i=1,2$ and $m_3'' = (B,c_2) n_3$.  We have
\begin{equation}
\label{eq:SumOverDivisorsofABwithGCDcondition}
 \sum_{n | (AB)^{\infty}} \frac{(n,Q)}{n} = \prod_{p | AB} \sum_{j=0}^{\infty} \frac{(p^j,Q)}{p^j} \ll (ABQ)^{\varepsilon}.
\end{equation}
Using this successively on $n_1, n_2, n_3$, and trivially summing over $g_1, g_2, D$, we obtain
\begin{equation*}
 |Z_0| \ll (qr)^{\varepsilon} 
 \sum_{c_2 | (AB)^{\infty}} \frac{1}{c_2^{\sigma}}
\sum_{\substack{D_1 D_2 hk = q_o   }} 
\frac{1}{\varphi(D_2)} 
\frac{(A,c_2)(B,c_2)}{(A,c_2)^2 (B,c_2)} \frac{1}{A' B'} k^{1-\sigma} h^{-\sigma}.
\end{equation*}
We use the estimate \eqref{eq:SumOverDivisorsofABwithGCDcondition} again on the sum over $c_2$  to get 
\begin{equation*}
|Z_0| \ll \frac{(qr)^{\varepsilon}}{AB} \sum_{D_1 D_2 hk = q_o} k^{1-\sigma} h^{-\sigma} D_2^{-1} \ll \frac{(qr)^{\varepsilon}}{AB}.
\end{equation*}
This proves the bound \eqref{eq:Z00boundRightOfCriticalStrip}, as desired.

Next we turn to $Z'$.  For this, we use the large sieve inequality to give a bound on the $4$th moment of Dirichlet $L$-functions.  Following e.g. \cite[Lemma 8]{Petrow} we find that for $\sigma = 1/2 + \varepsilon$, we have
\begin{equation}
\label{eq:fourthmoment}
 \frac{1}{\phi(Q)} \sum_{\chi \shortmod{Q}} |L(s, \chi)|^4 \ll Q^{\varepsilon} (1 + |s|)^{1 + \varepsilon}. 
\end{equation}
Using H\"{o}lder's inequality, we have for $\sigma_i = 1/2 + \varepsilon$ for $i=1,2,3,4$, that
\begin{multline*}
|Z'| \ll 
(qr)^{\varepsilon} \sum_{\substack{ m_1'', m_2'', m_3'', c_2 | (AB)^{\infty} \\ (A,c_2) | (m_1'',m_2'') \\ (B,c_2) | m_3''}} \frac{1}{(m_1''  m_2''  m_3''  c_2)^{1/2 + \varepsilon}}
\sum_{\substack{D_1 D_2 hk = q   }} \sum_{\substack{g_1 g_2 | \frac{c_2}{(A,c_2)} \\ g_1 = (\frac{m_1''}{(A,c_2)}, \frac{c_2}{(A,c_2)}) \\ g_2 = (\frac{m_2''}{(A,c_2)}, \frac{c_2}{g_1 (A,c_2)})}}
  \\
\sum_{D | \frac{c_2 AB}{(c_2, A)^2 (c_2, B)}} 
(A,c_2) (B,c_2) D^{1/2} 
D_2^{3/2}
 k^{1-\sigma} h^{-\sigma}
\prod_{j=1}^{4} (1+ |s_j|)^{1/4+\varepsilon} 
 . 
\end{multline*}
Using similar methods to estimate the sums over the $m_i''$ and $c_2$ as in the bound on $Z_0$, we obtain
\begin{equation*}
|Z'| \ll (qr)^{\varepsilon} q^{3/2+\varepsilon} (AB)^{1/2} \prod_{j=1}^{4} (1+ |s_j|)^{1/4+\varepsilon} .
\end{equation*}
Finally, we show \eqref{eq:Z'integralbound}.  The proof is essentially the same as before, except we use a hybrid large sieve in place of \eqref{eq:fourthmoment}, as in \cite{Gallagher}, namely
\begin{equation*}
 \int_{|t| \leq T} \sum_{\chi \shortmod{q}} |L(1/2 + it, \chi)|^4 dt \ll (qT)^{1+\varepsilon}. 
\end{equation*} 
\end{proof}

We conclude this section by studying $G_{A,B}(m_1, m_2, m_3;c)$ when some $m_i=0$.  The formulas greatly simplify.

\begin{mylemma}\label{lemma:GboundOther}
Suppose some $a_j=0$. If $a_i \neq 0$ write $a_i = a''_ia'_i$ where $a''_i \mid (2AB)^\infty$ and $(a'_i,2AB)=1$, and if $a_i=0$ write $a''_i=a'_i= 0$. Then $$ cq\phi(q) G_{A,B}(a_1,a_2,a_3;c) = g(a''_1,a''_2,a''_3, A,B,c, q) R_{q_o}(a'_1)R_{q_o}(a'_2)R_{q_o}(a'_3),$$ where $g$ is a function satisfying the bound $$|g(a''_1,a''_2,a''_3, A,B,c, q)| \leq 64(A,c)(B,c).$$
\end{mylemma} 
\begin{proof}

We have according to \eqref{eq:GABchineseremaindertheorem} and \eqref{eq:G11crt} that 
\begin{multline*}
G_{A,B}(a_1,a_2,a_3;c) = \chi_q(AB) 
 G_{1,1}(a_1,a_2, \overline{ABc_2}\overline{[c_e,q_e]}^3c_e^2a_3;c_o) \\ 
\times G_{1,1}(a_1,a_2, \overline{ABc_2c_o}a_3;c_e)  G_{A,B}(a_1,a_2, \overline{[c_1,q]}^3c_1^2a_3;c_2).
\end{multline*}
We treat each of these in turn. Let us begin with  $G_{1,1}(a_1,a_2, \overline{ABc_2}\overline{[c_e,q_e]}^3c_e^2a_3;c_o)$. The evaluation \eqref{eq:G11CIformula} of $G_{1,1}$ did not require $a_1a_2 a_3 \neq 0$. Inspecting \eqref{eq:G11CIformula}, and noting $h=1$, $k=q_o$ under the assumption $a_1 a_2 a_3 = 0$,  we have \begin{multline*} G_{1,1}(a_1,a_2, \overline{ABc_2}\overline{[c_e,q_e]}^3c_e^2a_3;c_o) = \\
\begin{cases} \frac{\chi_{q_o}(-1)}{c_oq_o\phi(q_o)} R_{q_o}(a_1)R_{q_o}(a_2)R_{q_o}(a_3) & \text{ if } (\frac{c_o}{q_o} ,a_3)=1 \text{ and } (q_o, \frac{c_o}{q_o} , a_1a_2)=1 \\
0 & \text{ else.}\end{cases}\end{multline*}

Next consider $G_{1,1}(a_1,a_2, \overline{ABc_2c_o}a_3;c_e) $. An inspection of the proof of Lemma \ref{lemma:GevalPowerof2} shows that when $a_1a_2a_3=0$ that $$G_{1,1}(a_1,a_2, \overline{ABc_2c_o}a_3;c_e) = \begin{cases} \frac{1}{q_ec_e\phi(q_e)} \chi_{q_e}(-1) R_{q_e}(a_1)R_{q_e}(a_2)R_{q_e}(a_3) & \text{ if } q_e=c_e \\ 
\frac{1}{q_ec_e \phi(q_e)} \delta_{(a_3,c_e)=1} & \text{ if } q_e = 1 \\ 
0 & \text{ else,}\end{cases}$$
where note that in the case $q_e = c_e= 8$, $a_1 a_2 a_3 = 0$, the third case in \eqref{eq:Hs1} may be discarded, and in the second case of \eqref{eq:Hs1}, the condition $4|a_1, a_2, a_3$ may be dropped since the Ramanujan sum vanishes otherwise.
This function only depends on $q_e$,$c_e$, and the $2$-part of $a_1,a_2,a_3$ and is bounded above by $64/q_ec_e\phi(q_e)$, since each Ramanujan sum is bounded by $4$ in absolute value.

Lastly, consider $G_{A,B}(a_1,a_2, \overline{[c_1,q]}^3c_1^2a_3;c_2)$.   As mentioned in the proof of Lemma \ref{lemma:GABevalSimplifiedForm},  \eqref{eq:GABformulaWithTildeVariables} is valid without the assumption that $a_1a_2a_3\neq 0$. We have that $G_{A,B}(a_1,a_2, \overline{[c_1,q]}^3c_1^2a_3,c_2)$ only depends on $a_1'',a_2'',a_3'', A,B$, and $c_2$. In particular, by \eqref{eq:RcgBound} we have $$ |G_{A,B}(a_1,a_2, \overline{[c_1,q]}^3c_1^2a_3;c_2)| \leq \frac{(A,c_2)(B,c_2)}{c_2}.$$

Therefore $$G_{A,B}(a_1,a_2,a_3;c) = g(a''_1,a''_2,a''_3, A,B,c, q)\frac{1}{cq\phi(q)} R_q(a'_1)R_q(a'_2)R_q(a'_3),$$ and $| g(a''_1,a''_2,a''_3, A,B,c, q)| \leq (A,c_2)(B,c_2).$
 \end{proof}

\section{Weight functions}
\subsection{Inert functions}
\label{section:inert}
We begin this section by quoting a definition of \cite{KiralPetrowYoung}.

Let $\mathcal{F}$ be an index set and $X=X_T: \mathcal{F} \to \R_{\geq 1}$ be a function of $T \in \mathcal{F}$.
\begin{mydefi}\label{inert}
A family $\{w_T\}_{T\in \mathcal{F}}$ of smooth  
functions supported on a product of dyadic intervals in $\R_{>0}^d$ is called $X$-inert if for each $j=(j_1,\ldots,j_d) \in \mathbb{Z}_{\geq 0}^d$ we have 
\begin{equation}
\label{eq:inert}
C(j_1,\ldots,j_d):=  \sup_{T \in \mathcal{F}} \sup_{(x_1, \ldots, x_d) \in \R_{>0}^d} X_T^{-j_1- \cdots -j_d}\left| x_{1}^{j_1} \cdots x_{d}^{j_d} w_T^{(j_1,\ldots,j_d)}(x_1,\ldots,x_d) \right|  < \infty.
\end{equation}
\end{mydefi}

We often abbreviate the sequence of constants $C(j_1,\ldots, j_d)$ associated to a family of inert functions by $C_{\mathcal{F}}$.  We will most often use the definition of $X$-inert with $X=X_T=1$,
although occasionally (e.g., in the proof of Lemma \ref{lemma:KboundNonOscillatory}) $X$ will be slightly larger than $1$, with $X \ll (qr)^{\varepsilon}$.
%although a different choice of $X$ also appears in Lemma \ref{lemma:KboundNonOscillatory}.

The purpose of this definition is to encode natural conditions on a weight function that lets us separate variables efficiently.  For instance, if $w_T$ satisfies \eqref{eq:inert}, then by Mellin inversion,
\begin{equation}
\label{eq:wTseparation}
 w_T(x_1, \dots, x_d) = \frac{1}{(2\pi )^d} \int_{ \mathbb{R}^d} \widetilde{w_T}(it_1, \dots it_d) x_1^{-it_1} \dots x_d^{-it_d} dt_1 \dots dt_d,
\end{equation}
where
\begin{equation*}
 \widetilde{w_T}(s_1, \dots, s_d) = \int_{(0,\infty)^d} w_T(x_1, \dots, x_d) x_1^{s_1} \dots x_d^{s_d} \frac{d x_1}{x_1} \dots \frac{d x_d}{x_d}.
\end{equation*}
Integrating by parts shows for any choices of $j_1,\ldots, j_d = 0,1, \ldots$ , we have
\begin{equation*}
 \widetilde{w_T}(s_1, \dots, s_d) =  \Big(\prod_{a=1}^{d} \prod_{b=0}^{j_a-1} \frac{-1}{s_a + b} \Big) \int_{(0,\infty)^d} w_T^{(j_1, \dots, j_d)}
 (x_1, \dots, x_d)  x_1^{s_1+j_1} \dots x_d^{s_d+j_d} \frac{d x_1}{x_1} \dots \frac{d x_d}{x_d}.
\end{equation*}
Therefore, by \eqref{eq:inert}, we have
\begin{equation*}
 |\widetilde{w_T}(it_1, \dots, it_d)| \leq 
 \Big(\frac{X_{T}}{|t_1|}\Big)^{j_1} \dots \Big(\frac{X_{T}}{|t_d|}\Big)^{j_d} C(j_1, \dots, j_d) (\log 2)^d.
\end{equation*}
If $|t_i| \geq X_{T}$, then we take $j_i$ as unspecified (arbitrarily large), while if $|t_i| < X_{T}$, we choose $j_i = 0$.  In this way, we obtain
\begin{equation}
 |\widetilde{w_T}(it_1, \dots, it_d)| \leq \Big(1 + \frac{|t_1|}{X_{T}}\Big)^{-j_1} \dots \Big(1 + \frac{|t_d|}{X_{T}}\Big)^{-j_d} C'(j_1, \dots, j_d),
\end{equation}
where $C'$ is some other sequence depending only on $C$.  Our interpretation of this estimate combined with \eqref{eq:wTseparation} is that $w_T$ can have its variables separated ``at cost'' $X_{T}^d$, meaning that each integral has essential length $\ll X_{T}$.

\subsection{Integration by parts}
Often an integral can be shown to be small by repeated integration by parts.  For this, we quote \cite[Lemma 8.1]{BKY} with some slight changes of notation and terminology.
\begin{mylemma}[\hspace{1sp}\cite{BKY}]
\label{lemma:exponentialintegral}
 Suppose that $w = w_T(t)$ is a family of $X$-inert functions, 
 with compact support on $[Z, 2Z]$, so that
$w^{(j)}(t) \ll (Z/X)^{-j}$.  Also suppose that $\phi$ is smooth and satisfies $\phi^{(j)}(t) \ll \frac{Y}{Z^j}$ for some $Y/X^2 \geq R \geq 1$ and all $t$ in the support of $w$.  Let
\begin{equation*}
 I = \intR w(t) e^{i \phi(t)} dt.
\end{equation*}
 If $|\phi'(t)| \gg \frac{Y}{Z}$ for all $t$ in the support of $w$, then $I \ll_A Z R^{-A}$ for $A$ arbitrarily large.
\end{mylemma}

\subsection{Stationary phase} \label{subsec:stationaryPhase}
Now we quote the main theorem from \cite{KiralPetrowYoung} which extends \cite[Proposition 8.2]{BKY}.
\begin{mytheo}[\hspace{1sp}\cite{KiralPetrowYoung}]
 \label{thm:exponentialintegral}
 Suppose $w_T$ is $X$-inert in $t_1, \dots t_d$, supported on $t_1 \asymp Z$ and $t_i \asymp X_i$ for $i=2,\dots, d$.  Suppose that on the support of $w_T$, $\phi = \phi_T$ satisfies
\begin{equation}
\label{eq:phiderivatives}
 \frac{\partial^{a_1 + a_2 + \dots + a_d}}{\partial t_1^{a_1} \dots \partial t_d^{a_d}} \phi(t_1, t_2, \dots, t_d) \ll_{C_\mathcal{F}} \frac{Y}{Z^{a_1}} \frac{1}{X_2^{a_2} \dots X_d^{a_d}},
\end{equation}
for all $a_1, \ldots, a_d \in \mathbb{Z}_{\geq 0}$.
Suppose $ \phi''(t_1, t_2, \dots, t_d) \gg \frac{Y}{Z^2}$, (here and below, $\phi'$  and $\phi''$ denote the derivative with respect to $t_1$) for all $t_1, t_2, \dots, t_d$ in the support of $w_T$, and there exists $t_0 \in \mr$ such that $\phi'(t_0, t_2,\dots,t_d) = 0$ with some $t_0 \in \R$ depending on $t_2,\dots,t_d$ (note $t_0$ is necessarily unique). 
Suppose that $Y/X^2 \geq R \geq 1$.  Then  
\begin{equation}
\label{eq:IasymptoticMainThm}
I = \int_{\mathbb{R}} e^{i \phi(t_1, \dots, t_d)} w_T(t_1, \dots, t_d) dt_1 =  \frac{Z}{\sqrt{Y}} e^{i \phi(t_0, t_2, \dots, t_d)} W_T(t_2, \dots, t_d)  + O_{A}(ZR^{-A}),
\end{equation}
for some $X$-inert family of functions $W_T$, and where $A>0$ is arbitrarily large.
The implied constant in \eqref{eq:IasymptoticMainThm} depends only on $A$ and $C_{\mathcal{F}}$.
\end{mytheo}

\subsection{The integral transform}
Here we obtain useful expressions for $K$, which was defined in \eqref{eq:Kdef}.  The key is not an exact formula for $K$, but rather a Mellin formula with the variables separated.  Throughout the remainder of this section, $w_T$ will denote a member of a  $1$-inert family of functions, which may change from line-to-line without explicit mention.  
We also recall that $[c,q] = c \frac{q_e}{(c,q_e)}$, where $\frac{q_e}{(c,q_e)}$ takes the possible values $1,2,4$.

\begin{mylemma}[Oscillatory Case]
 \label{lemma:Kbound}
Suppose that $|m_i| \asymp M_i$ for $i=1,2,3$, and $c \asymp C$, with $\tilde{q} | c$.  
 Suppose that
 \begin{equation}
 \label{eq:OscCondition}
  \frac{\sqrt{ABN_1 N_2 N_3}}{C} \gg (qr)^{\varepsilon}.
 \end{equation}

Then
\begin{multline}
\label{eq:Kseparated2}
K(m_1, m_2, m_3, c) = \frac{C^{3/2} (N_1 N_2 N_3)^{1/2} e(\frac{-m_1 m_2 m_3c^2}{AB[c,q]^3})}{(M_1 M_2 M_3)^{1/2}}  L(m_1, m_2, m_3, c)
\\
+ O((qr)^{-1/\varepsilon} \prod_{i=1}^{3} (1+ |m_i|)^{-2} ),
\end{multline}
where $L$ has the following properties.  Firstly, $L$ vanishes (meaning $K$ is very small) unless
\begin{equation}
\label{eq:MsizeOscillatoryRange}
M_i \asymp \frac{(AB N_1 N_2 N_3)^{1/2}}{N_i}, \quad i=1,2,3,
\end{equation}
and all the $m_i$ have the same sign.
Moreover, we have with
\begin{equation}
\label{eq:Pdef}
P = \frac{M_1 M_2 M_3}{ABC}.
\end{equation}
that
\begin{multline}
\label{eq:LIntegralFormulaOscillatoryRange}
L(m_1, m_2, m_3, c) = \frac{1}{P^{1/2}} \int_{|{\bf u}| \ll (qr)^{\varepsilon}} \int_{|y| \ll (qr)^{\varepsilon}} F({\bf u};y) \Big(\frac{|m_1 m_2 m_3| c^2}{[c,q]^3}\Big)^{iy} 
\\
\Big(\frac{M_1}{|m_1|}\Big)^{u_1} \Big(\frac{M_2}{|m_2|}\Big)^{u_2} \Big(\frac{M_3}{|m_3|}\Big)^{u_3} \Big(\frac{C}{c}\Big)^{u_4} d{\bf u} dy,
\end{multline}
where $F = F_{A,B, C, N_1, N_2, N_3, M_1, M_2, M_3}$ is entire in terms of ${\bf u}$, and satisfies $F({\bf u};y) \ll_{\text{Re}(\bf{u})} (1+|{\bf u}|)^{-J} (1+|y|)^{-J}$, for $J$ arbitrarily large.
Here $F$ additionally depends on the choice of signs of the $m_i$, and on the values $q_e$, $(c,q_e)$.
\end{mylemma}

\begin{mylemma}[Non-oscillatory case]
\label{lemma:KboundNonOscillatory}
Suppose that $|m_i| \asymp M_i$ for $i=1,2,3$,  $c \asymp C$, and
\begin{equation}
\label{eq:NonOscCondition}
 \frac{\sqrt{AB N_1 N_2 N_3}}{C} \ll (qr)^{\varepsilon}.
\end{equation}
Then
\begin{multline}
\label{eq:KintegralformulaNonOscillatory}
 K(m_1, m_2, m_3, c) = \\
N_1 N_2 N_3 \Big(\frac{\sqrt{ABN_1 N_2 N_3}}{C}\Big)^{\kappa -1} e\Big(\frac{-m_1 m_2 m_3c^2}{AB[c,q]^3}\Big)   
\int_{|{\bf u}| \ll (qr)^{\varepsilon}} F({\bf u})
\int_{|t| \ll (qr)^{\varepsilon} + P} \\ f(t) \Big(\frac{|m_1 m_2 m_3| c^2}{[c,q]^3}\Big)^{it}  
\Big(\frac{M_1}{|m_1|}\Big)^{u_1} \Big(\frac{M_2}{|m_2|}\Big)^{u_2} \Big(\frac{M_3}{|m_3|}\Big)^{u_3} \Big(\frac{C}{c}\Big)^{u_4} dt d{\bf u}
 \\
  + O((qr)^{-1/\varepsilon} \prod_{i=1}^{3} (1+ |m_i|)^{-2} ),
\end{multline}
where $P$ is given by \eqref{eq:Pdef}, $f(t) \ll (1+|t|)^{-1/2}$, and $F({\bf u}) \ll_{J, \text{Re}({\bf u})} \prod_{\ell=1}^{4} (1+\frac{(qr)^{\varepsilon}}{|u_{\ell}|})^{-J}$.  Moreover, $f$ vanishes (meaning, $K$ is small) unless
\begin{equation}
\label{eq:smallkiMi}
 \frac{M_1 N_1}{C} \ll (qr)^{\varepsilon}, \qquad  \frac{M_2 N_2}{C} \ll (qr)^{\varepsilon}, \qquad  \frac{M_3 N_3}{C} \ll (qr)^{\varepsilon}.
\end{equation}
If $P \gg (qr)^{\varepsilon}$, the function $f$ may be chosen to have support on $|t| \asymp P$.
\end{mylemma}

\begin{mylemma}[Other cases]
\label{lemma:KboundOther}
Suppose some $m_i = 0$.  If \eqref{eq:OscCondition} holds, then $K$ is small.  If \eqref{eq:NonOscCondition} holds, then $K$ is small unless $|m_j| \ll \frac{C}{N_j} (qr)^{\varepsilon}$ for $j=1,2,3$, in which case
\begin{equation}
\label{eq:KboundOther}
K(m_1, m_2, m_3;c) \ll  \Big(\frac{\sqrt{ABN_1 N_2 N_3}}{C}\Big)^{\kappa-1} N_1 N_2 N_3.
\end{equation}
If say $m_3 = 0$ but $m_1 m_2 \neq 0$, then with $N = N_1 N_2 N_3$, we have a Mellin formula
\begin{multline}
\label{eq:Kformulawithm3zero}
K(m_1, m_2, 0 ;c) = \Big(\frac{\sqrt{ABN}}{C}\Big)^{\kappa-1} N 
\int_{|v_1| \ll (qr)^{\varepsilon}} \int_{|v_2| \ll (qr)^{\varepsilon}} \Big(\frac{C}{N_1 |m_1|}\Big)^{v_1} \\  \Big(\frac{C}{N_2 |m_2| }\Big)^{v_2} R(v_1, v_2, c) dv_1 dv_2 + O((qr)^{-1/\varepsilon} \prod_{i=1}^{3} (1+|m_i|)^{-2}),
\end{multline}
where $R(v_1, v_2,c)$ is analytic in $\text{Re}(v_i) > 0$ for $i=1,2$, and satisfies the bound
\begin{equation*}
R(v_1, v_2, c) \ll_{J, \text{Re}(v_1), \text{Re}(v_2)} \prod_{j=1}^{2} (1+\frac{(qr)^{\varepsilon}}{|v_j|})^{-J}.
\end{equation*}
Here $R$ depends on the choices of sign of the $m_i$, but we suppress it from the notation.  Similar formulas hold when $m_1=0$ or $m_2=0$.

If two $m_i = 0$ but the other $m_i$ is nonzero, then a formula similar to \eqref{eq:Kformulawithm3zero} holds, but with one of the integrals omitted.
\end{mylemma}

\begin{proof}
We prove all three lemmas.

\textbf{Truncations.}
As our first step, we integrate by parts three times in each of the $t_i$ for which $m_i \neq 0$, allowing us to obtain a crude bound of the form
\begin{equation*}
 K(m_1, m_2, m_3, c) \ll P(q, r, N_1, N_2, N_3, c) \prod_{i=1}^{3} (1+ |m_i|)^{-3},
\end{equation*}
where $P$ is some fixed polyomial (which may as well be taken to be some polynomial in $qr$ by \eqref{eq:Csize}).  This bound is sufficient for the lemmas when some $|m_i|$ is $\gg (qr)^{A'}$ for some large $A'$ depending polynomially on $1/\varepsilon$.
For the rest of the proof, suppose that $|m_i| \ll (qr)^{A'}$ for some $A'$, and each $i$.

In the oscillatory case (i.e., if \eqref{eq:OscCondition} holds), then using the fact that $J_{\kappa-1}(x) = e^{ix} W_{+}(x) + e^{-ix} W_{-}(x)$ where $W_{\pm}(x)$ have controlled derivatives (cf.~Watson \cite[Page 205]{Watson}), we see that repeated integration by parts (Lemma \ref{lemma:exponentialintegral}) shows that $K$ is very small unless \eqref{eq:MsizeOscillatoryRange} holds, and moreover, all the $m_i$ must have the same sign.
Similarly, in the non-oscillatory case where \eqref{eq:NonOscCondition} holds, then repeated integration by parts shows that $K$ is small unless \eqref{eq:smallkiMi} holds.

\textbf{Proof of Lemma \ref{lemma:Kbound}.}
Now we show the expression \eqref{eq:Kseparated2}, with $L$ given by \eqref{eq:LIntegralFormulaOscillatoryRange}.
Using the Fourier integral (valid for $n$ an odd integer)
\begin{equation*}
 J_{n}(x) = 
 \sum_{\pm} \frac{\pm 1}{\pi i} \int_0^{\pi/2} \sin(n v) e^{\pm ix \sin v} dv,
 \end{equation*}
we have
\begin{multline*}
 K = K(m_1, m_2, m_3;c)  
 = \\
\sum_{\pm}
 \frac{\pm 1}{  \pi i} 
 \int_{0}^{\pi/2} \sin( (\kappa-1)v)
 \int_{\mathbb{R}^3} e\Big(\frac{\pm 2 \sqrt{AB t_1 t_2 t_3}}{c} \sin(v) \Big) \\ e\Big(\frac{-m_1 t_1 - m_2 t_2 - m_3 t_3}{[c,q]}\Big) 
 w_{T}(t_1,t_2, t_3)  dt_1 dt_2 dt_3 dv .
\end{multline*}

Next we change variables $t_3 = u\frac{ N_1 N_2}{t_1 t_2}$, giving
\begin{equation*}
K =  \sum_{\pm}
 \frac{\pm 1}{  \pi i} \int_{0}^{\pi/2} \sin( (\kappa-1)v) \int_0^{\infty}
 e\Big(\frac{\pm 2 \sqrt{AB u N_1 N_2}}{c} \sin(v)\Big)  I(u) du dv,
 \end{equation*}
 where
\begin{equation*}
 I(u) = \int_{\mathbb{R}^2}
 e\Big(\frac{-m_1 t_1 - m_2 t_2 - m_3 \frac{u N_1 N_2}{t_1 t_2}}{[c,q]}\Big)
 w_{T}(t_1,t_2, u) dt_1 dt_2 .
\end{equation*}
The conditions \eqref{eq:OscCondition} and \eqref{eq:MsizeOscillatoryRange} imply that
\begin{equation*}
 \frac{M_i N_i}{C} \asymp \frac{\sqrt{AB N_1 N_2 N_3}}{C} \gg (qr)^{\varepsilon}.
\end{equation*}
Recall that we already showed, in the paragraphs following the statement of Lemma \ref{lemma:KboundOther}, that under the assumption of \eqref{eq:OscCondition}, $K$ is small unless \eqref{eq:MsizeOscillatoryRange} holds. 
The conditions are now in place to analyze the inner $t_1, t_2$ integrals using stationary phase.  Evaluating the $t_1$-integral first, we find a stationary point at $t_1^0 = (\frac{m_3 u N_1 N_2}{m_1 t_2})^{1/2}$.  Following this, the stationary point in terms of $t_2$ occurs at $t_2^0 = m_2^{-2/3} (m_1 m_3 u N_1 N_2)^{1/3}$.  We therefore deduce
\begin{multline}
 I(u) = \frac{C}{(N_1 N_2 M_1 M_2)^{1/2}} e\Big( \frac{-3(m_1 m_2 m_3 u N_1 N_2)^{1/3}}{[c,q]}\Big) w_T(u, m_1, m_2, m_3, c) 
 \\
 + O((qr)^{-1/\varepsilon} \prod_{i=1}^{3} (1+|m_i|)^{-2}),
\end{multline}
where $w_T$ is $1$-inert in all variables and $T$ stands for the tuple $$(M_1, M_2, M_3, N_1, N_2, N_3, C, q_e, (c,q_e))$$ and supported on $u \asymp N_3$.  The condition that $m_1, m_2, m_3$ must all have the same sign may be encoded in the support of the inert function.  If all three terms are negative, we naturally interpret the expression $(m_1 m_2 m_3)^{1/3}$ as  $- (|m_1 m_2 m_3|)^{1/3}$.
Therefore 
\begin{multline}
 K = \Big[\sum_{\pm} \frac{\pm 1}{\pi i} \frac{C}{(N_1 N_2 M_1 M_2)^{1/2}}\int_{0}^{\pi/2} \sin((\kappa - 1) v) \\
 \int_0^{\infty}
 e\Big(\frac{ \pm 2 \sqrt{AB u N_1 N_2}}{c} \sin(v)\Big)  
  e\Big( \frac{-3(u m_1 m_2 m_3 N_1 N_2)^{1/3}}{[c,q]}\Big) w_T(u, \cdot) du dv\Big] \\
 + O((qr)^{-1/\varepsilon} \prod_{i=1}^{3} (1+|m_i|)^{-2}).
\end{multline}
Here we use the notation $w_T(u, \cdot)$ to denote a function where we are currently focusing on the variable $u$ only, and so do not display the other variables.

Finally, we study
\begin{multline*}
K_0(v) := \frac{C}{(N_1 N_2 M_1 M_2)^{1/2}} \int_0^{\infty}  e(\frac{\pm 2  \sqrt{ABu N_1 N_2}}{c} \sin(v))\\
e\Big( \frac{-3(u m_1 m_2 m_3 N_1 N_2)^{1/3}}{[c,q]}\Big) w_T(u, \cdot) du.
\end{multline*}
We will presently show that
\begin{equation}
 K_0(v) = \frac{C^{3/2} (N_1 N_2 N_3)^{1/2}}{(M_1 M_2 M_3)^{1/2}} e\Big(\frac{-m_1 m_2 m_3c^2}{AB [c,q]^3 \sin^2 v}\Big) w_{T}(\sin v, \cdot)  + O((qr)^{-1/\varepsilon} \prod_{i=1}^{3} (1+|m_i|)^{-2}),
\end{equation}
where $w_{T}(\sin v, \cdot)$ is part of a $1$-inert family of functions of $\sin v, m_1, m_2, m_3, c$, with $T$ as before but in addition depending on the choice of $\pm$ sign.

This integral defining $K_0(v)$ is small unless there is a stationary point (by Lemma \ref{lemma:exponentialintegral} again), which occurs at
\begin{equation*}
 u = u_0 = \frac{(m_1 m_2 m_3 )^2c^6}{N_1 N_2 (AB)^3 \sin^6 v [c,q]^6},
\end{equation*}
under the additional assumption that the choice of $\pm$ sign matches the sign of $m_1 m_2 m_3$ (which in turn has the same sign as each individual $m_i$).  Note that for this stationary point to lie inside the support of the inert function, we need
\begin{equation}
\label{eq:Vsize}
\sin v \asymp \frac{(M_1 M_2 M_3)^{1/3}}{(AB)^{1/2} (N_1 N_2 N_3)^{1/6}} =: V.
\end{equation}

Thus we obtain, in both cases of $\pm$ sign, that
\begin{equation*}
K_0(v) =  (\text{scaling factor}) e\Big(\frac{-m_1 m_2 m_3c^2}{AB  \sin^2 v [c,q]^3}\Big) w_{T}(\sin v, \cdot)  + O((qr)^{-1/\varepsilon} \prod_{i=1}^{3} (1+|m_i|)^{-2}).
\end{equation*}
A short calculation shows the scaling factor is $\asymp \frac{C^{3/2} (N_1 N_2 N_3)^{1/2}}{(M_1 M_2 M_3)^{1/2}}$, proving the claimed expression for $K_0(v)$.

Thus we obtain 
\begin{equation*}
K = \frac{C^{3/2} (N_1 N_2 N_3)^{1/2}}{(M_1 M_2 M_3)^{1/2}} \int_{0}^{\pi/2} 
 \sin((\kappa -1)v)  e\Big(\frac{-m_1 m_2 m_3c^2}{AB  \sin^2 v [c,q]^3}\Big) w_T(\sin v, \cdot) dv,
\end{equation*}
plus an error of size $O((qr)^{-1/\varepsilon} \prod_i (1 + |m_i|)^{-2})$,  
where the support of the inert function is given by \eqref{eq:Vsize}.

Next we factor out the desired exponential, giving now
\begin{multline*}
K = e\Big(\frac{-m_1 m_2 m_3 c^2}{AB [c,q]^3}\Big) \frac{C^{3/2} (N_1 N_2 N_3)^{1/2}}{(M_1 M_2 M_3)^{1/2}} 
\\
\times 
\int_{0}^{\pi/2} 
 \sin((\kappa - 1)v)  e\Big(\frac{m_1 m_2 m_3c^2}{AB [c,q]^3}\Big(1-\frac{1}{\sin^2 v}\Big)\Big) w_T(\sin v, \cdot) dv,
\end{multline*}
plus a small error term. 
Define
\begin{equation}
 K_{00}(x) := \int_{0}^{\pi/2} 
 \sin((\kappa - 1)v)  e\Big(x(-1+\frac{1}{\sin^2 v})\Big) w_T(\sin v, \cdot) dv.
\end{equation}
Note that for our particular values of the parameters, we have
\begin{equation*}
 \frac{x}{V^2}  \gg (qr)^{\varepsilon}.
\end{equation*}

Now we asymptotically evaluate $K_{00}(x)$.  First we dispense with the case where the support of the inert function is so that the integrand vanishes unless  $v \leq \frac{\pi}{2} - \frac{\pi}{100}$, say.  Thus $\cos(v) \asymp 1$ and $\sin(v) \asymp v \asymp V$ in the support of the integrand.  We claim that $K_{00}(x) \ll (qr)^{-A}$  in this case.  To see this, we first note that it suffices to bound
\begin{equation*}
\int_{-\infty}^{\infty} w_T(\sin v,\cdot) e(\phi(v) ) dv, 
\quad \text{where} \quad 
\phi(v) = \pm \frac{\kappa-1}{2\pi} v + x \frac{\cos^2{v}}{\sin^2{v}}.
\end{equation*}
We have
\begin{equation*}
\phi(v) = \pm \frac{\kappa -1}{2 \pi } v + \frac{x}{v^2}(1 + c_2 v^2 + c_4 v^4 + \dots),
\end{equation*}
for certain constants $c_i$.  Then we have
\begin{equation*}
\phi'(v) \asymp \frac{x/v^2}{v}, \qquad \phi^{(j)}(v) \ll \frac{x/v^2}{v^j}, \quad j=2,3,\dots,
\end{equation*}
using $\frac{x}{v^3} \gg \frac{x}{v^2} \gg (qr)^{\varepsilon} \gg \kappa -1$, since $\kappa$ is fixed.
By Lemma \ref{lemma:exponentialintegral} yet again, the integral is very small.
If the inert function has support on an interval containing $\pi/2$, then the above argument breaks down.  So now suppose that the inert function has support on $v \geq \pi/4$, so in particular $V \asymp 1$ and $x \gg (qr)^{\varepsilon}$.  Change variables $v = \pi/2-u$, giving
\begin{equation}
K_{00}(x) = (-1)^{\frac{\kappa-2}{2}} \int_0^{\pi/4} \cos((\kappa -1)u) e\Big(x \frac{\sin^2 u}{\cos^2 u}\Big) w_T(\cos u, \cdot) du.
\end{equation}

Next we argue that the main part of this integral comes from $u \ll x^{-1/2} (qr)^{\varepsilon}$, provided we use a smooth truncation.  Let us apply a partition of unity, and consider the portion of the integral with $u \asymp U$ (with $U \ll 1$), which we denote $K_{00}^{U}(x)$.  By Lemma \ref{lemma:exponentialintegral} yet again, if $xU^2 \gg (qr)^{\varepsilon}$, then $K_{00}^{U}(x)$ is small.  We may also use that the integrand is even to extend to $-\pi/4$ to $\pi/4$, giving
\begin{equation*}
K_{00}(x) = \int_{-\pi/4}^{\pi/4} W(u) e(x \tan^2 u) du,
\end{equation*}
plus a small error term, 
where $W$ has support on $|u| \ll U$ with $U = x^{-1/2+\varepsilon}$.   

We may now derive an asymptotic expansion of $K_{00}$, with leading term $c_0 W(0) x^{-1/2}$, where $c_0$ is some absolute constant.  By developing this expansion carefully, we have that for $x \gg (qr)^{\varepsilon}$, $K_{00}^{(j)}(x) \ll x^{-1/2-j}$, and so by Mellin inversion, we have that
\begin{equation*}
K_{00}(x) = x^{-1/2} \intR f(t) x^{it} dt,
\end{equation*}
plus a small error term, where $|f(t)| \ll_A (1 +|t|)^{-A}$, with $A > 0$ arbitrarily large.  In our application, we may thus truncate at $|t| \ll (qr)^{\varepsilon}$.

The previous discussion gives a formula for $K$ of the form \eqref{eq:Kseparated2}, where $L$ takes the form
\begin{equation*}
 L(m_1, m_2, m_3, c) = \frac{1}{P^{1/2}} \int_{|t| \ll (qr)^{\varepsilon}}  w_T(t, m_1, m_2, m_3, c) \Big(\frac{|m_1 m_2 m_3| c^2}{[c,q]^3}\Big)^{it} dt,
\end{equation*}
where $w_T$ is 1-inert in the variables $m_1,m_2,m_3,c$ and has rapid decay in $t$, uniformly in all other parameters.  We may then write
\begin{equation*}
 w_T(t,m_1, m_2, m_3, c) = \int F({\bf u}; t) \Big(\frac{M_1}{|m_1|}\Big)^{u_1} \Big(\frac{M_2}{|m_2|}\Big)^{u_2} \Big(\frac{M_3}{|m_3|}\Big)^{u_3} \Big(\frac{C}{c}\Big)^{u_4} d{\bf u},
\end{equation*}
where the integral is over four vertical lines in the complex plane, one for each of the $u_i$, $i=1,2,3,4$.  By the rapid decay of $F$ beyond $(qr)^{\varepsilon}$, due to the fact that $w_T$ is inert, we may truncate the integrals at $|{\bf u}| \ll (qr)^{\varepsilon}$.
This expression gives \eqref{eq:LIntegralFormulaOscillatoryRange}, and so completes the proof of Lemma \ref{lemma:Kbound}.

\textbf{Proof of Lemma \ref{lemma:KboundNonOscillatory}}.
Now suppose \eqref{eq:NonOscCondition} holds.  As previously mentioned in the paragraphs following Lemma \ref{lemma:KboundOther}, $K$ is small unless \eqref{eq:smallkiMi} holds, a condition that we assume henceforth.  Assuming $x \ll X = 1+ \frac{\sqrt{ABN_1 N_2 N_3}}{C}$, we
have that $J_{\kappa-1}(x) = x^{\kappa - 1} W(x)$ where $W$ is a smooth function satisfying $x^{j} W^{(j)}(x) \ll X^j$.  That is, $W$ satisfies the same derivative bounds as an $X$-inert function, and so it may be absorbed into the definition of the inert function.  Therefore, by the separation of variables discussion from Section \ref{section:inert}, we have
\begin{multline}
\label{eq:KformulaseparatedVariables}
K = N_1 N_2 N_3 \Big(\frac{\sqrt{ABN_1 N_2 N_3}}{C}\Big)^{\kappa -1} \int_{|{\bf u}| \ll (qr)^{\varepsilon}} F({\bf u}) 
\Big(\frac{M_1}{|m_1|}\Big)^{u_1} \Big(\frac{M_2}{|m_2|}\Big)^{u_2} \\ \Big(\frac{M_3}{|m_3|}\Big)^{u_3} \Big(\frac{C}{c}\Big)^{u_4} d{\bf u}
+ O((qr)^{-A} \prod_{i=1}^{3} (1+ |m_i|)^{-2}),
\end{multline}
where $F({\bf u}) \ll_{J, \text{Re}({\bf u})} \prod_{\ell=1}^{4} (1+\frac{X}{|u_{\ell}|})^{-J}$.

We also want to factor out the exponential term $e_{AB[c,q]^3/c^2}(-m_1 m_2 m_3 )$.
It is not clear whether 
\begin{equation*}
P:= \frac{M_1 M_2 M_3}{AB C} 
\end{equation*}
is $\gg 1$ or $\ll 1$, so we treat both cases separately.  

If $P \ll (qr)^{\varepsilon}$, then essentially the exponential term $e_{AB[c,q]^3/c^2}( m_1 m_2 m_3)$ is not oscillatory, so by Mellin inversion there exists a simple formula of the form
\begin{multline}
\label{eq:exponentialMellinSeparationFormula}
 e\Big(\frac{m_1 m_2 m_3 c^2}{AB[c,q]^3}\Big)w(\pm \frac{m_1m_2m_3c^2/(AB[c,q]^3)}{P})  = \\ \int_{|t| \ll (qr)^{\varepsilon}} \Big(\frac{|m_1 m_2 m_3| c^2}{[c,q]^3}\Big)^{it} f(t) dt + O((qr)^{-1/\varepsilon}),
\end{multline}
where $f_{A,B, C, M_1,M_2,M_3, \pm}(t) = f(t) \ll 1$ and $w(t)$ is a smooth function of compact support on $(0,\infty)$, that is identically $1$ on the support of the inert function $w_T$ in the definition of $K$.

If $P \gg (qr)^{\varepsilon}$, then we claim that a formula like \eqref{eq:exponentialMellinSeparationFormula} holds, but with $|t| \asymp P$ and $f(t) \ll |t|^{-1/2}$.  For this, we argue as follows.  By Mellin inversion, we have
$e^{ix} w(x/P) = \intR f(t) x^{it} dt$, where $f(t) = \frac{1}{2 \pi} \int_0^{\infty} e^{ix} w(x/P) x^{it-1} dx$.  Since the support on $w$ causes $x \asymp P$, repeated integration by parts shows $f(t)$ is very small except when $|t| \asymp P$.  A standard stationary phase bound gives $f(t) \ll |t|^{-1/2}$.

In either case, we obtain \eqref{eq:KintegralformulaNonOscillatory}.

\textbf{Proof of Lemma \ref{lemma:KboundOther}.}
The claims that if \eqref{eq:OscCondition} holds, then $K$ is small, and that if \eqref{eq:NonOscCondition} holds, $K$ is small unless $|m_i| \ll \frac{C}{N_i} (qr)^{\varepsilon}$ for $i=1,2,3$ have already been shown in the previous analysis.  It remains to show the integral formula.

Suppose that $m_3=0$, $m_1, m_2 \neq 0$, and \eqref{eq:NonOscCondition} holds.  
The idea is to apply an analog of \eqref{eq:KformulaseparatedVariables}, but only in the $m_1, m_2$ variables.  This case is easier because the exponential factor simplifies as $1$, so there is no need to separate out $e_{AB[c,q]^3}(m_1 m_2 m_3 c^2)$.
By taking a Mellin transform in the $m_1, m_2$ variables, we get
\begin{multline*}
K(m_1, m_2, 0 ;c) = 
\Big(\frac{\sqrt{ABN_1 N_2 N_3}}{C}\Big)^{\kappa-1} N_1 N_2 N_3 \int_{|v_1| \ll (qr)^{\varepsilon}} \int_{|v_2| \ll (qr)^{\varepsilon}} \Big(\frac{C}{N_1 |m_1|}\Big)^{v_1} \Big(\frac{C}{N_2 |m_2| }\Big)^{v_2} \\ R(v_1, v_2, c) dv_1 dv_2
+ O(q^{-A} \prod_{i=1}^{3} (1+|m_i|)^{-2}),
\end{multline*}
where $R(v_1, v_2,c)$ is analytic in $\text{Re}(v_i) > 0$ for $i=1,2$, and satisfies the bound
\begin{equation*}
R(v_1, v_2, c) \ll_{J, \varepsilon, \text{Re}(v_1), \text{Re}(v_2)} \prod_{j=1}^{2} (1 + \frac{(qr)^{\varepsilon}}{|v_j|})^{-J}.
\end{equation*}
This bound on $R$ is precisely analogous to the bound on $F(\bf{u})$ in the proof of Lemma \ref{lemma:KboundNonOscillatory}.
This completes the proof.
 \end{proof}

\section{Recombination}
\label{section:recombination}
Now we prove Proposition \ref{prop:B''bound}.  Recall formulas \eqref{eq:Sdef}, \eqref{eq:S'def}, and \eqref{eq:S'Poisson}.
Let us write $\mathcal{S} = \mathcal{S}_0 + \mathcal{S}_{1}$ where $\mathcal{S}_0$ corresponds to the terms with some $m_i = 0$, while $\mathcal{S}_{1}$ corresponds to the terms with all $m_i \neq 0$.

\subsection{Bounding $\mathcal{S}_0$}
\begin{mylemma} 
We have $\mathcal{S}_0 \ll R^{-1} (qr)^\eps.$
\end{mylemma}
\begin{proof}

From Lemma \ref{lemma:KboundOther}, we see that $K$, and hence $\mathcal{S}_0$, is small unless \eqref{eq:NonOscCondition} holds. By \eqref{eq:NonOscCondition} and \eqref{eq:Csize} we have that $A,B$, and $C$ are each bounded by $(qr)^{10}$, say. This allows us to replace several factors of $(AB)^\eps$ and $C^\eps$ by $(qr)^\eps$ in the following.

Let us further decompose $\mathcal{S}_0 = \mathcal{S}_{00} + \mathcal{S}_{01}+\mathcal{S}_{02}$ where $\mathcal{S}_{02}$ corresponds to the terms with exactly two of the $m_i \neq 0$, $\mathcal{S}_{01}$ corresponds to the terms with precisely one of the $m_i\neq 0$, and $\mathcal{S}_{00}$ corresponds to the terms with $m_1 = m_2 = m_3 = 0$.

We first bound $\mathcal{S}_{02}$.
 Suppose for the sake of argument that $m_3=0$, and $m_1, m_2 \geq 1$, the other cases being similar, and let $\mathcal{S}_{02}^{+}$ denote these terms.  
For $i=1,2$, let $m_i = m_i'm_i''$ where $(m_i',2AB)=1$ and $m_i'' \mid (2AB)^\infty$. Then by Lemmas \ref{lemma:GboundOther} and \ref{lemma:KboundOther} we have \begin{multline}\label{eq:S02plusstart} |\mathcal{S}_{02}^+| \ll \sum_{N_1,N_2,N_3, C} \frac{\sqrt{N} }{C}  \left(\frac{ \sqrt{ABN}}{C}\right)^{\kappa -1 } \\  \sum_{\substack{c \equiv 0 \shortmod{\tilde{q}R}\\ c \asymp C}}\frac{1}{qc} 
 \Big|
 \sum_{m_1, m_2 \geq 1} g(m_1'',m_2'', 0, A,B,c,q)  
  R_{q_o}(m_1')R_{q_o}(m_2') \\ \int_{|v_1| \ll (qr)^\eps}  \int_{|v_2| \ll (qr)^\eps}   \left( \frac{C}{N_1m_1}\right)^{v_1} \left( \frac{C}{N_2m_2}\right)^{v_2} R(v_1,v_2,c) \,dv_1\,dv_2 \Big|,\end{multline}
plus a small error term.

 The goal is to form Dirichlet series over $m'_1,m'_2$. Since $q_o$ is square-free and coprime to $2AB$, we have that $$Q_{q_o}(s):=\sum_{(m,2AB)=1} \frac{R_{q_o}(m)}{m^s} = \mu(q_o) \zeta^{(2AB)}(s) \prod_{p \mid {q_o}}(1-p^{1-s}).$$ 
 Then factoring $m_i=m''_im'_i$, and using the bound on $g$ from Lemma \ref{lemma:GboundOther} we have
 \begin{multline}\label{eq:S01plusformula}  |\mathcal{S}_{02}^+| \ll \sum_{N_1,N_2,N_3, C} \frac{\sqrt{N} }{C}  \left(\frac{ \sqrt{ABN}}{C}\right)^{\kappa -1 }  \sum_{\substack{c \equiv 0 \shortmod{\tilde{q}R}\\ c \asymp C}}\frac{(A,c)(B,c)}{qc}  \sum_{m_1'',m_2'' \mid (2AB)^\infty}    \\
 \Bigg|  \iint_{\substack{|v_1| \ll (qr)^\eps \\ |v_2| \ll (qr)^\eps}}   \left( \frac{C}{N_1m''_1}\right)^{v_1} \left( \frac{C}{N_2m''_2}\right)^{v_2} 
 Q_{q_o}(v_1) Q_{q_o}(v_2)   R(v_1,v_2,c) \,dv_1\,dv_2 \Bigg|,\end{multline}
 plus a small error term.

We may assume without loss of generality that $q>8$, so that $q_o>1$. Note that under the assumption that $q_o>1$, the function $Q_{q_o}(s)$ is holomorphic in $\real(s)>0$. On the $\real(s)= \eps$ line, it satisfies $Q_{q_o}(s) \ll (1+|s|)^{1/2} \phi(q_o)(qr)^\eps.$
So we can move the vertical contour integrals to the lines $\real(v_1)=\real(v_2)= \eps$. The short horizontal segments created in doing so are extremely small by the bounds on $R(v_1,v_2,c)$ from Lemma \ref{lemma:KboundOther}. From the vertical lines of integration we obtain $S_{02}(N_1,N_2,N_3,c) \ll  (qr)^{\eps} \phi(q)^2(A,c)(B,c)$. Therefore we have 
\begin{equation}\label{eq:S01plusbound} |\mathcal{S}_{02}^+| \ll (qr)^{\eps}\sum_{N_1,N_2,N_3, C} \frac{\sqrt{N}}{C}  \left(\frac{ \sqrt{ABN}}{C}\right)^{\kappa -1 }  \sum_{\substack{c \equiv 0 \shortmod{\tilde{q}R}\\ c \asymp C}}\frac{\phi(q)^2}{qc}   (A,c)(B,c) ,\end{equation}
plus a small error term.

Using $(qR, AB) = 1$, Cauchy's inequality, and
\begin{equation}
 \sum_{n \leq x} (d,n )^2 \ll x \tau(d) d, 
\end{equation}
we derive
\begin{equation*}
 \sum_{\substack{c \equiv 0 \shortmod{\tilde{q}R} \\ c \asymp C}}
   (A,c)(B,c) \ll \frac{C}{qR} (AB)^{1/2} (qr)^{\varepsilon}.
\end{equation*}
Therefore by \eqref{eq:Csize} and \eqref{eq:NonOscCondition}
\begin{equation}
|\mathcal{S}_{02}^{+}| \ll 
\sum_{N_1, N_2, N_3, C} (qr)^{\varepsilon}   \frac{NAB}{ C^2 R}   \ll \frac{(qr)^{\varepsilon}}{R},
\end{equation}
which is sufficient for the bound in the statement of the lemma.
By a symmetry argument, this shows the desired bound on $\mathcal{S}_{02}$.

Similarly to the method used to bound $\mathcal{S}_{02}$, we have in analogy with \eqref{eq:S01plusformula}, the bound
\begin{multline*}
 |\mathcal{S}_{01}^{+}| \ll 
\sum_{N_1, N_2, N_3, C}  \frac{(qr)^{\varepsilon}}{\sqrt{N} C}  \Big(\frac{\sqrt{ABN}}{C}\Big)^{\kappa-1} N 
\sum_{\substack{c \equiv 0 \shortmod{\tilde{q}R} \\ c \asymp C}} \frac{(A,c) (B,c)\phi(q)}{cq} 
\\
 \sum_{m_1'' \mid (2AB)^\infty} \Big|
  \int_{|v| \ll (qr)^{\varepsilon}}  \Big(\frac{C}{N_1m''_1}\Big)^{v}
   Q_{q_o}(v)  R(v, c)  dv \Big|, 
\end{multline*}
where $Q_{q_o}(v) R(v,c) $ is analytic in $\text{Re}(v) > 0$, and satisfies the bound $$Q_{q_o}(v) R(v,c)  \ll_{\text{Re}(v)} (1+|v|)^{1/2} (qr)^{\varepsilon}\phi(q_o).$$  We move the contour to the line $\text{Re}(v) = \varepsilon$.  The short horizontal segment created in doing so is extremely small by the bounds on $R(v,c)$ from Lemma \ref{lemma:KboundOther}. From the vertical segment we get
\begin{equation}
 |\mathcal{S}_{01}^{+}| \ll 
\sum_{N_1, N_2, N_3, C}  \frac{(qr)^{\varepsilon}}{\sqrt{N} C}  \Big(\frac{\sqrt{ABN}}{C}\Big)^{\kappa-1} N
\sum_{\substack{c \equiv 0 \shortmod{\tilde{q}R} \\ c \asymp C}} \frac{(A,c) (B,c)}{c q } \phi(q)^2
\end{equation}
plus a small error term.
This is precisely the same bound as \eqref{eq:S01plusbound}, and so our final bound on $\mathcal{S}_{01}$ is identical to the bound on $\mathcal{S}_{02}$.

Finally, we bound $\mathcal{S}_{00}$, which is the easiest case.  Using only the upper bound \eqref{eq:KboundOther} and the upper bound from Lemma \ref{lemma:GboundOther}, we obtain a bound on $\mathcal{S}_{00}$ of the exact same shape as \eqref{eq:S01plusbound}, so the proof is complete.
 \end{proof}

\subsection{Bounding $\mathcal{S}_1$}
Here we show the desired bound 
\begin{equation}
\label{eq:S1bound}
\mathcal{S}_1 \ll (qr)^{\varepsilon} \Big(\frac{(AB)^{1/2} }{R^{1/2}} +  \frac{r^{3/4}}{q^{1/2} R}\Big). 
\end{equation}
Let us write $\mathcal{S}_1 = \sum_{\epsilon_1, \epsilon_2, \epsilon_3 \in \{\pm \} } \mathcal{S}_1^{\epsilon_1, \epsilon_2, \epsilon_3}$ where this sum is restricted to $\epsilon_i m_i > 0$ for $i=1,2,3$.  The same method will apply to each of these terms, so for simplicity we estimate $\mathcal{S}_1^{+, +, + }$, which we denote with shorthand by $\mathcal{S}_1^{+}$.  
 
We have
 \begin{multline*}
  \mathcal{S}_1^{+} = \sum_{\substack{M_1, M_2, M_3 \\ N_1, N_2, N_3, C}} \frac{1}{ N^{1/2} C} 
  \sum_{m_i \asymp M_i} \sum_{c \equiv 0 \shortmod{\tilde{q}R}} w_C(c) \\ \Big(e\Big(\frac{-m_1 m_2 m_3}{AB[c,q]^3/c^2}\Big)  G_{A,B}(m_1, m_2, m_3;c) \Big)
 \Big(e\Big(\frac{ m_1 m_2 m_3}{AB[c,q]^3/c^2}\Big) K(m_1, m_2, m_3;c) \Big). 
 \end{multline*}
There are two main cases to consider, depending on if \eqref{eq:OscCondition} holds, or if \eqref{eq:NonOscCondition} holds, and we correspondingly write $\mathcal{S}_1^{+} = \mathcal{T} + \mathcal{U}$.

\textbf{Case} $\mathcal{T}$. 
By Lemma \ref{lemma:Kbound}, we have (with shorthand $M = M_1 M_2 M_3$, and $M_i$ satisfying \eqref{eq:MsizeOscillatoryRange})
\begin{multline*}
\mathcal{T} \ll \sum_{\substack{N_1, N_2, N_3, C \\ M_1, M_2, M_3}} \Big| \frac{C^{3/2} N^{1/2}}{C^2 q \sqrt{M N}} \sum_{\substack{c \equiv 0 \shortmod{\tilde{q}R} \\c \asymp C}} \\ \sum_{m_1, m_2, m_3}  \Big(e\Big(\frac{-m_1 m_2 m_3}{AB[c,q]^3/c^2}\Big) cq G_{A,B}(m_1, m_2, m_3;c) \Big)
\frac{1}{P^{1/2}} \int_{|{\bf u}| \ll (qr)^{\varepsilon}} \int_{|y| \ll (qr)^{\varepsilon}} \\ F({\bf u};y) \Big(\frac{m_1 m_2 m_3 c^2}{[c,q]^3}\Big)^{iy} 
\Big(\frac{M_1}{m_1}\Big)^{u_1} \Big(\frac{M_2}{m_2}\Big)^{u_2} \Big(\frac{M_3}{m_3}\Big)^{u_3} \Big(\frac{C}{c}\Big)^{u_4} d{\bf u} dy \Big|,
\end{multline*} 
plus a small error term.
 Assume that $\text{Re}(u_i)  = 1+ \eps $ for all $i=1,2,3,4$.  Then by \eqref{eq:ZRqdefinition}, we have 
 \begin{multline*}
\mathcal{T} \ll  \sum_{N_1, N_2, N_3, C} \Big| \frac{C^{-1/2}}{q(PM)^{1/2}} 
 \int_{|{\bf u}| \ll (qr)^{\varepsilon}} 
\int_{|y| \ll (qr)^{\varepsilon}} F({\bf u};y) \sum_{\substack{(\epsilon_1,\epsilon_2,\epsilon_3)\in \{\pm 1\}^3 \\ \delta \in \{1,2,4,8\}}} \left(\frac{\delta^3}{\tilde{q}Rq_e^3}\right)^{iy} 
\\
Z_{\delta,R,q}^{(\epsilon_1,\epsilon_2,\epsilon_3)}(u_1-iy, u_2-iy, u_3-iy, u_4+iy)
M_1^{u_1} M_2^{u_2} M_3^{u_3}  \Big(\frac{C}{\tilde{q}R}\Big)^{u_4} d{\bf u} dy \Big|,
\end{multline*} 
plus a small error term.
Now we decompose further by $Z = Z_0 + Z'$, as in Proposition \ref{prop:Zproperties}, and write $\mathcal{T} = \mathcal{T}_0 + \mathcal{T}'$.  For $\mathcal{T}_0$ we have
\begin{equation*}
\mathcal{T}_0 \ll \frac{C^{1/2}}{(PM)^{1/2}} \frac{M}{q^2 R AB} (qr)^{\varepsilon}.
\end{equation*}
 Recall $P = \frac{M}{ABC}$,  $M \asymp (AB)^{3/2} N^{1/2}$, $C \ll \sqrt{NAB} (qr)^{\varepsilon}$, and $N \ll (q^2 r)^{3/2 + \varepsilon}$ giving
\begin{equation}
\label{eq:T0bound}
\mathcal{T}_0 \ll \frac{ C^{}}{q^2 R (AB)^{1/2}}  (qr)^{\varepsilon} \ll \frac{N^{1/2}}{q^2 R} (qr)^{\varepsilon} \ll \frac{1}{\sqrt{q}} \frac{r^{3/4}}{R} (qr)^{\varepsilon}.
\end{equation}

For $\mathcal{T}'$, we move the contours to $\text{Re}(u_i) = 1/2 + \varepsilon$ for $i=1,2,3,4$.  The short horizontal segments created in doing so are extremely small by the bounds on $F({\bf u};y)$ from Lemma \ref{lemma:Kbound}. Then we obtain
\begin{equation*}
\mathcal{T}' \ll (qr)^{\varepsilon} \sum_{N_1, N_2, N_3, C}  \frac{C^{-1/2}}{q(PM)^{1/2}} 
q^{3/2}\sqrt{AB} M^{1/2} \Big(\frac{C}{qR}\Big)^{1/2}  
\ll (qr)^{\varepsilon} 
\sum_{N_1, N_2, N_3, C} \frac{\sqrt{AB}}{P^{1/2} R^{1/2}}
.
\end{equation*} 
Since $P \gg 1$, we have
\begin{equation*}
\mathcal{T}' \ll (qr)^{\varepsilon} \frac{(AB)^{1/2}}{R^{1/2}},
\end{equation*} 
 as desired.
 
\textbf{Case $\mathcal{U}$}. 
Here we obtain from Lemma \ref{lemma:KboundNonOscillatory}  the bound
\begin{multline}
\label{eq:UIntegralBound}
\mathcal{U} \ll \sum_{\substack{N_1, N_2, N_3, C \\ M_1, M_2, M_3}} \Big| \frac{N \Big(\frac{\sqrt{ABN}}{C}\Big)^{\kappa-1}}{C^2 q \sqrt{ N}} 
 \int_{|{\bf u}| \ll (qr)^{\varepsilon}} \int_{|y| \ll (qr)^{\varepsilon} + P} F({\bf u}) f_P(y) \\ \sum_{\substack{(\epsilon_1,\epsilon_2,\epsilon_3)\in \{\pm 1\}^3 \\ \delta \in \{1,2,4,8\}}} \left(\frac{\delta^3}{\tilde{q}Rq_e^3}\right)^{iy} 
Z_{\delta,R,q}^{(\epsilon_1,\epsilon_2,\epsilon_3)}(u_1-iy, u_2-iy, u_3-iy, u_4+iy) \\
M_1^{u_1} M_2^{u_2} M_3^{u_3}  \Big(\frac{C}{\tilde{q}R}\Big)^{u_4} d{\bf u} dy
 \Big|,
\end{multline}  
plus a small error term.
As in the case of $\mathcal{T}$, write $\mathcal{U} = \mathcal{U}_0 + \mathcal{U}'$.  Consider first the case $\mathcal{U}'$. As in case $\mathcal{T}$, we move the contours to $\real(u_i)=1/2+\eps$. The short horizontal segments created in doing so are extremely small by the bounds on $F({\bf u};y)$ from Lemma \ref{lemma:KboundNonOscillatory}.  Using \eqref{eq:Z'integralbound}, we get
\begin{multline*}
\mathcal{U}' \ll (qr)^{\varepsilon} \frac{N \sqrt{AB}}{C^3 q} \sqrt{AB}q^{3/2} M^{1/2} \Big(\frac{C}{qR}\Big)^{1/2} (1+P^{1/2}) \\
\ll 
(qr)^{\varepsilon} \frac{NAB M^{1/2}}{R^{1/2} C^{5/2}} \Big(1 + \Big(\frac{M}{ABC}\Big)^{1/2}\Big).
\end{multline*} 
Now $M \ll \frac{C^3}{N} (qr)^{\varepsilon}$, and \eqref{eq:NonOscCondition} holds, so after simplification this leads to \begin{equation*} \mathcal{U}' \ll (AB)^{1/2} R^{-1/2} (qr)^{\varepsilon},\end{equation*} as desired.

Finally, we turn to the case of $\mathcal{U}_0$.  
To start, we suppose that $\text{Re}(u_i) = 1+\varepsilon$ for $i=1,2,3,4$.

Consider the case where $P \gg (qr)^{\varepsilon}$ with sufficiently large implied constant.  Then we shift contours to $\text{Re}(u_i) = 1/2 + \varepsilon$, for all $i$.  No poles are encountered in doing so, since $|u_i| < |y|$ throughout the integral \eqref{eq:UIntegralBound} if $f_P(y)$ has support $|y| \asymp P$. The short horizontal segments created in preforming this contour shift are extremely small by the bounds on $F({\bf u})$ from Lemma \ref{lemma:KboundNonOscillatory}.  The contribution to $\mathcal{U}_0$ of the integral along the  vertical segments is certainly bounded by the same bound we obtained on $\mathcal{U}'$, since the bound on $Z_0$ appearing in Proposition \ref{prop:Zproperties}  is much stronger than the bound on $Z'$.  Therefore, $\mathcal{U}_0$ is bounded in a satisfactory way for $P \gg (qr)^{\varepsilon}$.  

Now suppose $P \ll (qr)^{\varepsilon}$.  
 Then by \eqref{eq:Z00boundRightOfCriticalStrip}, we have
\begin{equation}
\label{eq:UquadrupleResidueBound}
\mathcal{U}_{0} \ll (qr)^{\varepsilon}  \sum_{\substack{N_1, N_2, N_3, C \\ M_1, M_2, M_3}}  \frac{N \Big(\frac{\sqrt{ABN}}{C}\Big)^{\kappa-1}}{C^2 q \sqrt{ N}} 
  \frac{MC}{q R AB} 
\ll 
(qr)^{\varepsilon}  \sum_{\substack{N_1, N_2, N_3, C \\ M_1, M_2, M_3}}  N^{1/2} \Big(\frac{\sqrt{ABN}}{C}\Big)^{} 
  \frac{M}{AB C q^2 R}.
\end{equation}  
Since $P = \frac{M}{ABC} \ll (qr)^{\varepsilon}$ now, and $\frac{\sqrt{ABN}}{C} \ll (qr)^{\varepsilon}$ too, we obtain
\begin{equation*}
\mathcal{U}_{0}  \ll (qr)^{\varepsilon} N^{1/2} q^{-2} R^{-1} \ll (qr)^{\varepsilon} (q^3 r^{3/2})^{1/2} q^{-2} R^{-1} \ll (qr)^{\varepsilon} q^{-1/2} \frac{r^{3/4}}{R}.
\end{equation*}
This is the same bound as \eqref{eq:T0bound}, which completes the proof of \eqref{eq:S1bound}.

\end{document}